\pdfoutput=1

\documentclass[12pt,
headings=small]{scrartcl}

\usepackage[a4paper, left=2.8cm, right=2.8cm]{geometry}    %
\setkomafont{sectioning}{\rmfamily\bfseries} %

\usepackage{pifont}
\usepackage{times}
\usepackage{amsfonts}       %
\usepackage{amsmath} %
\usepackage{amssymb}        %
\usepackage{amsthm}         %
\usepackage{booktabs}
\usepackage{titlefoot}
\usepackage{microtype}
\usepackage{hyperref}
\usepackage{latexsym}
 \usepackage[matrix,curve, all]{xy}
\usepackage[OT2,T1]{fontenc}
  
\DeclareSymbolFont{cyrletters}{OT2}{wncyr}{m}{n}
\DeclareMathSymbol{\Sha}{\mathalpha}{cyrletters}{"58}

\let\svthefootnote\thefootnote
\newcommand\blankfootnote[1]{%
  \let\thefootnote\relax\footnotetext{#1}%
  \let\thefootnote\svthefootnote%
}

\newcommand{\A}{\mathbb{A}}

\newcommand{\Z}{\mathbb{Z}}
\newcommand{\Q}{\mathbb{Q}}
\newcommand{\R}{\mathbb{R}}
\newcommand{\C}{\mathbb{C}}

\newcommand{\HH}{\mathbb{H}} %

\newcommand{\PGSp}{\mathrm{PGSp}}
\newcommand{\GSp}{\mathrm{GSp}}
\newcommand{\Sp}{\mathrm{Sp}}
\newcommand{\GSO}{\mathrm{GSO}}
\newcommand{\Gl}{\mathrm{Gl}}
\newcommand{\Sl}{\mathrm{Sl}}
\newcommand{\Hom}{\mathrm{Hom}}

\newcommand{\virt}{\mathrm{virt}}
\newcommand{\cusp}{\mathrm{cusp}}

\newcommand{\GU}{\mathrm{GU}}
\newcommand{\SU}{\mathrm{SU}}

\newcommand{\SO}{\mathrm{SO}}

\newcommand{\gfrak}{\mathfrak{g}}

\newcommand{\im}{\mathrm{im}}

\DeclareMathOperator{\Ind}{\mathrm{Ind}}
\DeclareMathOperator{\Frob}{\mathrm{Frob}}

\newcommand{\fin}{\mathrm{fin}}
\newcommand{\new}{\mathrm{new}}
\newcommand{\der}{\mathrm{der}}

\newcommand{\elliptic}{\mathrm{ell}}
\DeclareMathOperator{\tr}{\mathrm{tr}}

\newcommand\nosf[1]{\begin{footnotesize}\textup{\textsf{#1}}\end{footnotesize}}
\DeclareMathOperator{\diag}{diag}

\newcommand{\coh}{\mathrm{coh}}
\DeclareMathOperator{\Res}{Res}

\newcommand{\hol}{\mathrm{hol}}
\newtheorem{thm}{Theorem}[section]
\newtheorem{cor}[thm]{Corollary}
\newtheorem{lemma}[thm]{Lemma}
\newtheorem{conjecture}[thm]{Conjecture}

\newtheorem{prop}[thm]{Proposition}
\newtheorem{rmk}[thm]{Remark}

\newtheorem{defn}[thm]{Definition}

\usepackage{enumitem}
\setlist[enumerate]{itemsep=-0.5ex plus0.1ex minus 0.2ex}
\setlist[description]{itemsep=-0.5ex plus0.1ex minus 0.2ex}
\setlist[itemize]{itemsep=-0.5ex plus0.1ex minus 0.2ex}

\everydisplay{
\thickmuskip=5mu plus 2mu
\medmuskip=4mu plus 1mu
\relax}

\newcommand{\GlobalField}{k}
\newcommand{\id}{\mathrm{id}}

\begin{document}
 
\title{
\begin{large} Global liftings between inner forms of $\mathrm{GSp}(4)$\end{large}
}
\date{}
 \author{M.~R\"osner\and R.~Weissauer}
\maketitle

\begin{abstract}
For reductive groups $G$ over a number field we discuss automorphic liftings of cohomological cuspidal irreducible automorphic representations $\pi$ of $G(\A)$ to  irreducible cohomological automorphic representations of $H(\A)$ for the quasi-split inner form $H$ of $G$, and other inner forms as well. 
We show the existence of  nontrivial weak global cohomological liftings in many cases,
in particular for the case where $G$ is anisotropic at the archimedean places.  
A priori, for these weak liftings we do not give a description of the precise nature 
of the corresponding local liftings at the ramified places, nor do we characterize the image of the lift. For inner forms of the group $H=\mathrm{GSp}(4)$ however we address these finer questions. Especially, we prove the recent conjectures of Ibukiyama and Kitayama on paramodular newforms of squarefree level.
\blankfootnote{2010 \textit{Mathematics Subject Classification.} Primary 22E55; Secondary 11F46, 11F70, 20G05.}
\end{abstract}

\setcounter{tocdepth}{1}

\textit{Introduction}.
Algebraic modular forms are automorphic forms defined for reductive groups $G$ over a number field $\GlobalField$ that are anisotropic at the real places; compare \cite{Gross_alg_mod_forms}.
In thm.~\ref{innerlifting} we show that irreducible cuspidal automorphic representations $\pi$  of $G(\A)$, defined by algebraic modular forms, always admit weak liftings to  cohomological automorphic representations of $H(\A)$, where $H$ is the quasi-split inner form  of $G$ or another inner form of $G$.
For this assertion we may assume that $\pi$ is of general type, i.e.\ not weakly isomorphic to an Eisenstein representation or endoscopic representation of $H(\A)$.
In the simplest case $H=\Gl(2)$ this gives the Jacquet-Langlands lift between automorphic forms on the multiplicative group $G=D^\times$ of quaternion algebras $D$, definite at the archimedean places, to automorphic forms on $H=\Gl(2)$. Similar to the work of Jacquet and Langlands, our method is based on the trace formula. The method used also  
provides weak liftings from $G$ as above to other inner forms $H$ of $G$ in a completely analogous way, under certain conditions on norm maps. In particular, since groups $H$ defining hermitian symmetric spaces 
always admit inner forms that are anisotropic at the archimedean places,  this gives examples of special interest, since liftings of representations $\pi$ of $G(\A)$ of general type as above, then give rise to holomorphic automorphic forms, i.e. to irreducible cuspidal automorphic representations of $H(\A)$ whose archimedean components belong to the holomorphic discrete series. This allows to attach Galois representations to $\pi$ by considering $\ell$-adic cohomology groups.

\medskip\noindent
We can extend the results described above to inner liftings of irreducible cohomological cuspidal automorphic representations from an arbitrary inner form $G$ to the quasi-split inner form~$H$ under a certain conjectural hypothesis  (conjecture~\ref{conj:middle_coh}).
We show these hypothesis in certain cases of interest, e.g. for inner forms of $\GSp(4)$ and automorphic representations whose archimedean discrete series representations are of regular weight $\lambda$ at the archimedean places.

\medskip\noindent
After the general discussion of weak inner liftings, in the second half of the paper we study the particular case $H=\GSp(4)$, starting from section~\ref{s:inner_forms_coh} onward, where we give a precise and explicit description of the inner lift at all the local places in terms of a local-global principle; see thm.~\ref{mainthm} and thm.~\ref{mainthm-2}. This in particular uses local results of
Chan and Gan \cite{Chan_Gan_GSp4_III}.

\medskip\noindent
Our investigations had three initial motivations. One of them arose from the desire to have explicit existence results for  paramodular Siegel modular forms of genus two and weight three that came from computations of Rama and Tornar\'{\i}a~\cite{Rama_Tornaria}.
They use quinary quadratic forms to construct algebraic modular forms for unitary groups
$\mathrm{U}(2)$ over quaternion algebras.
We show that these forms lift to weight three Siegel modular forms by a generalized Jacquet-Langlands correspondence.
For the special case of squarefree levels this was independently shown by van Hoften~\cite{van_Hoften} using methods from algebraic geometry via the weight filtration spectral sequence. 

\medskip\noindent
Another motivation for us came from various conjectures of Ibukiyama \cite{Ibukiyama06, Ibukiyama08, Ibukiyama17, Ibukiyama18}, formulated over the last decades. These predict inner liftings between automorphic forms related to the group $\GSp(4)$.
One of these conjectures~\cite[1.3]{Ibukiyama17} will be proved in this paper; see section~\ref{s:innerlift}.
We became also interested in the conjectures of Ibukiyama~\cite{Ibukiyama08},
suggesting that liftings may yield an answer to Harder's conjectures relating congruences between Siegel modular forms and congruences for special $L$-values of elliptic cusp forms \cite{Harder}.
The lifting proposed by Ibukiyama~\cite{Ibukiyama08} relates Siegel modular forms of genus two of integral weight to Siegel modular forms of genus two of half integral weight, and this should be thought as a generalization of the Shimura lift for elliptic modular forms.
In particular cases (not for the full modular group) this generalized Shimura lift, whose existence was also conjectured by Ibukiyama, should be obtained by the composition of two other liftings, namely from 
the inner lift, relating automorphic forms on $H=\GSp(4)$ to algebraic modular forms 
on inner forms $G$ of $\GSp(4)$, then followed by a theta lift, relating algebraic automorphic 
for $H$ with automorphic forms on $\Sp(4)$ of half integral weight,
by using definite quinary theta series. The latter  uses the exceptional isomorphism 
between the root systems $B_2$ and $C_2$. 

\medskip\noindent
Last but not least, our final incentive to study inner liftings for $\GSp(4)$ arose from computations of the local $L$-factors for global $L$-series defined for automorphic forms on $\GU_D(1,1)$.
These $L$-series, defined similarly as the Piateski-Shapiro $L$-series attached to automorphic forms 
for the group $\GSp(4)$, were studied by us \cite{Roesner_Weissauer_L_factor_inner_form} recently.
We computed the local $L$-factors $L(\pi,s)$ of these $L$-series except for local supercuspidal irreducible representations $\pi$ of the non-split inner form of $\GSp(4)$ over a local nonarchimedean field $\GlobalField_v$. 
Previously Piatetski-Shapiro had shown  that generic supercuspidal
irreducible representations $\pi_v$ of $\GSp(4,\GlobalField_v)$ over local nonarchimedean fields $\GlobalField_v$ have trivial local $L$-factors $L(\pi_v,s)$, whereas
the local factors $L(\pi,s)$ for non-generic supercuspidal representations of $\GSp(4,\GlobalField_v)$ are known by Dani\c{s}man~\cite{Danisman_Annals}.
In contrast, generic representations are not defined for groups that are not quasi-split over $\GlobalField_v$, so Piatetski-Shapiro's argument does not apply for inner forms.
However, using the global inner lifting from this paper and a comparison of $L$-factors at almost all places, it is possible to compute the supercuspidal local $L$-factors
by reducing the question to the knowledge of $L$-factors for the split case.
For this particular application of the global inner lift 
we refer to the forthcoming paper \cite{Roesner_Weissauer_L_factor_inner_form} where we study Piatetski-Shapiro's  $L$-functions \cite{PS-L-Factor_GSp4} in the context of non-quasi-split forms of $\GSp(4)$.

\medskip\noindent
Many of the assertions, in particular in the first sections of the present paper, could be stated and proven in somewhat greater generality.
Although this mainly requires technical improvements, due to our primary interest we did not want to go into this and preferred to postpone this to a more complete treatment at another occasion.

\tableofcontents

\section{Assumptions on \texorpdfstring{$G$}{G}}\label{section1}
Let $G$ be a connected reductive group over a totally real number field $F$.
For simplicity we make the hypothesis that $G$ satisfies the following two properties:
\begin{enumerate}
\item For the center $Z$ of $G$ the Galois cohomology $H^1(\GlobalField,Z)$ is trivial.
\item The derived group $G_{\der}$ of $G$ is simply connected.
\end{enumerate}
One may 
use $z$-extensions, in the sense of Langlands, to remove these assumptions. 
On the other hand, for specialists it is clear  that for the main results of sections 1-7 both assumptions are superfluous. They were made for non-expert readers in order to simplify   the presentation in the technical  earlier  sections, and also to be able to refer directly to statements \cite{Weissauer_Endo_GSp4}, where the second assumption was also made (for simplicity). This being said, let us remark that later we will also assume that the quasi-split inner form $H$ of $G$ over $\GlobalField$ has the property that $H(\R)$ admits discrete series representations, i.e. that $H(\R)$ has maximal tori that are $\R$-anisotropic. In contrast to the above
two assumptions, this later assumption will indeed turn out to be the key assumption for 
our arguments. Trying to avoid it, one has to consider twisted cohomological trace formulas
\cite{Weselmann_trace_formula} instead of the cohomological trace formula, something that lies outside of the intention of this article.

\medskip 
Since restriction of scalars preserves both properties, we will assume $\GlobalField=\Q$.
Indeed, both properties are inherited by inner forms of $G$.

\bigskip\noindent
\textit{Weak Lifts}.
Two irreducible representations $\pi=\otimes'_v \pi_v$
and $\pi' = \otimes'_v \pi'_v$ of $G(\mathbb A)$ are called weakly equivalent
if $\pi_v \cong \pi'_v$ holds for almost all places $v$.
An irreducible automorphic representation $\pi$ of $G(\mathbb A)$ is called
\textit{weakly $G$-Eisenstein} if it is weakly equivalent to an irreducible constitutent $\pi'$ of a parabolically induced representation $\Ind_{P(\mathbb A)}^{G(\mathbb A)}(\sigma)$ of an automorphic representation $\sigma$ of  $L(\mathbb A)$ for the Levi subgroup $L$ of a proper $\GlobalField$-rational parabolic subgroup $P=LN_P$ of $G$.
For an inner form $H$ of $G$ the local groups $H_v=H(\GlobalField_v)$ and $G_v=G(\GlobalField_v)$ are locally isomorphic at $v\notin S$ outside of some finite set $S$ of places $v$ of $K$.
An irreducible automorphic representation $\pi$ of $G(\mathbb A)$ is called a \textit{weak inner lift} of an irreducible automorphic representation $\widetilde\pi$ of $H(\mathbb A)$, and vice versa, if $\pi_v \cong \widetilde{\pi}_v$ and $G_v\cong H_v$ holds for almost every $v$.
It can happen that $\tilde\pi$ is weakly $H$-Eisenstein although $\pi$ is not weakly $G$-Eisenstein. 
If $\pi$ is cuspidal and is a weak inner lift of an $H$-Eisenstein representation of the quasi-split inner form $H$, then we call $\pi$ a \textit{CAP representation} of $G(\A)$. 
We say $\pi$ is of general type if $\pi$ is neither CAP nor endoscopic.

\section{Lefschetz numbers}\label{s:Lefschetz}

Let $K_\infty$ be a maximal compact subgroup of the Lie group $G_\infty=G(\R)$, let
$A_G$ be the maximal $\Q$-split subtorus of the center $Z$ of $G$ and $A_G(\R)^0$ be the topologically connected component of $A_G(\R)$.
Put $\tilde K_\infty= K_\infty \cdot A_G(\R)^0$ and $X_G= G_\infty/\tilde K_\infty$.
Fix compact open subgroups $\Omega_v$ of $G_v=G(\GlobalField_v)$ at the non-archimedean places $v$
that stabilize a special point in the Bruhat-Tits building adapted to a fixed choice of a minimal $\mathbb Q$-parabolic subgroup $P_0$ of $G$ as in \cite{Weissauer_Endo_GSp4}, p.29.
For compact open sufficiently small subgroups $K$ of $\prod_{v<\infty}\Omega_v$ of $G(\mathbb A_{\fin})$
put $$S_K(G)= G(\mathbb Q)\!\setminus (X_G \times G(\mathbb A_{\fin}))/K \ .$$
Finite dimensional irreducible complex representations $\tau=\tau_\lambda$ of $G$ of highest weight $\lambda$
define coefficient systems $V_\lambda$ on $S_K(G)$; here we follow the conventions
of \cite{Chai_Faltings}, \cite{Weissauer_Endo_GSp4}.
We denote the Lie algebra of $G_\infty$ by $\mathfrak{g}=\mathrm{Lie}(G_\infty)$.
An irreducible representation $\pi_\infty$ of $G_\infty$ is cohomological with respect to
$V_\lambda$ if the graded space $H^\bullet(\gfrak ,\tilde K_\infty,
\pi_{\infty}\otimes \tau_\lambda)$ of relative Lie algebra cohomology  does not vanish.   
Its Euler characteristic will be denoted 
$$ \chi(\pi_\infty,\lambda) = \sum_i (-1)^i \dim(H^i(\gfrak ,\tilde K_\infty,
\pi_{\infty}\otimes \tau_\lambda))\ .$$
Up to isomorphism there are only finitely many irreducible representations $\pi_\infty$ 
that are cohomological for given $V_\lambda$, \cite{Vogan_Zuckerman}, \cite{Borel_Wallach}.
Let $\pi$ be an irreducible cuspidal automorphic representation of $G(\mathbb A)$ whose archimedean factor
$\pi_\infty$ is cohomological with respect to  the coefficient system $V_\lambda$.
For a finite subset $S$ of the nonarchimedean places
we write $$\pi  = \pi_\infty\otimes \pi_S\otimes \pi^S\ ,$$ where $\pi^S$
denotes the restricted tensor product over all $\pi_v$ for the nonarchimedean places $v\notin S$.
We write $\A^S$ and $\A_S$ for the restricted tensor product of $\GlobalField_v$ for the nonarchimedean places $v\notin S$ resp.~$v\in S$.
For a cuspidal admissible representation $W$ of $G(\A^S)$ let $W(\pi^S)$ denote the $\pi^S$-isotypical subspace of $W$.
For the coefficient system $V_\lambda$ on $S_G(G)$ attached to $\lambda$ the cohomology
$$  H^\bullet(S_K(G),V_\lambda) \ ,$$
is considered as a $\mathbb Z$-graded object in the usual way.
Its generalized $\pi^S$ isotypic subspace $H^\bullet(S_K(G),V_\lambda)(\pi^S)$
defines the direct limit $$\varinjlim_K H^\bullet(S_K(G),V_\lambda)(\pi^S)$$ as a $G(\mathbb A_S)$ module.
Traces of $K$-biinvariant functions $f_S$ coincide with the trace on the finite
dimensional subspaces $H^\bullet(S_K(G),V_\lambda)(\pi^S)$.
By the work of Franke \cite{Franke} and Franke-Schwermer \cite{Franke_Schwermer}
all irreducible constituents of this direct limit are isomorphic to
the finite part of some irreducible automorphic representations $\pi \cong \otimes'_v \pi_v$ of $G(\A)$
whose infinite component $\pi_\infty$ is cohomological for $V_\lambda$.  

\bigskip\noindent
\textit{$\pi^S$-typic components}. If we fix an automorphic representation $\pi=\pi_\infty\pi_S\pi^S$ and consider the generalized eigenspace on the cohomology
with respect to $\pi^S$, it is convenient to work with test functions in $C_c^\infty(G_v,\omega_v)$, where $\omega_v$ is the inverse of the central character of $\pi_v$ of $Z(\GlobalField_v)$.
Here $C_c^\infty(G_v,\omega_v)$ denotes the subspace of functions 
$f_v$ in $C^\infty(G_v)$ that have compact support modulo $Z(F_v)$ and transform
as $f_v(g_vz_v)= \omega_v(z_v)f_v(g_v)$.
Although in the cohomological trace formula we use test functions in $C_c^\infty(G(\A_{\fin}))$,
we may pass to their projections $\int_{Z(\A_{\fin})} f(gz)\omega(z)^{-1}dz$ in $C_c^\infty(G(\A_{\fin}),\omega)$ by integration.
The supertrace, i.e.\ the alternating sum of the traces of $f_S\in C^\infty_c(G(\A_S),\omega_S)$ on each $H^i(S_K(G),V_\lambda)(\pi^S)$
with sign $(-1)^i$, defines the Lefschetz number
$$ L(f_S,V_\lambda,\pi^S) = tr_s(f_S, H^\bullet(S_K(G),V_\lambda)(\pi^S)) \ ,$$
and similarly $ L(f,V_\lambda) = tr_s(f, H^\bullet(S_K(G),V_\lambda))$ for
$f\in C_c^\infty(G(\mathbb A_{\fin}))$.
Often it is convenient to enlarge $S$ to some finite set $T$ of nonarchimedean places
that in particular contains all ramified places of $\pi$.
Fix a compact open subgroup $K=K_T K^T =\prod_{v\neq \infty} K_v \subseteq \Omega$
such that $K^T=\prod_{v\notin T} \Omega_v$.
The group $K_T$ may be assumed to be a good small subgroup of $G_T=\prod_{v\in T}G_v$
in the sense of \cite{Weissauer_Endo_GSp4}, p.~79 and p.~26.
Such groups define cofinal systems of open compact subgroups of $G_T$.
For a given finite set ${\cal P}$ of isomorphism classes of admissible irreducible representations of $G_T$,
all assumed to have nontrivial $K_T$ fixed vectors,
there exists a $K_T$-biinvariant function $f_T=f(\pi_T)$ in $C_c^\infty(G_T)$ with $tr(\pi_T,f_T)=1$ for a fixed $\pi_T$ in ${\cal P}$ and $tr(\pi'_T,f_T)=0$ for all $\pi'_T \not\cong \pi_T$ in ${\cal P}$.
We then call $f=f_T$ a projector for $\pi_T$ in ${\cal P}$.

\bigskip\noindent
In the cases where $S_K(G)$ defines a Hermitean symmetric space, 
one can similarly define the Lefschetz number for $L^2$-cohomology
$$ L_{(2)}(f_S,V_\lambda,\pi^S) = tr_s(f_S, H^\bullet_{(2)} (S_K(G),V_\lambda)(\pi^S)) \ .$$
By a theorem of Borel and Casselman \cite{Borel-Casselman}
$\tr(f_S, H^i_{(2)} (S_K(G),V_\lambda)(\pi^S))$ is equal to 
$$ \sum_{\pi} m_{disc}(\pi) \dim(
H^i(\gfrak ,\tilde K_\infty,
\pi_{\infty}\otimes \tau_\lambda)) \cdot\tr(\pi_S(f_S)) $$
where the sum extends over all irreducible automorphic representations $\pi$ of $G(\A)$
with given fixed $\pi^S$ that occur in the discrete spectrum of $L^2(G(\Q)\setminus G(\A),\omega)$ with multiplicity $m_{disc}(\pi)$.
Here the central character $\omega$ is determined by $\lambda$.
The cohomology groups $H^i_{(2)}(S_K(G),V_\lambda)(\pi^S))$ can be identified with
the middle weighted cohomology groups on the Baily-Borel compactification of $S_K(G)$
or the middle intersection cohomology groups of the Borel-Serre compactification
$S_K^+(G)$, and they are finite dimensional \cite{Saper_Stern}, \cite{Langlands_debut_trace_formula}, \cite{Goresky_Harder_MacPherson}, \cite{Goresky_MacPherson}.
If $\pi^S$ is not $G$-Eisenstein, by a theorem of Franke \cite{Borel_Labesse_Schwermer}, \cite{Franke}, \cite{Franke_Schwermer} 
for the ordinary cohomology, the trace
$$ \tr(f_S, H^i(S_K(G),V_\lambda)(\pi^S))$$ is equal to
$$ \sum_{\pi} m_{\cusp}(\pi) \dim(
H^i(\gfrak ,\tilde K_\infty,
\pi_{\infty}\otimes \tau_\lambda)) \cdot\tr(\pi_S(f_S)) \ $$
for the cuspidal multiplicity $m_{\cusp}(\pi)$ of $\pi$ in $L^2_{\cusp}(G(\Q)\setminus G(\A),\omega)$.
The supertrace is the alternating sum over the cohomology degrees
\begin{gather*}
\tr_s(f_S, H^\bullet(S_K(G),V_\lambda)(\pi^S))=\sum_{i=0}^\infty(-1)^i\tr_s(f_S, H^i(S_K(G),V_\lambda)(\pi^S))\\
=\sum_{\pi=\pi_\infty\pi_\fin}\chi(\pi_\infty, \lambda) m_{\cusp}(\pi) \cdot\tr\pi_S(f_S)
\end{gather*}
with the Euler characteristic
$$\chi(\pi_\infty, \lambda)=\sum_{i=0}^\infty(-1)^i\dim(H^i(\gfrak ,\tilde K_\infty, \pi_{\infty}\otimes \tau_\lambda))\ .$$
This defines the cohomological multiplicity 
$$m_{\coh}(\pi)=\chi(\pi_\infty, \lambda) m_{\cusp}(\pi)$$
and the virtual multiplicity as the sum
$$m_{\virt}(\pi_\fin)=\sum_{\pi_\infty}m_{\coh}(\pi_\infty\otimes\pi_\fin)$$ over the cohomological multiplicities of the automorphic cuspidal representations with the same non-archimedean factor  $\pi_\fin$.
Counted with their multiplicity in the cuspidal spectrum  this sum ranges over all irreducible cuspidal automorphic representations of $G(\A)$ whose non-archimedean factor is $\pi_{\fin}$.
The virtual multiplicity is the multiplicity of the $G(\A_S)$-module $\pi_{\fin}$ on the super-representation $H^\bullet(S_K(G),V_\lambda)(\pi^S)$.
The above supertrace is then the (finite) sum
$$
\tr_s(f_S, H^\bullet(S_K(G),V_\lambda)(\pi^S))=\sum_{\pi_S} m_{\virt}(\pi^S\pi_S) \tr(\pi_S(f_S))\ .$$

\section{Finiteness results and the cohomological trace formula}\label{s:Finiteness}

The central results of this section are proposition~\ref{p3.9} and its corollary
\ref{STF}. Their proof is basically contained in \cite[\S2]{Weissauer_Endo_GSp4}, but the construction
of auxiliary places there is slightly more involved due to comparisons with the Lefschetz-Grothendieck etale fixed  point formula.
For the convenience of the reader we present in this section a rough outline of the relevant arguments of \cite{Weissauer_Endo_GSp4}, which may be skipped on first reading.

\bigskip\noindent
For fixed $\lambda$ and $K$ as above, the set of isomorphism  classes of irreducible automorphic representations $\pi$ of $G(\mathbb A)$ with  nonvanishing $K$-fixed vector and with a cohomological representation $\pi_\infty$ attached to our fixed $\lambda$  is a finite set ${\cal P}={\cal P}(K,\lambda)$. 
This implies that the larger set ${\cal P}_{halo}={\cal P}(K,\lambda)_{halo}$
of isomorphism classes of irreducible automorphic representations
with nonvanishing $K$-fixed vector that occur in the halo $\Pi(\lambda)$ is also finite.
For the definition of the halo we refer to \cite[p.~45]{Weissauer_Endo_GSp4}.
The finiteness of $\mathcal{P}_{halo}$ implies that there exists a finite
set $T'$ of nonarchimedean places containing $T$ such that the following holds:
For any $\pi, \pi'$ in ${\cal P}$ and any nonarchimedean $v_0\notin T'$  
the existence of isomorphisms $\pi_v\cong \pi'_v$ for all nonarchimedean places $v\neq v_0$ implies $\pi_{v_0}\cong \pi'_{v_0}$. We call a set $T'$ with this property \textit{saturated}
with respect to $\lambda$, $K$ and $S$. If $T$ is already saturated in this sense, obviously any $\pi$ in ${\cal P}_{halo}$ is uniquely determined by $\pi_T$. Certain nonarchimedean places in the complement
of a saturated finite set $T$ will be used as \textit{auxiliary} places in the sequel.
By construction the representations $\pi^T$ attached to the representations $\pi$ 
in ${\cal P}$ define finitely many unramified representations
of $G(\mathbb A^T)$. Hence, these $\pi^T$ can be separated by the eigenvalues of suitable chosen $\Omega_{v_i}$-biinvariant spherical functions $f_i$ at finitely many auxiliary places $v_1,...,v_r$ in the complement  of $T$. This allows to define a spherical projector $$f(\pi^T)=\prod_{i=1}^r f_i \cdot \prod_{v\notin T\cup \{v_1,...,v_r\} } 1_{\Omega_v}$$ onto some fixed  representation $\pi^T$ for fixed $\pi$ in ${\cal P}$ within the set of ${\pi'}^T$ for 
all $\pi'\in {\cal P}_{halo}$. Now, also fixing this projector $f^T=f(\pi^T)$, we will also consider sufficiently large auxiliary places $v_0$ not in $T$ that are different from $v_1,...,v_r$. Their choice only depends on our fixed data $K,\lambda,T, f(\pi_T), f(\pi^T)$. Recall we fixed $K=K_TK^T$ in $G(\mathbb A_{\fin})$.

\begin{prop} \label{D1} For every $K$-biinvariant function $f\in C_c^\infty(G(\mathbb A_{\fin}))$,
$$ L(f,V_\lambda) = \sum_{P_0 \subseteq P \subseteq G}
\sum_{w\in W^P} (-1)^{\ell(w)} \cdot T_{\elliptic}^L(\overline f^{(P)}\!\cdot\!  \chi_P^G, w(\lambda+\rho_G)- \rho_L) $$ where $P=L N_P$ runs over the $\mathbb Q$-rational standard parabolic subgroups of $G$ with Levi group $L$.
The elliptic traces
$$T_{\elliptic}^L(h,\mu)=\sum_{\gamma\in L(\mathbb Q)/\sim} \chi(L_\gamma) O^L_\gamma(h)\tr(\tau_\mu(\gamma^{-1}))$$ for $h\in C_c^\infty(L(\mathbb A_{\fin}))$
 are sums over representatives $\gamma$ of the $L(\mathbb Q)$-conjugacy classes of
 strongly elliptic semisimple elements in $L(\mathbb Q)$.
\end{prop}
\begin{proof}
See \cite{Weissauer_Endo_GSp4} lemma~2.9 and its corollary~2.3.
\end{proof}

\noindent
For the definition of $\overline f^{(P)} \in C_c^\infty(G(\mathbb A_{\fin}))$, the Euler characteristics $\chi(L_\gamma)$ and the orbital integrals $O^L_\gamma(h)$ we refer to \cite{Weissauer_Endo_GSp4}, definition 2.2 and section 2.4, 2.5 and 2.6.
In every conjugacy class of parabolic subgroups there is a unique standard $P$ containing the fixed minimal $\Q$-parabolic subgroup $P_0$ of $G$.
Finally, $\chi_P^G$ denotes a real valued \textit{truncation function} on $L(\mathbb Q)$, defined by the characteristic function
 $$\chi_P^G = \tau_{F}\circ H_{P,\infty} \ .$$ 
Let us recall its definition: Let $F$ denote the subset $\Delta_F=\{\alpha_i, i\in F\}$ of the set $\Delta=\{\alpha_1,...,\alpha_r\} $ of simple $\mathbb Q$-roots  attached to $(G,P_0)$ that define the simple roots of the Levi group $L$ of $P$.
Semisimple elliptic elements $\gamma$ in $L(\mathbb Q)$
can be written in the form $\gamma = a_\infty \cdot x_\infty k_\infty x_\infty^{-1}$
for $a_\infty\in A_L(\R)^0$ and $k_\infty$ in the compact subgroup $K_{L,\infty}$
of $L(\R)$.
Then $\chi_P^G(\gamma)=1$ holds if and only if $\vert \alpha_j(a_\infty)\vert_\infty >1$ holds for all simple roots $\alpha_j, j\notin F$ of the $\Delta$, and 
$\chi_P^G(\gamma)=0$ holds otherwise. So $\chi_P^G$  restricts summation in the sum defining $T_{\elliptic}^L$ (as in \cite{Weissauer_Endo_GSp4}, p.36) to the $P$-contractive elements that arise by the work of Goresky and McPherson \cite{Goresky_MacPherson}.
 The condition $\chi_P^G(\gamma)=1$ only depends on the component $a_\infty$ of $\gamma$ in $A_L(\R)^0$, and in fact only depends on $a_\infty$
modulo the center $Z$ of $G$. To see this we recall the definition of the functions $\tau_F$ and $H_{P,\infty}$ in the next sections.

\bigskip\noindent
\textit{3. Root systems}.
For an abstract root system $\Phi$ in Euclidean space $V=\R^r$ with metric $(.,.)$ and  basis $\Delta=\{\alpha_1,...,\alpha_r\}$ let $\beta_1,...,\beta_r$ be the dual  basis defined by $2\frac{(\beta_j, \alpha_i)}{(\alpha_i,\alpha_i)}=\delta_{ij}$ for $1\leq i,j\leq r$.
Then $(\beta_i, \beta_j)\geq 0$ (resp. $>0$ for a simple root system $\Phi$)
holds for all $i,j$ and $(\alpha_i,\alpha_j)\leq 0$ holds for all $i\neq j$.
Notice $V^+ \subseteq {}^+V$ and $V^+ \cap -\overline{{}^+V}=\emptyset$, $\overline{ V^+} \cap -{}^+V=\emptyset$ for
the open cones ${}^+V= \sum_{i=1}^r \R_{>0} \cdot \alpha_i $ (the obtuse Weyl chamber)
and $V^+ = \sum_{j=1}^r \R_{>0} \cdot \beta_j$ in $V$ (the open Weyl chamber), 
and $\overline{{}^+V}$, $\overline{ V^+}$ denote their closures respectively.
For any subset $F\subseteq \{1,...,r\}$ we consider
$V_F= \bigoplus_{j\notin F} \R \cdot \beta_j$
and the orthogonal decomposition $V=V_F \oplus U$ for  $U=V_F^\perp$.
The intersection $\Phi\cap U$
of $\Phi$ with the subspace $U=\bigoplus_{i\in F} \R \cdot \alpha_i \subseteq V$ defines an abstract root system in $U$ with respect to  restricted metric and root basis $F$. By definition the orthogonal projection $$pr_F: V \to V_F$$ has kernel $U$.
Let $i_F: V_F \hookrightarrow V$ denote the inclusion and $r_F=i_F\circ pr_F: V\to V$
the corresponding retract.

\bigskip\noindent
Using $pr_F$ one defines the cones $V_F^+ := pr_F(V^+)$
and ${}^+V_F:= pr_F({}^+V)$ in $V_F$. So by definition ${}^+V_F:= pr_F(\sum_{i\notin F}
\R_{>0} \cdot \alpha_i)$. Let $\tau_F=1_{V_F^+}$ resp. $\widehat\tau_F=1_{{}^+ V_F}$ denote the characteristic functions of the open cones $V_F^+$ resp.  ${}^+ V_F$.  

\begin{rmk}
 $V_F^+ = \sum_{j\notin F} \R_{>0} \cdot \beta_j $ 
 \end{rmk}
\begin{proof}
For the convenience of the reader we recall the argument: By definition the right hand side is contained in $V_F^+$.
So, the assertion  means that  $pr_F(\beta_j)$ is contained in the right side, for each $j\in F$.
To show this consider the scalar product of $\alpha_{j'}$ for $j'\in F$ with $\beta_j$ and the orthogonal decomposition $\beta_j = \sum_{i\notin F} x_i\beta_i + \sum_{i\in F} y_i \alpha_i$.
This gives $\delta_{j'j} = \sum_{i\in F} y_i (\alpha_{j'},\alpha_i)$.
Hence $(\alpha_{j'}, v_j)= \delta_{j'j}$ for $j,j'\in F$ holds with respect to the scalar product on $U$, where $v_j=\sum_{i\in F} y_i \alpha_{i}$.
Hence the $v_j$ define the dual basis of  the root basis $\alpha_{j'}, j'\in F$ of $(U,\Phi\cap U)$. Thus $v_j \in U^+\subseteq {}^+U$, and therefore $y_i\geq 0$ for $i\in F$.
For $j\notin F$ on the other hand, considering  $( \alpha_{j},\beta_j)$
we obtain $x_j = 1  - \sum_{i\in F} y_i (\alpha_j, \alpha_i) >0$.
Indeed $j\neq i\in F$ implies $(\alpha_j, \alpha_i)\leq 0$ for the elements of the basis $\Delta$ of $\Phi$.
\end{proof}

For $x\in V_F$ it is clear that $x\in V_F^+$ if and only if $\alpha_j(x) >0$ for all $j\notin F$.

\bigskip\noindent
\begin{rmk}
  $\ -\overline{{}^+V_F} \cap \overline{V_F^+}   = \{0\} \ .$
\end{rmk}

\begin{proof}
Indeed, if $\sum_{j\notin F} x_j\beta_j $ is in $\sum_{i\notin F}
\R_{\leq 0}\cdot \alpha_i + \sum_{i\in F}
\R \cdot \alpha_i$ and $x_j\geq 0$, then $j\notin F$ holds. 
By taking scalar products with $\beta_k, k\notin F$
this implies $\sum_{j\notin F} x_j (\beta_j,\beta_k) \leq 0$.
Now from $(\beta_j,\beta_k) \geq 0$ we obtain $x_k=0$.
\end{proof}

\medskip\noindent
For each $F\subseteq \{\alpha_1,...,\alpha_r\}$ the roots $\alpha_j, j\notin F$
are nontrivial linear forms on $r_F(V)$
Composing them with the retract $r_F$, one obtains nontrivial linear forms $\alpha_j \circ r_F \circ w$
on $V$ for all $F$, all $j\notin F$ and all $w$ in the Weyl group $W$ of $\Phi$.
The kernels of these finitely many linear forms define hyperplanes in $V$ through the origin.
The complement in $V$ of the union of these hyperplanes is an open dense subset
of $V$ and decomposes into finitely many connected components, the \emph{refined Weyl chambers}, each of which is a convex cone.

\medskip\noindent
We now apply this to $V= \mathrm{Lie}(A_{L_0}/A_{L_0}\cap Z)$. 
Then $V_F = \mathrm{Lie}(A_{L}/A_{L}\cap Z)$, for the rational parabolic subgroup subgroup $P_F=L_F N_F$ of $G$ containing $P_0$.
The Levi subgroup $L_F$ has the root system that is spanned in $V$ by the $\alpha_i$ for $i\in F$ .
We consider these refined Weyl chambers for the corresponding root system over the number field $\GlobalField$,
and also for the corresponding root systems over the archimedean and nonarchimedean completions of $\GlobalField$ respectively.

\bigskip\noindent

\bigskip\noindent
\textit{4. Harish-Chandra height functions  $H_{P,v}$}. For an extension field $K$ of $\GlobalField$ let $A_{0,K}$ a maximal $K$-split torus of $G$,
let $L_{0,K}$ its centralizer in $G$, $P_{0,K}=L_{0,K}\cdot N_{0,K}$ a parabolic subgroup over $K$ containing $L_{0,K}$ as Levi subgroup. Let $\Phi=\Phi(K)$ be the root system for $A_{0,K}$ with basis $\Delta(K)$ corresponding to the choice of $P_{0,K}$.
Let $P=L \cdot N$ be a $K$-parabolic subgroup of $G$ containing $P_{0,K}$, and let
$A_{L,K}$ be the maximal $K$-split subtorus of the center of the Levi group $L$. 
Consider the projection $X_*(L)\to X_*(L)_{\Gamma_K}$ to the maximal quotient that is invariant under the absolute Galois group $\Gamma_K$ of $K$. 
We define ${}^K\! X^*(L)$ as the subgroup  of all characters $\chi\in X^*(L)$ in the character group
of $L$ that are trivial on the kernel of the projection $X_*(L)\to X_*(L)_{\Gamma_K}$.
Notice, ${}^K\!  X^*(.)$ is a contravariant functor in $L$.   
For each field extension $K\hookrightarrow K'$, we have ${}^K \!  X^*(L)\subseteq
{}^{K'} \!  X^*(L)$ and thus an induced surjective map $p_{K\to K'}$
from $\Hom({}^{K'}\!  X^*(L), \R)$ to $\Hom({}^{K}\!  X^*(L), \R)$. 
If $K=\GlobalField_v$ is a local field, there is a unique homomorphism 
$$  H_{P,\GlobalField_v}: L(\GlobalField_v)  \to \Hom({}^{\GlobalField_v}\!  X^*(L), \R) \ $$
such that
$ \langle H_{P,\GlobalField_v}(\gamma) , \chi \rangle  =  \log \vert \chi(\gamma)\vert_v $
holds for all $\gamma\in L(\GlobalField_v)$ and $\chi\in {}^{\GlobalField_v}\!X^*(L)$.

\bigskip\noindent
If $\GlobalField_v$ is nonarchimedean, this is related to \cite{Cartier}, p. 135 formula (D)
and the fact that the lattice $\Lambda$ in loc. cit. is commensurable
to the cocharacter group $X_*(A_{L,\GlobalField_v})$. 

\smallskip\noindent
If $\GlobalField_v$ is archimedean, $\Hom({}^{\GlobalField_v} X_*(L),\R)$ can be identified with the Lie algebra $\mathrm{Lie}(A_{L,\GlobalField_v})$.
This yields the alternative description
$\chi(\exp(H_{P,\GlobalField_v}(\gamma)))
=   \vert \chi(\gamma)\vert_v $ for all $\gamma\in L(\GlobalField_v)$.

\bigskip\noindent
\textit{5. Global situation}. For an embedding of the global field $\GlobalField$ into one of its completions $\GlobalField_v$ and a given $\GlobalField$-rational standard parabolic subgroup $P_F=L_F\cdot N_F$ of $G$,
containing the fixed minimal $\GlobalField$-rational parabolic $P_0$,
we consider the composite maps $H_{P,v}= p_{\GlobalField\to \GlobalField_v}\circ H_{P,\GlobalField_v}$ 
$$ H_{P,v}:  L(\GlobalField_v) \to {\cal X}_L \ ,$$
where ${\cal X}_L := \Hom({}^{\GlobalField} X^*(L), \R)$ by definition.
For nonarchimedean places $v$ the vector space
${\cal X}_L$ can be canonically identified with $\Lambda \otimes_{\mathbb Z}
\R$ for the image $$\Lambda \cong  L(\GlobalField_v)/{}^0\! L(\GlobalField_v) \ $$
 of the map $ord_L: L(\GlobalField_v) \to \Lambda$ 
where the kernel of $ord_L$ is a compact open subgroup ${}^0 L$ of $L(\GlobalField_v)$, see in \cite{Cartier}, page 134, formulas (D), (S).
Hence $H_{P,v}$ factorizes over the mapping $ord_L$. 
For the euclidean space $V=\sum_{i=1}^r \R \cdot \alpha_i$,
spanned by the roots in the fixed basis $\Delta = \{\alpha_1,...,\alpha_r\}$ of the $\GlobalField$-root system 
$\Phi$ of $(G,P_0)$, the root system $\Phi_L$ of the Levi subgroup $L$ of $G$ is spanned
by roots $\alpha_i\in \Delta$ for some subset $F\subseteq \{1,...,r\}$. 
With respect to the orthogonal decomposition $V = \bigoplus_{i\in F}\R \cdot \alpha_i
\oplus V_F$ for $V_F=\sum_{j\notin F}\R \cdot \beta_j$,
notice that we have an identification $\Lambda \otimes_{\mathbb Z} \R  \cong  V_F $.
For this use that $A_{L}(\GlobalField_v) \hookrightarrow L(\GlobalField_v)$ induces a injective map between lattices $X_*(A_L)\cong A_{L}(\GlobalField_v)/{}^0 A_L \hookrightarrow L(\GlobalField_v)/{}^0 L \cong \Lambda$.
Having the same rank $r - \#F$, these lattices are commensurable.
Since $X_*(A_L)\otimes_{\mathbb Z}\R = \sum_{j\notin F} \R \cdot \beta_j = V_F \subseteq X_*(A_{L_0})=V$,
tensoring with $\R$  induces a canonical identification
$$  V_F \ \simeq \ {\cal X}_L  \ .$$

\noindent
\textit{6. $P$-contractiveness}.
The product formula for the idele norm gives a sum formula for 
global elements $\gamma\in  L(\GlobalField)$ 
$$  \sum_v  H_{P,v}(\gamma) = 0 \ .$$
In terms of a $\mathbb Q$-root system  of $G$ with basis $\Delta=\{\alpha_1,...,\alpha_r\}$ corresponding to $P_0$
we identify  ${\cal X}_{L_0} = \Hom({}^{\GlobalField} X^*(L_0), \R)$ with
$V=\sum_{j=1}^r \R \cdot \beta_j = \sum_{i=1}^r \R \cdot \alpha_i $.
We write ${\cal X}={\cal X}_{L_0}$ and obtain ${\cal X}\cong V$.
The $\GlobalField$-parabolic subgroups
$P_F=L_F\cdot N_F$ containing $P_0$ are indexed by the subsets $F\subseteq \{\alpha_1,...,\alpha_r\}$, such that the simple roots in $L$ are those indexed by $F$.
By this description, the surjective map ${\cal X} \twoheadrightarrow {\cal X}_L$
induced from the inclusion $L_0 \hookrightarrow L$ from the contravariant functor ${}^\GlobalField X^*(.)$
can be identified with the surjective projection
$pr_F:  V \to V_F = \sum_{j\notin F} \R \cdot \beta_j$.
We thus obtain the commutative diagram 
\begin{equation*}
\xymatrix@R=+10mm@C=+25mm {
    {\cal X}  \ar@{->>}[r]^{\Hom({}^\GlobalField\! X^*(.),\R)} & {\cal X}_L  \cr
    V \ar[u]^{\simeq} \ar@{->>}[r]^{pr_F} & V_F \ .\ar[u]_\simeq
}  
\end{equation*}
A semisimple element $\gamma\in L(\mathbb Q)$, contained in a maximal torus of $L(\R)$ that is $\R$-anisotropic modulo $A_L$, is said to be \emph{$P$-contractive} if 
$\vert \sigma(\gamma) \vert_\infty >1$ holds for all simple roots (over the algebraic closure)
that occur in the expression of the roots of the Lie algebra of $N_P$ as sums of simple positive roots; for more details see \cite{Weissauer_Endo_GSp4}, p.21.
By our assumption on $\gamma$ it suffices to consider this condition for $\mathbb Q$-rational roots $\alpha$.
Let $F$ be the index set of $\mathbb Q$-roots such that the conditions $\alpha_i(\gamma)=1 $ for $i\notin F$ characterize the Lie algebra of the center of $L$.
The condition of $P$-contractiveness can then be stated as
$ (T)_\infty:  H_{P,\infty}(\gamma) \in  V_F^+ \ ,$
or equivalently, by the sum formula,  as the condition
$$ (T)_{\fin}: \quad  \quad\quad  H_L(\gamma) = -\sum_{v \neq \infty} H_{P,v}(\gamma) \in - V_F^+ \ .$$
Recall that the truncation function $\chi_P^G= H_{P,\infty} \circ \tau_P^G$ for ${\cal X}_L$
from above is the characteristic function of those $\gamma$ with the property $H_{P,\infty}(\gamma) \in  -V_F^+ $, compare \cite[p.~34]{Weissauer_Endo_GSp4}

\begin{lemma}\label{D2} With respect to the inclusion $L(\GlobalField)\subseteq L(\mathbb A_{\fin})$
the real valued function $\chi_P^G$ on $L(\GlobalField)$ can  be extended  to the 
locally constant function
$$ \chi_P^G: L(\mathbb{A}_{\fin}) \to  \R \quad , \quad
\chi_P^G(\prod_{v\neq \infty} \gamma_v)= \tau_F(- \sum_{v\neq \infty} H_{P,v}(\gamma_v)) \ .$$ 
\end{lemma}

\begin{proof} 
The extension is well-defined since each $H_{P,v}$ is zero on compact
open subgroups. Hence $\chi_P^G$ defines a locally constant function on
$L(\mathbb A_{\fin})$.
\end{proof}

\bigskip\noindent
\textit{7. Parabolic descent}. For algebraic groups $H$ over $\GlobalField_v$ we often write $H_v=H(\GlobalField_v)$. This being said, fix a suitably chosen nonarchimedean auxiliary place $v$ of $\GlobalField$ such that $G_v$ is quasi-split and splits over an unramified extension of $\GlobalField_v$.
This conditions hold for almost all places of $\GlobalField$.
Then there exists a special, good, maximal compact subgroup $\Omega_v$
in the sense of \cite[p.~140]{Cartier}.
For the Levi subgroups $L_{F,v}$ of the standard $\GlobalField_v$-parabolic subgroups $P_{F,v}=L_{F,v} N_{F,v}$ note that $L_{F,v}\cap \Omega_v$ again is a special, good, maximal compact subgroup of $L_{F,v}$;
see \cite[p.~9]{Arthur_Invariant_Trace_Formula}\,.
For $f\in C_c^\infty(G_v)$ put
$$ \overline f^{(P_v)}(\ell) =  \vert \ell\vert_v^{\rho_{P_v}} \int_{N_{P_v}}  \overline f(\ell n) dn \ $$
where $\overline f(x)=\int_{\Omega_v} f(k_v^{-1}x k_v) dk_v$.
Haar measures are normalized such that  $vol(\Omega_v)=1$ and $vol(\Omega_v\cap N_{P_v})=1$.
This defines the relative Satake transform
$$  S_{L_v}^{G_v}: C_c^\infty(G_v) \longrightarrow C_c^\infty(L_v) \quad , \quad S_{L_v}^{G_v}(f) = \overline f^{(P_v)}\ .$$
Then $h=S_{L_v}^{G_v}(f)$ is invariant under conjugation with the compact group $L_v\cap \Omega_v$ because modulus factors are trivial on compact groups. This implies the transitivity
$$  S_{L'_v}^{L_v} \circ S_{L_v}^{G_v} = S_{L'_v}^{G_v} $$
for standard $\GlobalField_v$-parabolic subgroups $P'=L'N' \subseteq P=LN \subseteq G$ containing $P_{0,v}$.
The orbital integrals
$O_\gamma^{L_v}(h) = \int_{L_{\gamma,v}\setminus L_v} h(\ell^{-1}\gamma \ell) d\overline \ell$ 
for the Levi groups have the following well known properties:

\begin{prop} \label{D3}\label{D4}
Let $L_v$ be a Levi subgroup of a $\GlobalField_v$-rational parabolic subgroup of $G$
containing the minimal $\GlobalField_v$-rational parabolic subgroup $P_{0,v}$ of $G$.
Fix $f\in C_c^\infty(G_v)$ and its relative Satake transform $S_{L_v}^{G_v}(f)\in C_c^\infty(L_v)$.
\begin{enumerate}
\item For semisimple elements $\gamma\in L_v$ that are regular in $G_v$ we have
$$ D_{L_v}^{1/2}(\gamma)\cdot O_\gamma^{L_v}(S_{L_v}^{G_v}(f)) =  D_{G_v}^{1/2}(\gamma)\cdot O_\gamma^{G_v}(f)\ . $$
where we set $D_{L_v}(\gamma) = \prod_{\alpha\in \Phi_L} \vert \alpha(\gamma)- 1\vert_v$
for the $\GlobalField_v$-root system $\Phi_L\subseteq \Phi$ that belongs to the Borel datum  $(L_v, P_{0,v}\cap L_v, L_{0,v})$.
\item For irreducible admissible representations $\sigma_v$ of $L_v$ the trace satisfies
$$\tr(f, \Ind_{P_v}^{G_v}(\sigma_v)) =\tr(S_{L_v}^{G_v}(f),\sigma_v) \ .$$
\end{enumerate}
\end{prop}
Compare Kazdhan~\cite[\S2, lemma~1]{Kazdhan}.

\bigskip\noindent

\textit{8. Satake transform}. 
Let ${\cal H}(G_v,\Omega_v)$ be the algebra of $\Omega_v$-biinvariant
function in $C_c^\infty(G_v)$, and similarly define ${\cal H}(L_v,L_v\cap \Omega_v)$
for the Levi subgroups $L_{F,v}$ of standard $\GlobalField_v$-parabolic subgroups $P_{F,v}=L_{F,v} N_{F,v}$ containing the minimal $\GlobalField_v$-parabolic group $P_{0,\GlobalField_v}$. 
The transform $S_{L_{F,v}}^{G_v}$, or $S_L^G$ for short, maps the spherical Hecke algebra ${\cal H}(G_v,\Omega_v)$ into the spherical Hecke algebra ${\cal H}(L_{F,v},L_{F,v}\cap \Omega_v)$ 
$$  S_L^G: {\cal H}(G_v,\Omega_v) \longrightarrow {\cal H}(L_v,L_v\cap \Omega_v) \ .$$
Let $\C[\Lambda]$ be the polynomial algebra whose generators $x_i$ correspond to the 
characteristic functions $1_{\lambda_i}$ for fixed generators $\lambda_i, i=1,...,r_v$ of the lattice $\Lambda$.
If $L=L_{0,v}$ is $\GlobalField_v$-minimal, we identify the convolution algebra ${\cal H}(L_{0,v},L_{0,v}\cap \Omega_v)$
with the polynomial algebra $\C[\Lambda]$ in the sense of \cite[p. 147]{Cartier} via the isomorphism
\begin{equation*}
{\cal H}(L_{0,v},L_{0,v}\cap \Omega_v)\ni h\longmapsto
\sum_{\lambda\in \Lambda} h(\lambda)
\prod_{i=1}^r x_i^{n_i} 
\in\C[\Lambda]\ .
\end{equation*}
where the monomial $x_i^{n_i}$ corresponds to $\lambda=\sum_{i=1}^r n_i\lambda_i$ in $\Lambda$.
The Weyl group $W=W_G$, defined as in \cite[p. 140]{Cartier}, acts in a natural way on $\Lambda$ and hence on $\C[\Lambda]$. We recall the Satake isomorphism $S=S_{L_{0,v}}^G$.
\begin{prop}[{\cite{Satake}}]\label{D5}
For quasi-split $G_v$ that splits over an unramified extension of $\GlobalField_v$
the Satake transform $S:  {\cal H}(G_v,\Omega_v) \to  \C[\Lambda]$
defines an injective ring homomorphism between the convolution algebra ${\cal H}(G_v,\Omega_v) $ and the polynomial ring $\C[\Lambda]$.
The image of $S$ is the subring $\C[\Lambda]^{W_G}$ in $\C[\Lambda]$ of all invariants with respect to the Weyl group $W_G$.
\end{prop}

\noindent
Hence, by the transitivity of the $S_L^G$, we obtain the commutative diagram
$$\xymatrix{
    {\cal H}(G,\Omega_v)\ar[d]_{S_L^G} \ar[r]^{S}_\simeq &  \C[\Lambda]^{W_G} \ar@{^{(}->}[d] \cr
    {\cal H}(L_v,L_v\cap \Omega_v) \ar[r]^-{S}_-{\simeq} &  \C[\Lambda]^{W_L}
}
$$
By composition with $H_{P,v}$, any function $\chi: {\cal X}_L \to \R$
can be considered as a locally constant function on $L(\GlobalField_v)$ and, by abuse of notation,
will also be denoted $\chi$ in the following.
The homomorphism $H_{P,v}$ is trivial on any compact subgroup of $L(\GlobalField_v)$,
hence it is trivial on $\Omega_v$ and compactly generated subgroups of $L(\GlobalField_v)$.
Thus, for every function $f$ in the spherical Hecke algebra ${\cal H}(L_v, L_v\cap \Omega_v)$,
the product $\chi\cdot f$ is $\Omega_v$-biinvariant and has compact support, hence
again is in ${\cal H}(L_v, L_v\cap \Omega_v)$. 

\bigskip\noindent
The projection $pr_F: V\to V_F$ induces the natural projection $pr_L: {\cal X} \to {\cal X}_L$
with respect to the identification $V={\cal X}$ and $V_F = {\cal X}_L$.
Any real-valued function $\chi$ on ${\cal X}_L$ defines a real-valued function
$\chi\circ pr_L$.
Since the unipotent radical of $L\cap P_{0,v}$ is compactly generated
for nonarchimedean places $v$, the function $\chi$ is trivial on $L\cap P_{0,v}$.
Therefore, by the definition of $S$, multiplication by $\chi$
commutes with the Satake transform $S$ in the following sense

\begin{lemma} \label{D6}
In the context of proposition~\ref{D5}, fix a function $\chi: {\cal X}_L \to \R$ .
Then %
multiplication with $\chi$ resp. $\chi\circ pr_L$ commutes with the Satake transform $S$:
$$ \xymatrix@+3mm{  {\cal H}(L_v, L_v\cap \Omega_v) \ar[d]_S  \ar[r]^{\chi} & {\cal H}(L_v, L_v\cap \Omega_v)\ar[d]^S \cr
{\cal H}(L_{0,v}, L_{0,v}\cap \Omega_v)  \ar[r]^{\chi\circ pr_L  } & {\cal H}(L_{0,v}, L_{0,v}\cap \Omega_v) } $$
\end{lemma}

\bigskip\noindent
\textit{9. Relative support}.
Let $\chi_x:\mathcal{X}_L\to \R$ denote the characteristic function of a point $x\in {\cal X}_L$.   
For $f\in {\cal H}(G_v,\Omega_v)$ in the spherical Hecke algebra and a $\GlobalField$-rational standard parabolic subgroups $P_F=L_FN_F$
as before consider $h= S_{L_F}^G(f)$ in ${\cal H}(G_v,\Omega_v)$.
The \textit{relative support} of $f$ with respect to $L_F$, by definition, will be the subset
of points $x$ in ${\cal X}_{L_F}$ such that $\chi_x \cdot h$ is a nonzero
function. 

\begin{lemma}\label{D7}
The relative support of $f\in {\cal H}(G_v,\Omega_v)$ with respect to $L=L_F$ is the image of
the support of the Satake transform $S(f):\Lambda \to \R$
under the canonical morphism $$\Lambda\hookrightarrow {\cal X} \stackrel{p_L}\to {\cal X}_L\ .$$
\end{lemma}
\begin{proof}
The function $\chi_x h$ is in ${\cal H}(L_v, L_v\cap \Omega_v)$,
so it is zero if and only if its Satake transform $S(\chi_x h)$ is zero if and only if $S(\chi_x  h)$ is zero on $\Lambda \subseteq {\cal X}$.
By lemma~\ref{D6} we therefore obtain the assertion.
\end{proof}

\bigskip\noindent
\textit{10. Good projectors}.
The following is a reformulation of lemma 2.11 in \cite{Weissauer_Endo_GSp4}:

\begin{lemma}\label{D8}
Suppose $G_v$ splits over $\GlobalField_v$.
For an $\Omega_{v}$-spherical irreducible smooth representation $\pi_{v}$ of $G_{v}$ 
and for every real constant $C>0$ there exists a spherical Hecke operator $f_{v,C}\in {\cal H}(G_{v}, \Omega_{v})$ such that
\begin{enumerate}
 \item $\tr(\pi_{v},f_{v,C})=1$,
 \item every $\mathbb Q$-rational standard parabolic subgroup $P_F=L_F N_F$ of $G$, and every $y$ in the relative support of $S_{L_F}^G(f_{v,C})$ satisfies $\vert \alpha_j(y)\vert_v \geq C$ for all $j\notin F$ .
\end{enumerate}
\end{lemma}
\begin{proof}%
One can construct functions in ${\cal H}(G_{v}, \Omega_{v})$ as in proposition~\ref{D5} by choosing a $W$-invariant polynomial in $\C[\Lambda]$ and consider its convolution action on $\pi_{v}$.
For this we fix some refined Weyl chamber $\cal C$ in ${\cal X} = \Lambda
\otimes_\Z \R\,.$

We claim there exists $x_0$ in the monoid ${\cal C}\cap \Lambda$ such that
$x_0 + x_\nu \in \mathcal{C} \cap \Lambda$ holds for a $\Z$-basis $(x_\nu)_\nu$ of the lattice $\Lambda$.
Indeed, $\Lambda \otimes_{\mathbb Z} \mathbb Q$ is dense in ${\cal X}$. Fix an arbitrary point $x_0\in \mathcal{C}\cap \Lambda$,
then for sufficiently large positive integer $M$ the elements $x_0 + M^{-1}x_\nu$
are in $\cal C$ for every $\nu=1,\dots , r$.
If we multiply by $M$ and replace $x_0$ by $M\cdot x_0$,
we see that $x_0$ and $x_0 + x_\nu$ are in $\Lambda\cap \mathcal{C}$.

We claim an $x_0$ as above can be chosen such that for every $x\in x_0+\mathcal{C}\cap \Lambda$ the function $f_x$ satisfies the second assertion.
Indeed, any $x\in \Lambda$ defines a monomial in $\C[\Lambda]$.
If we now average this monomial by the action of the Weyl group $W$ on $\C[\Lambda]$,
we obtain a polynomial $P_x \in \C[\Lambda]^W$ defining
a spherical Hecke operator $f_x$
in ${\cal H}(G_{v}, \Omega_{v})$ by proposition~\ref{D5}.
By lemma~\ref{D7} every point $y$ in the relative support of $S_L^G(f_x)$
is of the form $y= pr_L(x')$ for some $x'$ in the Weyl-orbit $W(x)$ of $x$.
All these points $x'$ are contained in some refined Weyl chamber.
Hence $\alpha_j(y)\neq 0$ holds for all $j\notin F$.

It remains to be shown that there is $x\in (x_0 + {\cal C})\cap \Lambda$ with $\tr(\pi_{v},f_x)\neq 0$, then a scalar multiple of $f_x$ satisfies the assertion.
Indeed, $\pi_{v}$ is spherical, so it is the unique spherical constituent of a representation $I_\chi$ that is induced from an unramified character $\chi$ of the Borel group of $G$ over $\GlobalField_{v}$.
Notice, $\chi$ factorizes over an unramified character of $A_{L_{v},0}$. Hence by abuse of notation $\chi$ can be considered as a character  $\chi:\Lambda\to \C^*$ of the abelian group $\Lambda$.
In this sense, proposition~\ref{D4} implies
$$\tr(\pi_{v},f_x) =\tr(I_\chi, f_x) = \sum_{w\in W} \chi^w(x)\ .$$
It remains to be shown that pairwise different characters of $\Lambda$ are linearly independent
on $(x_0+{\cal C})\cap \Lambda$, in the sense that there exist points $x$ in the cone $(x_0 + {\cal C}) \cap \Lambda$ such that $\sum_{\nu=1}^s \chi_\nu(x)$ is not zero.
Indeed, let the Weyl-orbit of the character $\chi$ consist of say $s$ pairwise different characters $\chi_1,...,\chi_s$ of $\Lambda$.
Then linear Independence is shown by the usual argument based
on induction with respect to a linear relation of minimal length between 
the characters on $x_0 + \Lambda\cap {\cal C}$; see \cite{Weissauer_Endo_GSp4}, p. 41.
This exploits the fact that the additive group $({\cal C}\cap \Lambda)-({\cal C}\cap \Lambda)$ generated by $\Lambda\cap \mathcal{C}$ contains the given basis $x_1,...,x_r$ of $\Lambda$ which is granted by our choice of $x_0$.
\end{proof}

\bigskip\noindent
\textit{11. CAP localization.}
By proposition~\ref{D1}, the Lefschetz number attached to $f\in C_c^\infty(G(\mathbb A_{\fin})$ and a local coefficient system $V_\lambda$ equals
$$ L(f,V_\lambda) = \sum_{P_0 \subseteq P \subseteq G}\sum_{w\in W^P} (-1)^{\ell(w)} 
T_{\elliptic}^L(\overline f^{(P)}\!\cdot\!  \chi_P^G, w(\lambda+\rho_G)-\rho_L)\ .$$
Fix a compact subset $K$ such that $f$ is $K$-biinvariant and a finite set of places $S$ such that $K=K_S K^S$ where $K^S= \prod_{v\notin S} \Omega_v$ is a good maximal subgroup.
For ${\cal P}={\cal P}(K,\lambda)$ as before, the Lefschetz number $L(f,V_\lambda)$ is given by the trace of a virtual representation with finitely many constituents $\pi'_{\fin}$ occurring in ${\cal P}$.
If we fix some $\pi_{\fin}$ occurring in ${\cal P}$ whose
virtual multiplicity
$m_{virt}(\pi_{\fin})\in \Z$ 
as defined in section~\ref{s:Lefschetz} is nonzero,
we may enlarge the set $S$ to a finite saturated set $T$ of nonarchimedean places such that
there exists a  $\pi^S$-projector with respect to the classes in ${\cal P}_{halo}(K,\lambda)$ given by a suitable function $$f^S=f_{T\setminus S}\otimes \prod_{v\notin T} 1_{\Omega_v}\ .$$
If furthermore $f_S$ is chosen to be a projector  onto $\pi_{\fin}$ within ${\cal P}_{halo}$, then the Lefschetz number becomes
\begin{equation*}
m_{virt}(\pi_{\fin}) \tr(\pi_{\fin}(f)) = L(f,V_\lambda)%
\ .
\end{equation*}
If we suppose that $\pi_{\fin}$ is not weakly $G$-Eisenstein, then
$$m_{virt}(\pi_{\fin})= \sum_{\pi\cong\pi_\infty\otimes\pi_{\fin}} m_{\cusp}(\pi) \chi(\pi_\infty,\lambda)$$
where $m_{\cusp}(\pi)$ is the multiplicity of $\pi$ in the cuspidal spectrum.

\begin{prop} \label{p3.9}
Suppose given a compact group $K$, a coefficient system $V_\lambda$, a compact subset $A\subseteq G(\A_{\fin})$ and an irreducible automorphic representation $\pi= \pi_\infty\otimes \pi_{\fin}$ of $G(\mathbb A)$ that is not weakly $G$-Eisenstein.
For the set of ramified places $S$ of $\pi_{\fin}$ fix a saturated set $T$ of non-archimedean places containing
$S$.
Then for every nonarchimedean place $v_0\notin T$ where $G_{v_0}$ splits there exists a spherical function $f^{trunc}_{v_0}\in {\cal H}(G_{v_0},\Omega_{v_0})$ with $\tr(f^{trunc}_{v_0},\pi_{v_0})=1$ such that for every $K$-biinvariant function
$f_S\in C_c^\infty(G(\mathbb A_{S}))$ with support of $f=f_Sf^S$ in $A$ and $f_{v_0}=1_{\Omega_{v_0}}$ the supertrace satisfies
$$\tr_s\bigl(f_S, H^\bullet(S_K(G),V_\lambda)(\pi^S)\bigr) = T_{\elliptic}^G(f_Sf^{S}(\pi^S)f^{trunc}_{v_0}, \lambda)\ $$
for the elliptic trace defined in proposition~\ref{D1}.
\end{prop}
\begin{proof}
 Fix an auxiliary place $v_0\notin T$ such that $G_{v_0}$ splits over $\GlobalField_{v_0}$.
Since $T$ is saturated, we may assume $f=\prod_{v\neq \infty} f_v$ and replace $f_{v_0}=1_{\Omega_{v_0}}$ at an auxiliary place $v_0\notin T$ by an arbitrary spherical function $f^{trunc}_{v_0}$
without changing the left side of the formula, provided $\tr(\pi_{v_0}(f^{trunc}_{v_0}))=1$ holds.
For $f^{trunc}_{v_0}=f_{v_0,C}$ as in lemma~\ref{D8} denote the function thus modified by $f^{aux}=f^{v_0}f_{v_0,C}$.
Lemma~\ref{D2} implies $\chi_P^G(\gamma) = \tau_F(- \sum_{v\neq \infty} H_{P,v}(\gamma_v)) $,
for $\gamma\in L(\mathbb Q)$ diagonally embedded into $L(\mathbb A_{\fin})$ via $\gamma_v=\gamma$.
Since $\prod_{v\neq v_0} f_v$ has support in a fixed compact set $A$,
the sum $- \sum_{v\neq v_0,\infty} H_{P,v}(\gamma)$ is bounded by some constant
$c=c(K,V_\lambda,T,\pi,A)$ for every $\gamma\in A$.
If we choose the constant $C$ sufficiently large,
then
$$ \overline f^{(P)}\!\cdot\!  \chi_P^G (\gamma)\ =\ \prod_{v\neq v_0}\overline f_v^{(P)}(\gamma)\  \otimes\
\overline f_{v_0}^{(P)}(\gamma) \cdot \tau(- H_{P,v_0}(\gamma))  \  $$
because $\chi_P^G(\gamma)$ equals $\tau(- H_{P,v_0}(\gamma))$ for $\gamma\in L(\Q)$.
This way, the truncation condition imposed by $P$-contractiveness in the topological trace formula condenses to a truncation condition at the particular place $v_0$.
As in \cite{Weissauer_Endo_GSp4}, section 2.8 therefore Frankie's theorem \cite{Franke} and proposition~\ref{D4} imply the assertion.
\end{proof}

\noindent
Notice that $\tr_s\bigl(f_Sf^S(\pi^S), H^\bullet(S_K(G),V_\lambda)\bigr) =
\tr_s\bigl(f_S, H^\bullet(S_K(G),V_\lambda)(\pi^S)\bigr)$.
Since we may choose the auxiliary non-archimedean place $v_0$ in the complement of $T$, the finiteness results for automorphic representations and the last proposition imply

\begin{cor} \label{STF} Fix an irreducible cuspidal automorphic representation $\pi= \pi_\infty\otimes \pi_{\fin}$ of $G(\mathbb A)$ that is not weakly $G$-Eisenstein.
For finitely many $K$-biinvariant nonarchimedean $f_T\in C_c^\infty(G(\A_T))$ there exists a $\pi^S$-projector $f(\pi^S)f^{trunc}_{v_0}\in C_c^\infty(G(\A^T))$ such that 
$$%
\tr_s\bigl(f_S,  H^\bullet(S_K(G),V_\lambda)(\pi^S)\bigr) = T_{\elliptic}^G(f_S f^S(\pi^S)f^{trunc}_{v_0}, \lambda) %
\ .
$$
\end{cor}

Summarizing, we can replace a good projector $f_{v_0}$ by some $f_{v_0}^{aux}$ such that $f_\pi^{aux}=f_vf_{v_0}^{aux}$
is a good projector and such that the Satake transformation of $f_{v_0}^{aux}$ has support far from the walls of the Weyl chambers.
Recall that in section~\ref{s:Lefschetz}, in contrast to Arthur~\cite{Arthur_classification_04}, we have fixed $\pi^S$ at the non-archimedean places $v\notin S$. This implies that only finitely many automorphic $\pi$ contribute to the cohomology $H^\bullet(S_K(G),V_\lambda)(\pi^S)$, so there are good projectors.

\section{Discrete series at the archimedean place} \label{sec5}

Using Weil restriction, from now on we will always assume $k=\Q$.
We suppose that $G$ contains a torus $B$ such that $B(\R)$ is a maximal torus in $G_\infty=G(\R)$ and is anisotropic modulo $A_G(\R)^0$. This assumption on $G$ is equivalent to the assumption that $G_\infty$ is an inner form of an anisotropic reductive
group over $\R$. This assumption implies the existence of discrete series representations of $G(\R)$. Put $q(G)=\frac{1}{2} \dim(X_G)$.

\bigskip\noindent
Let $\tau_\lambda$ be an irreducible finite dimensional complex representation of $G$ of highest weight $\lambda\in X^*(B)$.
For each $\tau_\lambda$ there are $d(G)\geq1$ isomorphism classes of discrete series representations by the Langlands classification where $d(G)$ is independent of $\tau_\lambda$.
To every discrete series irreducible representation $\pi_\infty$ of $G(\R)$ one can attach a (non-unique) pseudo-coefficient $f_{\pi_\infty}\in C^\infty(G(\R))$ with compact support modulo $A_G(\R)^0$,
see Clozel and Delorme \cite{Clozel-Delorme} and Kazdhan \cite[prop.~5.3]{Kazdhan}.
One can then consider the stable cuspidal function on $G(\R)$
$$f_{\tilde\lambda} = (-1)^{q(G)}  \sum_{\pi_\infty} f_{\pi_\infty}\ ,$$
where the sum runs over the $d(G)$ discrete series representations $\pi_\infty$ attached to $\lambda$ up to isomorphism.
For any tempered irreducible representation $\pi'$ of $G(\R)$ the trace $\tr( \pi'(f_{\tilde\lambda}))$ is zero unless $\pi'$ belongs to the discrete series attached to $\lambda$, in which case the trace is $(-1)^{q(G)}$.
For the contragredient representation $\mu =\tilde \lambda$ this coincides with the $C^\infty$-function $f_{\mu}$ on $G(\R)$ with compact support modulo $A_G(\R)^0$ defined by Arthur \cite[lemma~3.1]{Arthur81_Lefschetz}.

Our assumptions on $G$ imply that the centralizers $G_\gamma$ are connected algebraic groups.
A semisimple strongly elliptic element $\gamma$ is regular if and only if it is strongly regular. In that case its centralizer is conjugate to $B$.
For a local field $\GlobalField_v$ let $G_v=G(k_v)$, then the orbital integral $O_\gamma^{G_v}(f_v)$ of $f_v\in C_c^\infty(G_v)$ at an element $\gamma\in G_v$ is
defined above. %
The stable orbital integral $$SO^{G_v}_\gamma (f_v) = \sum_{\gamma'\sim\gamma} e(G_{\gamma'}) O^{G_v}_{\gamma'}(f_v)$$
is the sum over a set of representatives $\gamma'$ for the $G(\GlobalField_v)$-conjugacy classes in the stable conjugacy class of $\gamma$.
The Kottwitz signs $e(G_{\gamma'})$ are trivial if $\gamma$ is regular. For details, see Kottwitz~\cite[\S 5]{Kottwitz_2}.
Fix the smooth function
$$  %
f_\infty=f_{\infty,G} := \frac{f_{\tilde\lambda}}{d(G)} \in C^\infty(\R)\ .\ %
$$
Let $v(G_\gamma)= (-1)^{q(G_\gamma)}vol(\overline G_{\gamma,\infty}/A_G(\R)^0)d(G_\gamma)^{-1}$ as in Arthur~\cite[\S5]{Arthur81_Lefschetz}, see formula (4.5) of loc.~cit.~for $G=M$.

\begin{lemma}\label{ORB}
For all $\gamma$ in $G(\Q)$, the orbital integral of $f_{\infty,G}$ is
\begin{equation*}
O^{G(\R)}_\gamma(f_{\infty,G}) =
\begin{cases}
\frac{\tr(\tau_\lambda(\gamma^{-1}))}{d(G)\cdot v(G_\gamma)}\ &\text{if}\ \gamma\ \text{is strongly elliptic and semisimple}\ , \\0 & \text{otherwise}\ .
\end{cases}
\end{equation*}
The stable orbital integral is $SO^{G(\R)}_\gamma(f_{\infty,G}) = d(G) \cdot O^{G(\R)}_\gamma(f_{\infty,G})$ for regular $\gamma\in G(\Q)$\ .
For a fixed embedding $B\hookrightarrow H$ and every $\gamma\in B(\Q)$, we have
$$SO^{G(\R)}_\gamma( f_{\infty,G}) =  SO^{H(\R)}_\gamma( f_{\infty,H})\ .$$
\end{lemma}

\begin{proof}
The orbital integral of $f_{\widetilde{\lambda}}$ is given by thm.~5.1 of Arthur~\cite{Arthur81_Lefschetz} as
\begin{equation*}
O^{G(\R)}_\gamma(f_{\widetilde{\lambda}}) =
 \int_{G_\gamma(\R)\setminus G(\R)} f_{\tilde\lambda} (x^{-1}\gamma x) dx
= v(G_\gamma)^{-1} \sum_{\tau_1} \Phi_G(\gamma,\tau_1)\tr\widetilde{\tau_1}(f_\mu)\ .
\end{equation*}
if $\gamma$ is semisimple and is zero otherwise. The sum runs over irreducible finite dimensional complex representation $\tau_1$ of $G$.
If $\gamma$ is not strongly elliptic, then by definition $\Phi_G(\gamma,\tau_1)$ vanishes so we can assume $\gamma\in B(\R)$.
In Arthur's notation we have $\Phi_G(\gamma,\tau_1)=\tr(\tau_1(\gamma))$ for $\gamma\in B(\R)$ and
$\tr\ \widetilde \tau_1(f_\mu) = (-1)^{q(G)}\tr\ \tilde\pi_1(f_\mu)$ for any discrete series $\pi_1$ attached to $\tau_1$.
By choice of $f_\mu$ the only non-zero contribution to the sum comes from
$\tau_1=\widetilde{\tau_\lambda}$ where $\tr(\widetilde{\tau_1}(f_\mu))=1$,
so the sum equals $\Phi_G(\gamma,\widetilde{\tau_\lambda})=\tr(\tau_\lambda(\gamma^{-1}))$. This determines the orbital integral.

For regular $\gamma\in G(\Q)$ we have $v(G_\gamma)=v(B)=\mathrm{vol}(B(\R)/A_G(\R)^0)$.
For stably conjugate $\gamma\sim\gamma'$, the orbital integrals $O_\gamma^{G_\infty}(f_v)=O_{\gamma'}^{G_\infty}(f_v)$ coincide, because $v(G_\gamma)^{-1}\tr(\tau_\lambda(\gamma^{-1}))$ only depends on the stable conjugacy class of $\gamma$.
The stable conjugacy class has exactly $d(G)$ representatives $\gamma'$,
so we obtain the second assertion $SO_\gamma^{G(\R)}(f_{\infty,G})=d(G) O_\gamma^{G(\R)}(f_{\infty,G})$, compare \cite[prop.~2.2]{Shelstad_79}.
The last identity holds for regular $\gamma$ because %
$v(B)^{-1}\tr(\tau_\lambda(\gamma^{-1}))$ only depends on $\gamma$ and not on the choice of the inner form $G$ of $H$.
It holds in general by the continuity argument given by Shelstad \cite[lemmas~2.9.2 and 2.9.3]{Shelstad}.
\end{proof}

\begin{prop}\label{5.1}
For $f_{\fin}\in C_c(G(\A_{\fin}))$ with regular support the elliptic trace is
$$%
T^G_{\elliptic}(f_{\fin},\lambda) = d(G)\cdot \sum'_{\gamma\in G(\Q)/\sim} \tau_{Tam}(G_\gamma)
\cdot O_\gamma^{G(\A)}(f_{\fin} f_{\infty,G}) %
\ ,$$
where the sum extends over $G(\Q)$-conjugacy classes.
Non-zero contributions can only come from regular strongly elliptic $\gamma$.
\end{prop}

\begin{proof}
Recall that the elliptic trace is defined in prop.~\ref{D1}.
The Tamagawa number $\tau_{Tam}(G_\gamma)$ is related to the Euler characteristic $\chi(G_\gamma)$ that appears in the definition of the elliptic trace via $\chi(G_\gamma)= v(G_\gamma)^{-1}\tau_{Tam}(G_\gamma)$; see \cite{Weissauer_Endo_GSp4}, p.~31.
Lemma~\ref{ORB} implies
$$ \tau_{Tam}(G_\gamma) \cdot O^{G(\R)}_\gamma(f_{\infty,G}) = \frac{\tau_{Tam}(G_\gamma)\tr(\tau_\lambda(\gamma^{-1})}{v(G_\gamma)}= \chi(G_\gamma) \cdot\tr(\tau_\lambda(\gamma^{-1}))\ .$$
By definition of the elliptic trace this implies the statement.
If $\gamma$ is not strongly elliptic and regular, the corresponding term vanishes by lemma~\ref{ORB}.
\end{proof}

\section{Stabilization of the trace formula}\label{s:stab_trace}

We give a short summary of results by Kottwitz \cite{Kottwitz_2, Kottwitz_3}.
All fields $\GlobalField$ (resp. $\GlobalField_v$) are assumed to be global (resp.~local) fields  of characteristic zero with absolute Galois group $\Gamma$ (resp. $\Gamma_v$).
Let $\gamma$ be semisimple in $G(\GlobalField)$ and $I=(G_\gamma)^0$ be the connected
component of the centralizer; notice $I=G_\gamma$ under our assumption that $G_{\der}$ is simply connected.
The set $${\cal D}_G(I/\GlobalField)=\ker(H^1(\GlobalField,I)\to H^1(\GlobalField,G)) $$ can be identified with a set of representatives of the $G(\GlobalField)$-conjugacy classes in the stable conjugacy class of $\gamma$ if $G_{\der}$ is simply connected.
Analogous statements hold for $\gamma\in G(\GlobalField_v)$ in the local case. 

\bigskip\noindent
Over a local field $\GlobalField_v$ there exists a map of sets ${\cal D}_G(I/\GlobalField_v) \rightarrow {\cal E}(I/\GlobalField_v)$
and a commutative diagram with exact horizontal lines
$$ \xymatrix{
    {\cal D}_G(I/\GlobalField_v)\ar[d] \ar@{^{(}->}[r] & H^1(\GlobalField_v,I) \ar[d]_{\alpha_I} \ar[r] & H^1(\GlobalField_v,G) 
    \ar[d]_{\alpha_G} \cr
    {\cal E}(I/\GlobalField_v) \ar@{^{(}->}[r] & A(I) \ar[r] & A(G) 
 } 
$$
where $A(X):= \pi_0(Z(\widehat X)^\Gamma)^D$ are abelian groups defined for an arbitrary reductive connected group $X$ over $\GlobalField_v$ in terms of the Langlands dual groups $\widehat X$.
The vertical abelianization maps $\alpha_X$ are defined via local duality theory for Galois cohomology \cite{Kottwitz_3}, thm.~1.2.
If $X=T$ is a torus, then $\alpha_T: H^1(\GlobalField_v,T) \cong A(T)$ is a group isomorphism \cite{Kottwitz_2}, (3.3.1). If $\gamma$ is regular, $I$ is a torus and the map ${\cal D}_G(T/\GlobalField_v) \rightarrow {\cal E}(T/\GlobalField_v)$ is injective.
For $\gamma,\gamma'\in {\cal D}_G(I/\GlobalField_v)$
let $\mathrm{inv}_v(\gamma,\gamma')\in {\cal E}(I/\GlobalField_v)$ be the difference of the images of $\gamma,\gamma'$  in the group ${\cal E}(I/\GlobalField_v)$, see \cite[\S5.6]{Kottwitz_3}. If $\GlobalField_v$ is nonarchimedean, then ${\cal D}_G(I/\GlobalField_v) \rightarrow {\cal E}(I/\GlobalField_v)$ induces a bijection.
If $\GlobalField_v=\R$,
then $\ker(\alpha_X)= \im(H^1(\R,X_{sc})\to H^1(\R,X))$ and $\im(\alpha_X)= \ker(A(X) \to \pi_0(Z(\widehat X))^D)$ by Kottwitz~\cite[thm.~1.2]{Kottwitz_3}.

\bigskip\noindent
For each place $v$ of $\GlobalField$, one defines a group homomorphism
${\cal K}(I/\GlobalField)\to {\cal K}(I/\GlobalField_v)$
by the following commutative diagram with exact rows in the category of abelian groups
\begin{equation*}
\xymatrix@+3mm{
1\ar[r]&{\cal K}(I/\GlobalField) \ar[d]\ar@{^{(}->}[r] &  \pi_0\bigl((Z(\widehat I)/Z(\widehat G))^\Gamma \ar[r]\ar[d]& \Sha^1(\GlobalField,Z(\widehat{G}))  \ar[d]^{=1} \cr
1\ar[r]&{\cal K}(I/\GlobalField_v) \ar@{^{(}->}[r] &  \pi_0\bigl((Z(\widehat I)/Z(\widehat G))^{\Gamma_v}  \ar[r] &
H^1(\GlobalField_v,Z(\widehat{G}))\ . 
}
\end{equation*}

\bigskip\noindent
\begin{defn}\emph{An endoscopic triple} $(H,s,\eta)$ for $G$ as in \cite[\S 7.4]{Kottwitz_2} or \cite[\S9.3]{Kottwitz_3} is composed of a quasi-split reductive connected group\footnote{The letter $H$ in this section denotes a quasi-split reductive group that belongs to an endoscopic triple. This can be an inner form of $G$, but not necessarily.} $H$ over $\GlobalField$, an
embedding $\eta: \widehat H \to \widehat G$ and an element $s\in Z(\widehat H)$ satisfying the conditions:
\begin{enumerate}
\item $\eta(\widehat{H})$ is the connected component of the centralizer $\widehat{G}_{\eta(s)}$.
\item The $\widehat{G}$-conjugacy class of $\eta$ is fixed by $\Gamma$.
\item The image of $s$ in $(Z(\widehat H)/Z(\widehat G))$ is fixed by $\Gamma$.
\item Suppose $\gamma_H\in H(\GlobalField)$ comes from the fixed $\gamma\in G(\GlobalField)$ by an admissible embedding. Then the image of $s$ via the mapping $Z(\widehat H) \to Z(\widehat{H_{\gamma_H}})\cong
Z(\widehat I)\to \pi_0((Z(\widehat I))/Z(\widehat G))^\Gamma$ defines an element $\kappa$ in ${\cal K}(I/\GlobalField)$.
\end{enumerate}
The definition of a local endoscopic triple is analogous, replace $\GlobalField$ by
$\GlobalField_v$ and $\Gamma$ by $\Gamma_v$.
\end{defn}
The finite abelian group $\mathcal{K}(I/\GlobalField_v)$ is isomorphic to the Pontryagin dual of the finite group ${\cal E}(I/\GlobalField_v)$ \cite[\S 4.6]{Kottwitz_3}.
Hence the restriction of $\kappa\in {\cal K}(I/\GlobalField)$ to ${\cal K}(I/\GlobalField_v)$ induces a group homomorphism $\kappa: {\cal E}(I/\GlobalField_v)
 \to S^1 \subseteq \C^\times$. Via Pontryagin duality, for $\mathrm{inv}_v(\gamma,\gamma') \in {\cal E}(I/\GlobalField_v)$ this defines $\langle \mathrm{inv}_v(\gamma,\gamma'), \kappa\rangle\in S^1$.
 
 \bigskip\noindent
\textit{Archimedean case.} For $\GlobalField_v\cong \R$ and a strongly elliptic and $G$-regular $\gamma\in G(\GlobalField)$,
the group $I_v=G_{\gamma,v}$ is conjugate to $B(\R)$ with the torus $B$ fixed in section~\ref{sec5}.
If $\Gamma_\infty$ acts trivially on $G^{ab}=G/G_{\der}$,
then $(Z(\widehat I)/Z(\widehat G))^{\Gamma_\infty} \cong\{\pm 1\}^r$ where $r=\dim(B/Z(G))$.
Hence there is an embedding 
$${\cal K}(I/\GlobalField) \hookrightarrow \{\pm 1\}^r  $$
and a bijection ${\cal E}(I/\R) \cong \{\pm 1\}^r$ such that
$$ {\cal D}_G(I/\R) \hookrightarrow \{\pm 1\}^r \ .$$
If $I$ is a maximal anisotropic torus $B$ of $G$ and $G_{\der}$ is simply connected, 
${\cal D}_G(I/\R)$ can be viewed as $B(\C)\!\setminus\! \{g\in G(\C)\, \vert \, g^{-1}B(\R)g \subseteq G(\R)\}/G(\R)$. For this choose the representatives $g$ in the normalizer of $B(\C)$ in $G(\C)$. This identifies the double coset with a right coset
of the Weyl group $\Omega=\Omega_G$ of $B(\C)\subseteq G(\C)$ by the real Weyl group
$\Omega_{G(\R)} = \Omega(B(\R),G(\R))$ as in \cite[lemma~2.4.1]{Shelstad}.
Via $\Omega\ni w\mapsto F_\infty(w)w^{-1}$ each double coset defines a cohomology class in $H^1(\R,B)$, which identifies the set ${\cal D}_G(B/\R) \subseteq {\cal E}(B/\R)$ and the set $B(\C)\!\setminus\! \{g\in G(\C)\, \vert \, g^{-1}B(\R)g \subseteq G(\R)\}/G(\R)$
as in \cite{La}.

\bigskip\noindent
\textit{Matching condition}. Fix a global elliptic endoscopic triple $(H,s,\eta)$ for $G$ as above.
Then for a semisimple element $\gamma_H$ of $H(\GlobalField)$ choose a torus $T_H\subseteq H$ containing $\gamma_H$
and an admissible embedding $j:T_H \to G$ and let $\gamma=j(\gamma_H)$.
The conjugacy class of $\gamma$ does not depend on these choices.
Langlands and Shelstad~\cite{Langlands_Shelstad} have defined transfer factors $$\Delta(\gamma_H,\gamma)=\prod_v\Delta_v(\gamma_H,\gamma)$$ that allow to define matching conditions for pairs of functions
$f=\prod_v f_v \in C_c^\infty(G(\mathbb A))$ and 
$f^H= \prod_v f_v^H \in C_c^\infty(H(\mathbb A))$: 
For $\kappa$ defined by $s$ as in the definition of the endoscopic triple, the $\kappa$-orbital integral is
$$ O^\kappa_\gamma(f_v)
    = \sum_{\gamma'} \langle inv_v(\gamma,\gamma'),\kappa\rangle   e(G_{\gamma'})O^{G(\GlobalField_v)}_{\gamma'}(f_v) \ ,$$
where $\gamma'$ runs over a set of representatives for the $G(\GlobalField_v)$-conjugacy
classes in $G(\GlobalField_v)$ in the stable conjugacy class of $\gamma$. 
By definition, the matching conditions for $(f,f^H)$ are satisfied if the formula
$$ SO_{\gamma_H}^{H(\GlobalField_v)}(f^H_v) = \Delta_v(\gamma_H,\gamma) O^\kappa_\gamma(f_v)
$$ holds for all $(G,H)$-regular semisimple $\gamma_H\in H(\GlobalField)$.
For details, see \cite{Kottwitz_Shelstad}.
If there is no $\gamma\in G(\GlobalField_v)$ belonging to $\gamma_H\in H(\GlobalField_v)$,
then the local transfer factor $\Delta_v(\gamma_H,\gamma)$ vanishes and the matching condition demands
that $SO_{\gamma_H}^{H(\GlobalField_v)}(f^H_v)$ has to vanish \cite[\S1.4]{Langlands_Shelstad}.
Notice $\Delta_v(\gamma_H,\gamma')= \Delta_v(\gamma_H,\gamma) \langle inv_v(\gamma,\gamma'),\kappa\rangle $ as in \cite[\S 5.6]{Kottwitz_3}, so the matching condition can also be stated as
$$ SO_{\gamma_H}^{H(\GlobalField_v)}(f^H_v) = \sum_{\gamma'} \Delta_v(\gamma_H,\gamma') e(G_{\gamma'})O^{G(\GlobalField_v)}_{\gamma'}(f_v)\ .$$
Recall that for regular semisimple $\gamma'$ in $G(\GlobalField)$, the Kottwitz sign is $e(G_{\gamma'})=1$.

\bigskip\noindent
\textbf{Remark}. Suppose $f$ is the good projector in the sense of section~\ref{s:Finiteness} at an auxiliary place $v_0$ where $D$ splits.
Then $O_{\gamma'}(f)\neq0$ implies that $\gamma'$ is regular.
 \begin{lemma} 
 If $H$ is an inner form of $G$ and $(H,1,id)$ is the maximal endoscopic datum,
 and if $G$ and $H$ both admit a maximal torus $B$ as in section~\ref{sec5}, then up to a constant scalar
 $f_{\infty,G}$ matches with $$f_{\infty,G}^H=
 f_{\infty,H}\ .$$
 \end{lemma}

\begin{proof}
This follows from the last assertion of lemma~\ref{ORB}.
\end{proof}

\bigskip\noindent
\textit{Stabilized trace formula.}
Langlands \cite{Langlands_debut_trace_formula} and Kottwitz \cite[\S9.6]{Kottwitz_3} have shown that the fundamental lemma, proven by Ng\^{o} \cite{Ngo_Fundamental_Lemma}, implies

\begin{thm}\label{STABI}
For every matching pair $(f,f^H)$ of test functions with $f=f_{\infty,G}f_{\fin}$ where $f_{\fin}\in C_c^\infty(G(\mathbb A_{\fin}))$, there is an equality
\begin{equation*}
\sum_{\gamma\in G(\Q)/\sim} \tau_{Tam}(G_\gamma) O_\gamma^{G(\A)}(f)
=   \sum_{(H,s,\eta)} \iota(G,H) ST_{\elliptic}^{H(\mathbb A),**}(f^H) \ .
\end{equation*}
The left hand sum runs over representatives $\gamma$ of the elliptic semisimple conjugacy classes in $G(\GlobalField)$ not in the center of $G$.
The right hand sum side runs over representatives $(H,s,\eta)$ of the isomorphism classes of elliptic endoscopic triples  for $G$.
\end{thm}
\noindent
Here the the stable elliptic trace is defined by
$$ST_{\elliptic}^{H(\mathbb A),**}(h) := \tau_{Tam}(H) \cdot \sum'_{\gamma_H/\sim}
SO^{H(\mathbb A)}_\gamma(h) \ ,$$
the sum runs over certain representatives $\gamma_H$ of the elliptic semisimple stable
conjugacy classes in $H(\GlobalField)$ specified by Kottwitz \cite[9.5]{Kottwitz_3}.

\medskip\noindent
\textit{The constants $\iota(G,H)$}.
These constants are equal to $\tau_{Tam}(G)\tau_{Tam}(H)^{-1}\lambda^{-1}$, where
$\lambda$ is the cardinality of $\mathrm{Aut}(H,s,\eta)/H_{ad}(\GlobalField)$ \cite[thm.~8.3.1]{Kottwitz_2}.
If $s\in Z(\widehat G)$, then $H$ is isomorphic to the quasi-split inner form $H$
of $G$ and thus $\widehat H=\widehat G$ defines a maximal elliptic endoscopic triple $(H,s,id_{\widehat H})$.
In this case $Aut(H,s,\eta)/H_{ad}(\GlobalField)$ is trivial, hence $\iota(G,H)= \tau_{Tam}(G)/\tau_{Tam}(H)$.
By the main result of Kottwitz \cite{Kottwitz_4}, $\iota(G,H)=1$ if $H$ is the quasi-split inner form of $G$.

\medskip\noindent
\textit{Support}.
In the following we only consider the case where $f=\prod_v f_v$ and at least for one nonarchimedean place $v$ the function $f_v$ has $G(\GlobalField_v)$-regular support.
Then in theorem~\ref{STABI} sums could be as well over the representatives $\gamma$ of the elliptic semisimple
conjugacy classes in $G(\GlobalField)$ resp.\ representatives $\gamma_H$
of the elliptic semisimple stable conjugacy classes in $H(\GlobalField)$. 

\medskip\noindent
\textit{Weakly endoscopic representations}. Let $(H,s,\eta)$ be an elliptic endoscopic triple for $G$.
Then for an irreducible automorphic representation $\pi'$ of $H(\mathbb A)$
there exists a finite set of places $S(H,\pi')$ such that the fundamental lemma by Ng\^{o}~\cite{Ngo_Fundamental_Lemma} defines an irreducible spherical representation $\tilde \pi^S$ of $G(\mathbb A^{S(H,\pi')})$ associated to
$\pi'$.
An irreducible automorphic representation $\pi$ of $G(\mathbb A)$
will be called weakly endoscopic if there exists an elliptic endoscopic triple for $G$ such that $s\not\in Z(\widehat G)$,
an irreducible automorphic representation $\pi'$ of $H(\mathbb A)$
with associated representation $\tilde \pi^S$ of $G(\mathbb A^{S(H,\pi')})$
and  a finite set of places $S\supseteq S(H,\pi')$ such that $\pi_v \cong \tilde\pi_v$
holds for all $v\notin S$.

\medskip\noindent
\textit{Non-archimedean places.} At the nonarchimedean places for  us it suffices 
to consider only functions $f_v$ in the spherical Hecke algebra for $v\notin S$, 
or functions $f_v,v\in S$ with regular semisimple support.
If the function $f$ has $G$-regular support, then
matching functions $f^H$ can be easily constructed by the inverse function theorem.
For functions in the spherical Hecke algebra of $G$ the fundamental lemma~\cite{Ngo_Fundamental_Lemma}
gives matching functions in the spherical Hecke algebra of $H$.

\section{Archimedean matching}\label{s:arch_match}

\medskip\noindent
The function $f_\infty=f_{\infty,G}$ on $G(\R)$ specified in section~\ref{sec5} is constant on the conjugacy classes within a stable conjugacy class and is zero on semisimple classes unless they are strongly elliptic.
For a fixed endoscopic datum $(H,s,\eta)$ there exist corresponding archimedean transfer factors
$\Delta(\gamma_H,\gamma)$
such that the normalized archimedean $\kappa$-orbital integral
$$ \sum_{\gamma'\sim\gamma} \Delta(\gamma,\gamma')e(G_{\gamma'})  O^G_{\gamma'}(f_{\infty,G})$$
is well-defined \cite{Shelstad}. 
Since $f_{\infty,G}$ is a stable function, this may  be written as 
$$ d(H,s,\eta,\gamma) \cdot \Delta(\gamma_H,\gamma)O^G_\gamma(f_{\infty,G}) \ $$
for $ d(H,s,\eta,\gamma) = \sum_{\gamma'\sim\gamma} \left<inv_\infty(\gamma,\gamma'),\kappa\right> e(G_{\gamma'})$.  For representatives $w$
of ${\cal D}_G(B/\R)$ in $N_G(B)(\C)$,  we can write for $\gamma' = \gamma^w$ %
$$ d(H,s,\eta,\gamma) = \sum_{w\in \mathcal{D}_G(B/\R)} \kappa(w) \ $$
via a character $\kappa: {\cal E}(B/\R) \to \{\pm 1\}$ determined by the data $(H,s,\eta)$.
Hence $d(H,s,\eta,\gamma)$ is an integer, defined up to a sign by the sum on the right hand side,
where the sign depends on the choice of the reference point $\gamma$.
This sign compensates the dependency of
$\Delta(\gamma_H,\gamma)$ on the choice of $\gamma$ within the stable conjugacy class. 
Notice $O^\kappa_\gamma(f_{\infty,G})$ vanishes unless $\gamma$ is strongly elliptic.
If $\gamma$ is strongly elliptic and $G$-regular, then $I=G_\gamma$ is conjugate to the anisotropic torus $B$. This being said, we summarize 
\begin{equation*}
\sum_{\gamma'\sim\gamma} \Delta(\gamma,\gamma')e(G_{\gamma'})
 O^G_{\gamma'}(f_{\infty,G}) = d(H,s,\eta,\gamma) \cdot  \frac{\Delta(\gamma_H,\gamma)\tr(\tau_\lambda(\gamma^{-1}))}{d(G)\cdot v(G_\gamma)} \ .
 \end{equation*}

\medskip\noindent
\textit{Example}. If the group $G$ is anisotropic over $\R$, then  ${\cal D}_G(B/\R)$ has cardinality one.
For global endoscopy this implies  that for global $\kappa$ the value $\langle inv_\infty(\gamma,\gamma'), \kappa \rangle$
is one for $\gamma'$ stably conjugate to $\gamma$ in $G(\R)$.
In other words the matching conditions for $\kappa$-orbital integrals
in the stabilization of the trace formula will simplify at the archimedean places
so that at the archimedean place $d(H,s,\eta,\gamma)=1$ holds for all global endoscopic triples $(H,s,\eta)$.

\medskip\noindent
\textit{Notations}. Now recall that assigned to every $V_\lambda$ there is an $L$-packet
of discrete series representations $\pi_\infty$. We have constructed the function $f_{\infty,G}$ on $G(\R)$ defined
by $d(G)^{-1} (-1)^{q(G)} \sum_{\pi_\infty} f_{\pi_\infty}$ from the pseudo-coefficients
$f_{\pi_\infty}$.
Let $\varphi$ denote the Langlands parameter describing this archimedean $L$-packet.
As in section~\ref{sec5} we now consider the set $\Phi_H(\varphi)$ of $L$-packets of discrete series representations on $H(\R)$ whose Langlands parameter
$\varphi_H\in  \Phi_H(\varphi)$
is assigned to the Langlands parameter $\varphi$ by $\eta$ via an equivalence between
$\eta\circ \varphi_H$ and $\varphi$.
Recall $f_{\infty,H}$, depending on $\varphi_H$ and therefore now better denoted $f_{\infty,H,\varphi_H}$, is defined by
$d(H)^{-1} (-1)^{q(H)} \sum_{\pi_\infty} f_{\pi'_\infty}$ for representatives $\pi'_\infty$
of the $L$-packet $\Pi(\varphi)$ and their pseudo-coefficients $ f_{\pi'_\infty}$.

\begin{lemma} \label{TRANSERARCH}
For every endoscopic triple $(H,s,\eta)$ attached to $G$ there is $f_{\infty,G}^H\in C^\infty(H(\R))$, such that for 
(regular\footnote{The statement holds for all semisimple elements, but for our application we only need it for regular semisimple $\gamma$. The general case follows by a limiting argument.})
$\gamma\in G(\R)$ the orbitals $O^H_{\gamma_H}(f_{\infty,G}^H)$ vanish unless
$\gamma_H$ comes from the anisotropic torus $B$ of $G$, in which case
$$ SO^{H(\R)}_{\gamma_H}(f_{G,\infty}^H)
=  \sum_{\gamma'\sim\gamma} \Delta_\infty(\gamma,\gamma')e(G_{\gamma'})  O^G_{\gamma'}(f_{G,\infty}) \ .$$
The function $f_{G,\infty}^H$ can be chosen as a linear combination\footnote{The signs $\det(\omega_*(\varphi_H))$ depend on an identification $\varphi \mapsto w_*(\varphi)$ of $\Phi_H(\varphi)$ with a subset $\Omega_*$
of $\Omega$ made with choices of a certain isomorphism $j$ and a certain Borel group in $G(\C)$.}
\begin{equation*}
f_{G,\infty}^H = (-1)^{q(G)+q(H)} \frac{d(H,s,\eta,j(\gamma_H))}{d(G)}
\sum_{\varphi_H\in \Phi_H(\varphi)} \det(w_*(\varphi_H)) \,
f_{H,\infty,\varphi_H}
\end{equation*}
of functions $f_{H,\infty}= f_{H,\infty,\varphi_H}$ attached to $L$-packets $\varphi_H$ of discrete series representations on $H$ assigned to $\varphi$ by $\eta$.
\end{lemma}

\begin{proof}
This follows from the proof of formula (7.4) 
in Kottwitz~\cite[p.~182--187]{Kottwitz_5}.

Here the factor $d(H,s,\eta,\gamma)$ plays a similar role as
the factor $\langle\beta(\gamma),s\rangle$ in loc. cit., and our $h_{\infty}^H$
is the analogue of $h_\infty$ in loc. cit. Put 
$$ h^H_\infty =  c \cdot \sum_{\varphi_H\in \Phi_H(\varphi)} \det(\omega_*(\varphi_H)) (-1)^{q(G)+q(H)}\,
f_{H,\infty,\varphi_H} $$
for $f_{H,\infty,\varphi_H} =  (-1)^{q(H)} d(H)^{-1} \sum_{\pi'_\infty\in \Pi(\varphi_H)} f_{\pi'_\infty}$ 
and $d(H) = \#(\Omega_H/\Omega_{H(\R)})$. 
In the notations of \cite{{Kottwitz_5}}, here
 $\varphi_H$ ranges over a system of representatives of an
$L$-packet $\Pi(\varphi_H)$ of discrete series representations $\pi_H$ of $H(\R)$ 
with Langlands parameter $\varphi_H$ such that $\eta \circ \varphi_H$
is equivalent to the Langlands parameter $\varphi$ corresponding to the discrete series $L$-packet of $G(\R)$
that underlies the definition of $f_{\infty,G}$ for $V_\lambda$. Notice $A_G(\R)=A_H(\R)$.  The constants $\det(\omega_*(\varphi_H))$ will be discussed below.
Notice, the suitable chosen factor $c$ is an analogue of $\langle \mu_h,s\rangle$ in \cite{{Kottwitz_5}}.

\medskip\noindent
Temporarily we write $T$ for our fixed anisotropic torus $B$ in $G(\R)$,
to avoid confusion with the different meaning of $B$ (as a Borel group in $G(\C)$
containing $T$) in \cite{Kottwitz_5}.
To compare $ \sum_{\gamma'\sim\gamma} \Delta_\infty(\gamma,\gamma')e(G_{\gamma'})
 O^G_{\gamma'}(f_{\infty,G}) $ with $ SO^{H(\R)}_{\gamma_H}(h_\infty^H)$,  it suffices   
 to assume  $\gamma=j(\gamma_H)$ 
and that $\gamma$ is $G$-regular and contained in $T(\R)$ for some fixed $j: T_H \to T$
  with $j\in J$; as in \cite{Kottwitz_5}, p.\,183--186. 
 After this rigidification of $\gamma$ one may fix a pair
 $(j,B)$ as in \cite{Kottwitz_5} to rigidify the choice of transfer factor. It becomes $$\Delta_{\infty}(\gamma_H;\gamma)= \Delta_{j,B}(\gamma_H, \gamma) = (-1)^{q(G)+q(H)} \chi_{G,H}(\gamma)
 \Delta_B(\gamma^{-1})\Delta_{B_H}(\gamma_H^{-1})^{-1}$$ for $B_H$ assigned 
 to $(j,B)$, as in the work of Shelstad.
 Here $\chi_{G,H}(\gamma) =
 \chi_B(\gamma)\chi_{B_H}(\gamma_H^{-1})$, for certain quasicharacters defined in \cite{Kottwitz_5}.
Depending on the choice of $(j,B)$, one obtains a set $\Omega_*\subseteq  \Omega$ of representatives of the cosets $\Omega_H\setminus \Omega$ in the Weyl group $\Omega=\Omega_G$ of $(G(\C),B(\C))$; see \cite{Kottwitz_5}, p.184.
So elements $w$ of 
 $\Omega$  can be written uniquely in the form $w =w_Hw _*$ for $w_H\in \Omega_H$. 
Furthermore the choice of $(j,B)$ defines a bijection $\Phi_H(\varphi) \cong \Omega_*$
that aligns $\varphi_H\in \Phi_H(\varphi)$ to $w_*(\varphi_H)\in \Omega_*$; \cite{Kottwitz_5}, p.185.
For the sign $\det(w_*(\varphi_H))$, defined by the natural action of $w_*(\varphi_H)\in \Omega_G$ on $X^*(T)\otimes_{\mathbb Z} \R$, one has  
the key formula $\Delta_{j,w_Hw_*(B)}(\gamma_H, j(\gamma_H)) = \det(\omega_*(\varphi_H))\cdot
\Delta_{j,B}(\gamma_H, j(\gamma_H))$. Combined with the character formulas for discrete series representations this allows to compute
$SO_{\gamma_H}(h_\infty^H)$. Indeed,
as shown in \cite{Kottwitz_5}, p. 186, the stable orbital integral of $h_\infty^H$ is zero unless $\gamma_H$ is (strongly) elliptic, and for elliptic elements it satisfies
$$ SO_{\gamma_H}(h_\infty^H) = \frac{c\, e(H_\gamma)}{ vol(H_\gamma/A_G(\R)^0 )}\cdot \Delta_{j,B}(\gamma_H,\gamma)  \tr(\tau_\lambda(\gamma^{-1})) \ .$$
If $\gamma$ is $H$-regular, $e(H_\gamma)=1$ and $vol(H_\gamma/A_G(\R)^0)
= vol(B(\R)/A_G(\R)^0)$. 
We briefly summarize the arguments for this:
First lemma~\ref{ORB} gives 
$$SO_{\gamma_H}(f_{H,\infty,\varphi_H}) = v(H_{\gamma_H})^{-1} tr(\tau_\lambda(\gamma_H^{-1})) \ .$$ 
Then one expresses $(-1)^{q(H)}tr(\tau_\lambda(\gamma_H^{-1}))$ by 
the character values $\chi_{\pi'_\infty}(\gamma_H)$, summed over the 
$\#(\Omega_H/\Omega_{H(\R)})$ members 
of the discrete series $L$-packet of $H(\R)$ attached to $\varphi_H$.  
For the elliptic Langlands parameter $\varphi_H$ each choice of a Borel group $B_H$ of $H(\C)$ containing $T_H(\C)$ allows to define an irreducible discrete series representation $\pi'_\infty(\varphi_H,B_H)$ in the $L$-packet of $\varphi_H$, uniquely characterized by its character 
$$ \chi_{\pi'_\infty(\varphi_H,B_H)}(\gamma_H^{-1}) = (-1)^{q(H)} \sum_{w_H\in \Omega_{H(\R)}}
 \chi_{w_H(B_H)}(\gamma_H^{-1})
\Delta_{B_H}(\gamma_H^{-1})^{-1} $$
on regular elements $\gamma_H\in T_H(\R)$ in the anisotropic torus $T_H$ of $H(\R)$.
See \cite[p.~183]{Kottwitz_5}, where this is stated for $G$ instead of $H$.
Summing up these characters for the representations in the $L$-packet, one obtains 
the formula
$$ SO_{\gamma_H}(f_{H,\infty,\varphi_H})= v(H_{\gamma_H})^{-1}
\sum_{w_H\in \Omega_H} \chi_{w_H(B_H)}(\gamma_H^{-1}) \Delta_{w_H(B_H)}(\gamma_H^{-1})^{-1} \ .$$
So by $\gamma= j(\gamma_H)$ and the definition of the transfer factor,
this is equal to
$$   (-1)^{q(G)+q(H)} \frac{1}{v(H_{\gamma_H})}
\sum_{w_H\in \Omega_H} \Delta_{j,w_H(B)}(\gamma_H, \gamma) \chi_{w_H(B)}(\gamma^{-1}) \Delta_{w_H(B)}(\gamma^{-1})^{-1} \ .$$
$\Phi_H(\varphi)$ can be identified with $\Omega_*$ in such a way that
a change of the $L$-packet $\varphi_H\in \Phi_H(\varphi)$  amounts to
replace $B$ by $w_*(B)$ to obtain $SO_{\gamma_H}(f_{H,\infty,w_*(\varphi_H}))$,
Hence the formula $\Delta_{j,w_Hw_*(B)}(\gamma_H, j(\gamma_H)) = \det(\omega_*(\varphi_H))\cdot
\Delta_{j,B}(\gamma_H, j(\gamma_H))$ for $w$ in $\Omega_*$ and $\det(w_*)^2 =1$, together with the decomposition $\Omega = \Omega_H \Omega_*$, imply for $\gamma=j(\gamma_H)$  
$$ SO_{\gamma_H}(h_\infty^H) = c\cdot \frac{\Delta_{j,B}(\gamma_H,\gamma)}
{v(H_{\gamma_H})} \sum_{w\in \Omega} 
\chi_{w(B)}(\gamma^{-1}) \Delta_{w(B)}(\gamma^{-1})^{-1} \ $$
from the
definition of $h_\infty^H$. The sum on the right side gives $tr(\tau_\lambda(\gamma^{-1}))$
by the Weyl character formula. So for $\gamma= j(\gamma_H)$ we obtain
$$ SO_{\gamma_H}(h_\infty^H) = c\cdot \frac{\Delta_{j,B}(\gamma_H,\gamma) tr(\tau_\lambda(\gamma^{-1}))}
{v(H_{\gamma_H})}  \ .$$
A comparison  with the previous stated formula 
$$ d(H,s,\eta,j(\gamma_H)) \cdot  \frac{\Delta_{j,B}(\gamma_H,\gamma)\tr(\tau_\lambda(\gamma^{-1}))}{d(G)\cdot v(G_\gamma)} $$
for the sum
 $\sum_{\gamma'\sim\gamma} \Delta_\infty(\gamma,\gamma')e(G_{\gamma'})
 O^G_{\gamma'}(f_{\infty,G})$,  the rigidifications obtained from $(j,B)$ and $v(G_\gamma)=v(H_{\gamma_H})$ for regular $\gamma$ give
our desired matching equality for the now unambiguously defined factor 
 \begin{equation*}
 c= d(H,s,\eta,j(\gamma_H))/d(G) \ .\qedhere
 \end{equation*}
\end{proof}

\begin{lemma}\label{ARMATCH}
Suppose $(H,1,\id)$ is the maximal endoscopic datum, where
$H$ is the quasi-split inner form of $G$.
If $H(\R)$ and $G(\R)$ both admit discrete series representations, 
the matching function $f_{\infty,G}^H$ on $H(\R)$ for $f_{\infty,G}$ on $G(\R)$ is 
$$f_{\infty,G}^H =(-1)^{q(G)+q(H)}f_{\infty,H}\ .$$
\end{lemma}

\begin{proof} Since $G$ and $H$ have the same $L$-group, for the maximal endoscopic datum $(H,1,id)$ we have
$\#\Phi(\varphi)=1$ and the character $\kappa: {\cal E}(I/\R) \to \C^\times$ is trivial. 
Hence $f_{\infty,H,\varphi}= f_{\infty,H}$ and $d(H,s,id,\gamma) = d(G)$. 
So the claim follows from lemma~\ref{TRANSERARCH}. 
\end{proof}

\section{Algebraic modular forms in general}\label{s:algebraic_mod_forms}

The notion of algebraic modular forms deals with automorphic forms on connected reductive groups $G$ over $\Q$ whose archimedean completion is anisotropic.
Hence the associated symmetric space for $G_\infty=G(\R)$ is discrete and the archimedean components $\pi_\infty$ of irreducible automorphic representations $\pi$ of $G(\A)$ are finite dimensional.
Hence the space of algebraic modular forms is a function space on this discrete space
with values in a finite dimensional irreducible representation of $G_\infty$, defined by its highest weight $\lambda$.
The action of the Hecke operators in $G(\A_\fin)$ on this space of functions are accessible to explicit computer calculations using generalizations of the concept of Brandt matrices as studied by Eichler et al.
Algebraic modular forms in the above more general setting were later considered and studied
by Gross~\cite{Gross_alg_mod_forms} and many others.

\begin{table}
\begin{center}
\begin{footnotesize}
\begin{tabular}{|c|c|c|c|c|}
\hline
Type                   & Name             & max. compact subgroup             & hol. & Cartan index $i$      \\\hline
$A_1^\R$               & $\Sl(2,\R)$      & $S^1$                             & *    & --                    \\\hline
$A_n^{\C,i}$           & $\SU_n(i,n-i)$   & $A_{n-i}\times A_{i-1}\times S^1$ for $i>0$ & all  & $i=0,1,\dots, [\frac{n+1}{2}]$\\
                       &                  & $A_1$ for $i=0$                   &      &
                       \\\hline
$B_n^{\R,i}$, $n\geq2$ & $\SO_{2n+1}(i,2n+1-i)^0$  & $D_k\times B_{n-k}$ for $i=2k$ even  & $i=2$& $0,\dots,n$      \\
                       &                  & $B_k\times D_{n-k}$ for $i=2k+1$ odd  & --   &                   \\\hline
$C_n^\R$, $n\geq3$     & $\Sp(2n,\R)$     & $A_{n-1}\times S^1$               & *    &                       \\\hline
$C_n^{\HH,i}$, $n\geq3$&$\SU_n(i,n-i,\HH)$& $C_i\times C_{n-i}$               & --   &$0,\dots [n/2]$        \\\hline
$D_n^{\R,i}$, $n\geq 4$& $\SO_{2n}(i,2n-i)^0$&$D_k\times D_{n-k}$ for $i=2k$  & $i=2$& $0,2,4,\dots, \leq n$ \\\hline
$D_n^\HH$, $n\geq 5$   & $\SU_n(\HH)$     & $A_{n-1}\times S^1$               & *    &                       \\\hline
$G_{2,(i)}$            &                  & $G_2$                             & --   & $-14$                 \\
                       &                  & $A_1\times A_1$                   & --   & $2$                   \\\hline
$F_{4,(i)}$            &                  & $F_4$                             & --   & $-52$                 \\
                       &                  & $B_4$                             & --   & $-20$                 \\
                       &                  & $C_{3}\times A_1$                 & --   & $4$                   \\\hline
$E_{6,(i)}$            &                  & $E_6$                             & --   & $-78$                 \\
                       &                  & $D_5\times S^1$                   & *    & $-14$                 \\
                       &                  & $A_5\times A_1$                   & --   & $2$                   \\\hline
$E_{7,(i)}$            &                  &                                   & --   & $-133$                \\
                       &                  &                                   & *    & $-25$                 \\
                       &                  &                                   & --   & $-5$                  \\
                       &                  &                                   & --   & $7$                   \\\hline
$E_{8,(i)}$            &                  & $E_8$                             & --   & $-248$                \\
                       &                  & $E_7\times A_1$                   & --   & $-24$                 \\
                       &                  & $D_8$                             & --   & $8$                   \\\hline
\end{tabular}
\end{footnotesize}
\end{center}
\begin{small}
\caption{Real simple Lie groups with discrete series. Thm.~\ref{innerlifting} applies if the derived group of $G_\infty$ is a product of finitely many of these groups.
These groups are the inner forms of the compact inner form of minimal Cartan index. Notice, the quasi-split inner form $H_\infty$ of maximal Cartan index need not be split over $\R$.
In the fourth column we also indicated the indices $i$ where the corresponding symmetric space has 
a holomorphic structure. For a list of the simple case without discrete series see table~\ref{lasttable}. \label{Discreteseriesreps}}
\end{small}
\end{table}

\begin{table}
\begin{small}
\begin{center}
\begin{tabular}{|c|c|c|}\hline
  Type                    & Name               & Cartan index $i$ \\\hline
  $A_n^\R$, $n\geq2$      & $\Sl(n+1,\R)$      &         \\\hline
  $A_n^\HH$, $n\geq3$ odd & $\Sl(n+1,\HH)$     &         \\\hline
  $D_n^{\R,i}$, $n\geq4$  & $\SO_{2n}(i,2n-i)$ & $i$ odd \\\hline
  $E_6(i)$                &                    & $-26$   \\
                          &                    & $6$     \\\hline
  \end{tabular}
  \end{center}
  \end{small}
\caption{ Real simple Lie groups that do not allow discrete series. Further recall that complex simple groups, considered as real Lie groups, never have discrete series.\label{lasttable}}

\end{table}

\bigskip\noindent
In this section we show how algebraic modular forms can be weakly lifted to
inner forms of $G$, in particular to the quasi-split inner form $H$ of $G$.
We describe these liftings in terms of irreducible automorphic representations $\pi$
of $G(\A)$. In the setting of algebraic modular forms all irreducible automorphic representations $\pi$ of $G(A)$ are cuspidal.  
For endoscopic representations $\pi$ existence of the lifting follows from the theory of endoscopy. For CAP-representations $\pi$ on the other hand, $\pi$ by definition has an $H$-Eisenstein lift on the quasi-split inner form $H$. However notice, although $\pi$ is cuspidal, there are CAP-cases were there only exist noncuspidal liftings $\tilde\pi$ of $\pi$
for $H(\A)$. For example, the one-dimensional representations $\pi$ of $G(\A)$ only
lift to one-dimensional representations $\tilde\pi$ of $H(\A)$, the latter being noncuspidal
unless the $k$-rank is zero. For further information on this see \cite{Gan_SK_inner}, section (6.10)-(6.12). This being said, to construct the lifting we therefore may restrict ourselves to the case of algebraic modular forms where the automorphic representations $\pi$ of $G(\A)$ are of general type (in the sense of section \ref{section1}).

\medskip
For the proof we use the cohomological trace formula using the underlying fact
that for inner forms $H$ of $G$ over $\GlobalField$ the archimedean group $H_{\infty}$
is an inner form of the anisotropic real group $G_\infty$. Since a reductive real Lie group  admits discrete series representations if and only if it is an inner form of an anisotropic reductive group over $\R$, hence any inner form of $G_\infty$, hence in particular $H_\infty$, admits discrete series representations. This fact allows to apply the cohomological trace formula. 

\begin{thm}\label{innerlifting} Suppose $H$ is a quasi-split group over
$\GlobalField$ such that at the real archimedean places the local group $H(\R)$
admits discrete series representations.
Fix an inner form $G$ of $H$ with the property that $G(\R)$ is anisotropic, i.e. compact modulo center.
Let $\pi$ be an irreducible automorphic representation of $G(\mathbb A)$, which is
neither weakly $H$-Eisenstein nor weakly endoscopic.
Let $S$ be a finite set of non-archimedean places containing all places
where $G_v$ and $H_v$ are not isomorphic and those where $\pi_v$ is not spherical.
Then there exists an irreducible automorphic representation $\pi^H$ of $H(\A)$ such that
\begin{enumerate}
\item$\pi^H_v \cong \pi_v$ holds for all non-archimedean places $v\notin S$,
\item $\pi_\infty^H$ belongs the the $L$-packet of discrete series of $H(\R)$
uniquely determined by the complexification of
the finite dimensional representation $\pi_\infty$ of $G(\R)$.
\end{enumerate}
\end{thm}
We prove this result as a consequence of more general lifting theorems in this section. In fact, the last theorem is an immediate consequence of corollary~\ref{tracerelations}, which in turn follows from cor.~\ref{LIFTS}, since for anisotropic $G(\R)$ the completion of the symmetric space $S_K(G)$ associated to $G$ is is discrete, hence its cohomology is concentrated in degree zero. For anisotropic $G(\R)$ thus $m_{virt}$ counts the true multiplicity of $\pi_\infty(\lambda)\pi_\fin$, thus is positive for (cuspidal) automorphic $\pi$.

\medskip
For the convenience of the reader, we listed in table~\ref{Discreteseriesreps} the simple real Lie groups $H$ which admit discrete series representations. For the remaining simple real Lie groups see~table~\ref{lasttable}.

\bigskip\noindent
\textbf{Remark}. We now return to the more general setting of $\S$~\ref{s:stab_trace}
and allow $G$ to be an arbitrary inner form of a quasi-split group $H$ over $k$.
We only assume that $H(\R)$ admits discrete series representation, besides the
innocent simplifying assumptions 1. and 2. of $\S 1$. 
We specialize to the setting above, where $G$ is anisotropic at the archimedean place,
only at the end of $\S 7$ for corollary~\ref{tracerelations}.

\medskip
This being said,
fix an unramified cuspidal representation $\pi^S$ of $G(\A^S)$ as before, in particular assumed to be not weak endoscopic nor CAP. Let $H$ denote the quasi-split inner form
of $G$. Then combining corollary~\ref{STF}, proposition~\ref{5.1} and theorem~\ref{STABI} implies the following proposition.

\begin{prop} \label{stable} Let the situation be as in corollary \ref{STF}.
Fix a test function $$f= (d(G)f_{\infty,G})\cdot f_S f^S(\pi^S)f^{trunc}_{v_0} \in C^\infty(G(\A))$$
where $f_S\in C^\infty(G(\A_S))$ has regular support and $f^S(\pi^S)\in C(G(\A^{S\cup\{v_0,\infty\}}))$ and matching functions $f^H$ on $H(\A)$ for every endoscopic triple $(H,s,\eta)$.
Then for sufficiently small good compact subgroups $K$ of $G(\A_\fin)$\,,
\begin{gather*}
 \tr_s(f_S, H^\bullet(S_K(G),V_\lambda)(\pi^S) = T_{\elliptic}^G(f_S f^S(\pi^S)f^{trunc}_{v_0}, \lambda) \\
 =\!\! \sum_{\gamma\in G(\Q)/\sim}\!\! \tau_{Tam}(G_\gamma)
O_\gamma^{G(\A),*}(f) =  \!\!\sum_{(H,s,\eta)}\!\! \iota(G,H)d(H,s,\eta) ST_{\elliptic}^{H(\mathbb A),**}(f^H)  \ 
\end{gather*}
 with the sum running over strongly elliptic semisimple $\gamma$ as in theorem~\ref{STABI}.
\end{prop}

\medskip\noindent

\begin{cor} \label{7.3}\label{cor:matching_aut_spectrum_G_H}
Fix an inner form $G$ over $k$ of the quasi-split form $H$. Suppose  $G(\R)$, and hence $H(\R)$, admit discrete series
representations. Fix a finite set $S$ of nonarchimedean places
such that for all nonarchimedean places $v\notin S$ the local groups
$H$ and $G$ are isomorphic over $\GlobalField_v$. Suppose that $\pi$ is a irreducible cuspidal 
automorphic representation of $G(\mathbb A)$ such that $\pi$ is neither $H$-Eisenstein
(hence not $G$-Eisenstein), nor a weak endoscopic lift. Consider $S$
such that $\Pi^S$ is an unramified representation as above.
Suppose $f_S$ has regular support and $f_S$ and $f_S^H$ have matching stable orbital integrals. Then 
$$ \sum_{\pi}\frac{m_{\coh}(\pi)}{d(G)}\pi(f)  = \frac{1}{d(G)} \cdot\tr_s(f_S, \lim_{K_G} H^\bullet(S_{K_G}(G),V_\lambda)(\pi^S)) $$ 
coincides with
$$ \sum_{\pi'}\frac{m_{\coh}(\pi')}{d(H)}\pi'(f^H)  = \frac{1}{d(H)} \cdot\tr_s(f^H_S, \lim_{K_H} H^\bullet(S_{K_H}(H),\tilde V_\lambda)(\pi^S))\ .$$
The direct limits run over small good compact subgroups $K_G\subseteq G(\A_{\fin})$ and $K_H\subseteq H(\A_{\fin})$. 
\end{cor}

\bigskip\noindent
\begin{proof}
To prove corollary \ref{cor:matching_aut_spectrum_G_H},  we can choose the auxiliary $f_{v_0}^{trunc}$ in $C_c^\infty(G_{v_0})= C_c^\infty(H_{v_0})$
simultaneously adapted to the supports of $f_S$ and $f^H_S$ and the $\pi^S$ projectors for both groups $G(\mathbb A_{\fin})$ and $H(\mathbb A_{\fin})$ and use proposition
\ref{stable}:

\bigskip\noindent
For any inner form $G$ of a quasi-split form $H$ over $k$,
the set of stable semisimple conjugacy classes ${\cal S}_{G}$
of $G(\GlobalField)$ is a subset  of the set ${\cal S}_G$ of stable semisimple conjugacy classes
of $H(\GlobalField)$ by \cite{Kottwitz_5}, \S 6 and thm.~4.4.
This allows to apply corollary~\ref{STF} for both sides.
Suppose $\gamma'$ is in a stable conjugacy class of $H(\GlobalField)$ and comes from some $\gamma'$ in a stable conjugacy class of $G(\GlobalField)$.
If $\gamma$ is regular, the centralizers are maximal tori and hence independent and their Tamagawa numbers  coincide by trivial reason; in general this follows from \cite{Kottwitz_3}, thm.~3 and \cite{Kottwitz_2}, formula (5.1.1). 
One can choose common $\pi_S$-projectors  and a common saturated set $T$.
This allows to construct a common auxiliary place $v_0\notin S$ and common truncation functions $\tilde f_{v_0}=f_{v_0}$, since its construction only depends on the supports of $f_v, \tilde f_v$ at those places where $f_v$ resp. $\tilde f_v$ locally is not $1_{\Omega_v}$.

\bigskip\noindent
Stabilizing the trace formula as in prop.~\ref{stable}, our assumption on $\pi^S$ eliminates all summands except those from the trivial endoscopic triple $(H,1,\id)$ for the quasi-split inner form $H$ of $G$.
If we compare the terms for this remaining sum, we finally use that the summands
$ ST_{\elliptic}^{H(\mathbb A),**}(f^H)$ for $f=d(G)f_{G,\infty} f_{\fin}$ on $G(\A)$ in the formula of  prop.~\ref{stable} are the product of $SO^{H(\R)}(d(G) f_{\infty,G}^H)$ and the nonarchimedean stable orbital integral
$SO^{H(\A_\fin)}(f_{\fin}^H)$.
The latter coincides with the nonarchimedean stable orbital integral arising from $d(H) f_{H,\infty} \cdot f_{\fin}^H$ in the formula  of  prop.~\ref{stable}, when applied to $H$ instead of $G$.
By lemma~\ref{ORB} we have $SO^{H(\R)}( f_{\infty,H})= SO^{G(\R)}( f_{\infty,G})$.
From lemma~\ref{ARMATCH}, by comparison of the right hand sides of the formula of prop.~\ref{stable} for both $G$ and its quasi-split inner form, the claim follows.
\end{proof}

\begin{cor} \label{LIFTS} Let the situation be as in corollary \ref{cor:matching_aut_spectrum_G_H}.
Suppose that $\pi$ is an irreducible cuspidal automorphic representation of $G(\mathbb A)$ such that the virtual multiplicity of $\pi_\fin$ is non-zero, i.e.
$$ m_{\virt}(\pi_\fin)\ =\sum_{\pi'=\pi'_\infty\pi_{\fin}}  m_{\cusp}(\pi') \chi(\pi'_\infty,\lambda) \ \neq \ 0 \ .$$
Counted with their multiplicity $m_{\cusp}(\pi')$ in the cuspidal spectrum  the sum ranges over all irreducible cuspidal automorphic representations $\pi'$ of $G(\mathbb A)$ for which $\pi'_{\fin}\cong
\pi_{\fin}$ holds.
Then, under the assumptions of corollary~\ref{cor:matching_aut_spectrum_G_H} on $\pi$, there exists an automorphic
representation $\tilde \pi$ of $H(\mathbb A)$ such that $\pi^S \cong \tilde \pi^S$ holds.
\end{cor}

\noindent
\begin{conjecture}\label{conj:middle_coh} Let $S$ be a finite subset of the nonarchimedean
places of $k$.
If $G(\R)$ admits discrete series, suppose $\pi^S$ comes from some cuspidal representation of $G(\A)$ that is not weakly $G$-Eisenstein.
Then for any automorphic cuspidal $\pi\cong \pi_\infty\pi_S\pi^S$ the cohomology group
$$H^i(S_K(G),V_\lambda)(\pi^S) =0$$ vanishes for $i\neq q(G)$ and necessarily $\pi_\infty$ belongs to the discrete series.
\end{conjecture}

\bigskip\noindent
If this conjecture holds for $S=\emptyset$, the assumption $m_{\virt}(\pi_\fin)
\neq 0$ of corollary~\ref{LIFTS} would always be satisfied.

\bigskip\noindent
\textit{Algebraic modular form case}. Now let us specialize to the case
where $G(\R)$ is an anisotropic group. In other words, we now return to the discussion of
algebraic modular forms and the proof of theorem~\ref{innerlifting}.
Let $H$ be quasi-split over $\GlobalField$ with inner $\GlobalField$-form $G$.
Then $G(R)$ and hence $H(\R)$ admit discrete series representations.
Furthermore $\dim(S_K(G))=0$ implies $d(G)=1$ and $H^\bullet(S_K(G),V_\lambda)=H^0(S_K(G),V_\lambda)$ and $\pi_\infty$ is the $K$-type described by $\lambda$.
Hence $ m_{\virt}(\pi_\fin)=
1$ holds in the situation of corollary~\ref{LIFTS}.
So corollary \ref{LIFTS} gives the next corollary, which immediately implies theorem~\ref{innerlifting}.

\begin{cor} \label{tracerelations} 
If $G(\R)$ is anisotropic, then
 in the situation of corollary~\ref{7.3} (so $\pi^S$ is neither weakly $H$-Eisenstein nor are weak proper endoscopic lift) we obtain
\begin{equation*}
\tr_s(f_S, \varinjlim_{K_G} H^0(S_{K_G}(G),V_\lambda)(\pi^S)) =
 \frac1{d(H)} \tr_s(f^H_S, \varinjlim_{K_H} H^\bullet(S_{K_H}(H),\tilde V_\lambda)(\pi^S))\ .
\end{equation*} 
\end{cor}

Notice, the automorphic representations $\pi$ that contribute to the left side run over all irreducible automorphic representations  
of $G(\mathbb A)$ (counted with their cuspidal multiplicity) that at  nonarchimedean places of $k$ outside $S$ are isomorphic to $\pi^S$.

\medskip
We remark that the same arguments give cohomological weak lifting theorems for inner forms $G_1$ and $G_2$ over $k$, under the following conditions: $G_1(\R)$, $G_2(\R)$ admit discrete series, for both groups conjecture \ref{conj:middle_coh} holds and finally, the set of stable 
semisimple conjugacy classes ${\cal S}_{G_1}$
of $G_1(\GlobalField)$ is a subset  of the set ${\cal S}_{G_2}$ of stable semisimple conjugacy classes
of $G_2(\GlobalField)$, if considered as subsets of the set of stable 
semisimple conjugacy classes ${\cal S}_{H}$ for the quasi-split inner form $H$ of $G_1$
and $G_2$.

\section{Inner forms of \texorpdfstring{$\GSp(4)$}{GSp(4)} and cohomology}\label{s:inner_forms_coh}

We now prove certain cases of conjecture~\ref{conj:middle_coh}
for $\Q$-forms $G$ of the $\Q$-split group $H=\GSp(4)$.
Up to isomorphism over $\Q$, its inner forms $G$ are $\GU_D(1,1)$ and $\GU_D(2)$ for quaternion algebras $D/\Q$.
These groups are defined by
$gJ\overline{g}^t=\lambda(g)J$
for $g\in \Gl(2,D)$ and $\lambda(g) \in \mathbb{G}_m$ where we set $J=\left(\begin{smallmatrix}0&1\\1&0\end{smallmatrix}\right)$ for $\GU_D(1,1)$ and $J=I_2$ for $\GU_D(2)$.
They become isomorphic to $\GSp(4)$ whenever $D$ splits. It is well-known that these are all the $\Q$-forms of $\GSp(4)$, see lemma~\ref{lem:inner_forms}.
The derived groups of the real Lie groups $\GSp(4,\R)$, $\GU_D(1,1)(\R)$ and $\GU_D(2)(\R)$ correspond to the cases $B_2^{\R,i}$ for $i=2,1,0$ of table~\ref{Discreteseriesreps}.
\begin{lemma}\label{lem:inner_forms}
Every $\Q$-form $G$ of $H=\GSp(4)$ is an inner form and is isomorphic to $\GU_D(1,1)$ or to $\GU_D(2)$ for some quaternion algebra $D$ over $\Q$.
The only equivalence between these inner forms is $\GU_D(1,1)\cong \GU_D(2)$ if $D$ splits at $\infty$.
\end{lemma}
\begin{proof}
The Dynkin diagram of $\GSp(4)$ does not admit non-trivial symmetries, so every form is an inner form.
By the isomorphism $\mathrm{Aut}(\GSp(4))\cong \PGSp(4)\cong \SO(2,3)$, inner forms are in one-to-one correspondence with $H^1(F,\SO(2,3))$. The elements of this Galois cohomology correspond bijectively to equivalence classes of quadratic forms in five variables with fixed discriminant~\cite[prop.~2.8]{Platonov_Rapinchuk}. Over local non-archimedean fields $F$, they are classified by their Hasse invariant, over $F\cong \R$ by their signature and over $F\cong\C$ they are all equivalent. The three signatures are $(5)$, $(4,1)$ and $(3,2)$, corresponding to $\GU_D(2)$, to $\GU_D(1,1)$ and to the split form $\GSp(4)$.
Note that $\GU_D(1,1)$ is isomorphic to $\GU_D(2)$ over local non-archimedean fields.
By the local-global principle, quadratic forms over $\Q$ are uniquely determined by their localizations at every place. Every combination of local quadratic forms is realized by a global one provided the  Hasse invariants are nontrivial at an even finite number of places. The corresponding quaternion algebra $D$ is uniquely determined by the condition that it ramifies at these places.
\end{proof}

\bigskip\noindent
For the anisotropic group $G(\R)=\GU_D(2)(\R)$ where $D$ ramifies at
$\infty$, conjecture~\ref{conj:middle_coh}
holds.
Indeed, $G_\infty$ is connected and the symmetric space $X_{G}$, and hence $S_K(G)$, is a discrete set.
So its cohomology is concentrated in degree zero and
only $\pi_\infty =\tau_\lambda$ contributes to $H^0(\mathfrak{g},K_\infty, V_\lambda)$ with a one-dimensional space.
Thus the Euler characteristic is $\chi(\pi_\infty,\lambda)=1$.
Note that there are no CAP representations since $G$ does not have a proper $\Q$-parabolic subgroup.

\bigskip\noindent
For $G=\GSp(4)$ recall that $q(G)=3$ and that attached to a parameter $\lambda$ there are two non-isomorphic discrete series representations $\pi_+(\lambda)$ and $\pi_-(\lambda)$.
The first one is generic, the second one is not and gives the (anti)holomorphic discrete series representation attached to $\lambda$. Each of these two discrete series representations $\pi_{\pm}(\lambda)$ contributes a two dimensional space to the middle cohomology in degree $3$. We obtain a special case of conjecture~\ref{conj:middle_coh}:

\begin{prop}\label{prop:middle_coh_GSp}
Suppose $G=\GSp(4)$.
If an irreducible cohomological cuspidal automorphic representation $\pi=\pi_\infty\pi_{\fin}$  
of $\GSp(4,\A)$ is not a CAP representation,
then all cohomology groups $H^{i}(S_K(G),V_\lambda)(\pi_{\fin})$ vanish in degree $i\neq 3$.
\end{prop}

\begin{proof}
We have to discuss the cuspidal cohomology that contributes to cohomological
degrees different from the middle degree and we have to show
that such cohomology classes all come from CAP representations.
For this we fix an irreducible cuspidal automorphic
representation $\pi$ of $G(\mathbb A)$ that is not CAP. As usual we write $\pi=\pi_\infty \otimes
\pi_{\fin}$.
By well known  vanishing theorems for unitary admissible representations \cite{Borel_Wallach} one knows that
$H^i(S_K(G),V_\lambda)(\pi_{\fin})$ vanishes for $i\neq 2,3,4$.
In our discussion we may, in addition, assume $$H^3(S_K(G),V_\lambda)(\pi_{\fin})= 0$$
since the only contributions to $H^3(S_K(G),V_\lambda)$ come from irreducible automorphic representations $\pi=\pi_\infty\pi_\fin$ with $\pi_\infty$ in the discrete series.
If $\pi_\infty$ is in the discrete series,
then $H^i(S_K(G),V_\lambda)(\pi_{\fin})\neq 0$ for $i\neq 3$ implies that $\Pi$ is CAP.
For further details see \cite{Weissauer_Endo_GSp4}, in particular thm.1.1(2) and 
cor.1.2.
It remains to discuss the case where $\pi_\infty$ is not discrete series.
By \cite{Weissauer_Endo_GSp4}, section 1.5 the Lefschetz maps are isomorphisms
$$L^i: H^{3-i}(S_K(G),V_\lambda)(\pi_{\fin}) \cong H^{3+i}(S_K(G),V_\lambda)(\pi_{\fin})$$
 respecting the generalized $\pi_{\fin}$-eigenspaces since  $L^i$ commutes with Hecke operators.
The only relevant case to consider arises where $\rho_\ell :=  H^{2}(S_K(G),V_\lambda)(\pi_{\fin})$ does not vanish.
Then $H^{4}(S_K(G),V_\lambda)(\pi_{\fin})$ is isomorphic to $\rho_{\ell}$ as a vector space and,
up to a Tate twist, also as a Galois module of the absolute Galois of $\Q$.
Hence we know that
$$ \lim_K H^{\bullet}(S_K(G),V_\lambda)(\pi_{\fin}) = \rho_\ell \oplus \rho_\ell(-1) \ .$$
Let us first assume that $\Pi$ is not a weak endoscopic lift.
Under this assumption, for almost all primes $p$ we furthermore obtain
the following identity
\begin{equation}\label{eq:8_**}
\det(1-p^{-3/2}t \cdot \Frob_p,\, \rho_\ell \oplus \rho_\ell(-1))^{-4} = P_p(\pi,t)^{m_{coh}(\pi_{\fin})} \tag{$\ast\ast$}
\end{equation}
where $P_p(\pi,t) = (1-\alpha_{p,0}t)(1-\alpha_{p,0}\alpha_{p,1}t)
(1-\alpha_{p,0}\alpha_{p,2}t)(1-\alpha_{p,0}\alpha_{p,1}\alpha_{p,2}t)$
defines the local Euler factor $L_p(\pi,s)$ of the Piateski-Shapiro spinor L-series $L(\pi,s)=\prod_v L_v(\pi,s)$. The proof of this last identity \eqref{eq:8_**} is essentially the same as the proof of \cite[p.72(**)]{Weissauer_Endo_GSp4}, which is obtained from formula \cite[p.72(*)]{Weissauer_Endo_GSp4} by logarithmic differentiation. The latter in turn is deduced
from a comparison of the Grothendieck-Lefschetz trace formula for the Frobenius
$\Frob_p$ and the stable topological trace formula, see \cite[cor.~3.4]{Weissauer_Endo_GSp4}.
Our present situation is different to the situation there in the fact that in formula \cite[p.72(*)]{Weissauer_Endo_GSp4} the term
$4\cdot\tr_s(\Frob_p^n, H_{!,et}^3(S_K(\overline \GlobalField_p), E_\lambda)(\pi^p))$
has to be replaced by
$$4(-1)^3 \cdot\tr_s(\Frob_p^n, H_{!,et}^\bullet(S_K(\overline \GlobalField_p), E_\lambda)(\pi^p))\ .$$
Since $\tr_s$ is a supertrace and the cohomology now is concentrated in the degrees $i=2,4$, we get an additional minus sign in the exponent ${-4}$ on the left side.
But this immediately yields a contradiction since $P_p(\pi,t)$ is a polynomial in $t$, whereas
$\det(1-p^{-3/2}\Frob_p t, \rho_\ell \oplus \rho_\ell(-1))^{-4}$ 
is a power series in $t$. Obviously this power series is a polynomial in $t$ if and only $\rho=0$. Since by assumption $\pi$ is not CAP, $\rho\neq 0$ therefore implies that $\pi$ has to be a weak endoscopic lift. 

Suppose $\pi$ is a weak endoscopic lift from a cuspidal automorphic
representation of the elliptic endoscopic group $M$ of $G$.
All weak endoscopic liftings were classified in \cite[\S5.2]{Weissauer_Endo_GSp4} in terms of theta liftings.
Since $M$ is a reductive quotient group of $\Gl(2)\times \Gl(2)$,
the weak Ramanujan bounds for
cuspidal automorphic forms on $Gl(2)$ immediately give
an estimate for the size of the roots of $P_p(\pi,t)$.
For an endoscopic lift, $P_p(\pi,t)=P_{p,1}(\sigma_1,t)P_{p,2}(\sigma_2,t)$
is the product of two quadratic polynomials for almost all primes $p$
where $P_{p,i}(\sigma_i,t)$ for $i=1,2$
describe the local $L$-factors of two cuspidal automorphic forms $\sigma_i$ on $Gl(2,\A)$. 
The weak Ramanujan bounds imply that all roots $\alpha$ of $P_{p,i}(\sigma_i,t)$
satisfy $\vert \alpha\vert < p^{1/2}$ for almost all $p$.
On the other hand, using \cite{Weissauer_Endo_GSp4} section 1.4, the Weil conjectures imply that the Frobenius $\Frob_p$ 
has eigenvalues on the etale  cohomology
groups $H_{!,et}^i(S_K(\overline F_p),E_\lambda)(\pi^p))$ for $i=2,4$ 
that are pure of weight 2 resp. 4. Since
for almost all $p$ the eigenvalues of $p^{-3/2}\Frob_p$ define roots
of $P_p(\pi,t)$ by \cite{Chai_Faltings}, the weak estimate from above therefore immediately gives a contradiction.
We leave the details of the rather routine calculation to the reader.

Our contradiction implies that irreducible cohomological cuspidal 
automorphic representations $\pi$ of $G(\A)$ are CAP unless all cohomology groups
$H^{i}(S_K(G),V_\lambda)(\pi_{\fin})$ vanish in degree $i\neq 3$.
\end{proof}
  
\bigskip\noindent

Now let us consider the case $G = \GU_D(1,1)$ where $D$ is a quaternion algebra over $\Q$, which is definite over $\R$.
For the decomposition of the cohomology spaces $$H^\bullet(S_K(G),V_\lambda) = H_{\mathrm{Eis}}^\bullet(S_K(G),V_\lambda) \oplus H_{\cusp}^\bullet(S_K(G),V_\lambda)\ ,$$
the Eisenstein cohomology $H_{Eis}^\bullet(S_K(G),V_\lambda)$ of $S_K(G)$ was computed by Grobner \cite[thm.~6.1]{Grobner}, with the slight exception that he considers $\SU_D(1,1)$ instead of $\GU_D(1,1)$.
The cuspidal cohomology decomposes
\begin{equation*}H_{\cusp}^\bullet(S_K(G),V_\lambda) 
= \bigoplus_{\pi=\pi_\infty\pi_{\fin}} m_{\cusp}(\pi) \cdot H^\bullet(\gfrak ,K_\infty,\lambda)  \otimes \pi_{\fin}^K\ ,
\end{equation*}
where the sum is over the finitely many cuspidal automorphic representations 
$\pi=\pi_\infty \pi_{\fin}$ of $G(\A)$ with central character $\omega$
and $H^\bullet(\gfrak ,K_\infty,\lambda) \neq 0$.

\bigskip\noindent
As explained in \cite{Grobner}, thm.~7.1, for fixed $\lambda$ one only
has the following cuspidal cohomological representations $\pi = \pi_\infty \pi_{\fin}$ of $\SU_D(1,1)$: 
\begin{enumerate}
\item Representations $\pi_\infty \cong A^+(\lambda), A^-(\lambda)$ in the discrete series of $\SU_D(1,1)$
that exclusively contribute to the cohomology in degree $2$ with a one-dimensional cohomology space.
\item If $\lambda$ is a multiple of the weight of the 5-dimensional standard
representation of $so(5,\C)\cong \gfrak \otimes_{\R} \C$, in addition, 
there are non-tempered representations $\pi_\infty \cong A_1(\lambda)$ 
which contribute a one dimensional space to the cohomology in degree $1$ and $3$ and zero in the other degrees.
\end{enumerate}

\noindent
The Lie group $G_{nc}=\GU_D(1,1)(\R)$ is not connected, and
there exist elements $g\in G_{nc}$ such that the representation
$ A^+(\lambda)(ghg^{-1})$ and  $A^-(\lambda)(h)$ of $SU(1,1)$ are isomorphic.
This implies $d(G_{nc})=1$. So there is a unique discrete series representation $A_2(\lambda)$
of $G_{nc}$ that after restriction to $SU_D(1,1)(\R)$ decomposes into $A^+(\lambda)\oplus  A^-(\lambda)$.
Hence $A_2(\lambda)$ gives a two-dimensional contribution to the middle cohomology in degree 2.
There are two irreducible representations of $GU_D(1,1)$ with given central character $\omega$ that restrict to $A_1(\lambda)$ on $\SU_D(1,1)$. We denote one of them by $A_1(\lambda)$ again, then the other is its twist by the unique non-trivial quadratic character $\mathrm{sign}$ of $\R^\times$.
This immediately implies:

\begin{prop}\label{prop:middle_coh_GSp11}
Fix a parameter $\lambda=(\lambda_1,\lambda_2)$ and an irreducible smooth representation $\pi_\fin$ of $G(\A_\fin)$.
If $\lambda_1\neq\lambda_2$, then for every cohomological cuspidal automorphic representation $\pi=\pi_\infty\otimes\pi_\fin$ of $G= \GU_D(1,1)$ defined over $\Q$, the archimedean factor is isomorphic to $A_2(\lambda)$. The attached cuspidal cohomology space $H_{\cusp}^q(S_K(G),V_\lambda)(\pi_\fin)$ has dimension two in degree $q=2$ and is zero for $q\neq2$. Hence $d(G)=1$ and
$\chi_\infty(A_2(\lambda),\lambda) = 2$, whereas
$\chi_\infty(A_1(\lambda),\lambda) = -2$.
\end{prop}

In other words, conjecture~\ref{conj:middle_coh} holds for regular parameters $\lambda$.

\section{Inner forms of \texorpdfstring{$\GSp(4)$}{GSp(4)} and comparison of trace formulas}\label{Vergleich}

Fix the non-quasi-split inner form $G=\GU_D(1,1)$ of $H=\GSp(4)$ attached to a  quaternion division algebra $D$ over $\GlobalField = \Q$ (with the notation as at the beginning of $\S 8$).
The division algebra $D$ (and hence $G$) may be split at $\infty$ or not. According to Arthur's conjectural classification, the cuspidal automorphic representations for the quasisplit inner form $H$ with holomorphic discrete series at $\infty$ then fall in three classes: Saito-Kurokawa lifts, weakly endoscopic lifts and representations of general type.
It is expected that this classification carries over to the inner forms $G$ with a few modifications~\cite[3.7]{Gan_SK_inner}.
Weakly endoscopic automorphic representations of $G$ and $H$ and their local $L$-packets are well-understood, see Chan and Gan~\cite[\S3]{Chan_Gan_GSp4_III}, for Saito-Kurokawa lifts see Gan~\cite{Gan_SK_inner}.

In this section we consider a cohomological cuspidal automorphic representation of $G$, neither CAP nor weakly endoscopic, and show there is an $L$-packet of cohomological cuspidal automorphic representations $\pi'$ of $H$,
that are weakly equivalent to $\pi$ in the sense that $\pi^S\cong \pi'^S$ holds.
To be precise, fix the following data:
\begin{enumerate}
\item For every archimedean place, a local $L$-packet of discrete series representations with a parameter $\lambda=(\lambda_1,\lambda_2)$ depending on the local parameters  of each archimedean place.
\item A finite set $S$ of non-archimedean places such that for non-archimedean $v\notin S$ the local algebra $D_v$ splits, so $G_v\cong H_v$.
\item A smooth representation $\pi^S$ of $G(\A_\fin^S)\cong H(\A_\fin^S)=\prod_{v\notin S\cup\{\infty\}}H_v\,$, that is a factor of a cohomological cuspidal automorphic representation of $G(\A)$ which is neither CAP nor endoscopic.
\end{enumerate}

\bigskip\noindent
We construct packets of cohomological cuspidal automorphic representations $\pi'$ that contain $\pi^S$ as a local factor, determine their cuspidal multiplicity and explicitly describe the local factors $\pi'_v$ at the bad places $v\in S$.
These $\pi'$ belong to Arthur packets of \emph{general type} in the sense of Arthur's classification \cite{Arthur_classification_04}. This classification has not been fully proven, although important progress has been made by Gee and Ta\"ibi~\cite{Gee_Taibi} and others. However, their proof is this still conditional on other results announced by Arthur~\cite{Arthur13_Book}. We thus attempt to avoid to use Arthur's classification.
Recall that, adopting Arthur's nomenclature, we call a cuspidal automorphic representation $\pi'$ of $\GSp(4)$ to be \emph{of general type} if it is neither CAP nor weakly endoscopic.

\bigskip\noindent
With the notion of virtual multiplicity defined at the end of $\S $~\ref{s:Lefschetz}
\begin{equation*}
m_{G,\virt}(\pi_\fin) =\!\! \sum_{\pi=\pi_\infty\pi_{\fin}}\!\! m_{G,\mathrm{coh}}(\pi)\qquad\text{for}\qquad
m_{G,\mathrm{coh}}(\pi)
= \chi(\pi_\infty,\lambda)m_G(\pi)\ ,
\end{equation*}
the trace formula implies the following matching relation:
\begin{lemma}\label{MaTch}\label{eq:global_character_matching}
Fix $S$, $\pi^S$ and $\lambda$ as above.
For every pair of matching functions $(f_S,f_S^H)$ %
there is a cohomological matching relation
\begin{equation*}
\frac{1}{d(G)}\sum_{\pi_S} m_{G,\virt}(\pi_S \pi^S)\tr(\pi_S)(f_S) = \frac{1}{d(H)} \sum_{\pi'_S} m_{H,\virt}(\pi'_{S}\pi^S)\tr(\pi'_S)(f_S^H) %
\end{equation*}
The sums runs over irreducible representations $\pi_S$ of $G(\A_S)$ and $\pi'_S$ of $H(\A_S)$.
The coefficients %
\begin{equation*}
c_G(\pi)=\frac{m_{G,\mathrm{coh}}(\pi)}{d(G)}\qquad\text{and}\qquad c_H(\pi')=\frac{m_{H,\mathrm{coh}}(\pi')}{d(H)}=m_H(\pi')
\end{equation*}
are integers. They are non-negative if conjecture~\ref{conj:middle_coh} holds.
\end{lemma}
\begin{proof}
The first assertion comes from the trace formula of corollary~\ref{cor:matching_aut_spectrum_G_H}.
The second part is an archimedean assertion. It is clear that $c_G$ is an integer for non-split $G_\infty$, since then $d(G)=1$.
For split $G_\infty$ and for $H$ it follows because the Euler characteristic is %
a multiple of $d(H)=2$, see \cite{W_Asterisque}.
The non-negativity needs conjecture~\ref{conj:middle_coh}, hence holds for $H=\GSp(4)$ and $G_\infty\cong\GSp(4,\R)$ by prop.~\ref{prop:middle_coh_GSp} and it holds for $G=\GU_D(1,1)$ if $D_\infty$ is non-split and $\lambda$ is regular by prop.~\ref{prop:middle_coh_GSp11}.
\end{proof}

\noindent
\textit{Remark:} The cohomology of $\pi'$ is concentrated in degree three, so the Euler characteristic is negative. However, there is an additional sign in the definition of the good projector that takes care of this.

\bigskip\noindent
We make this explicit for $G=\GU_D(1,1)$ and $G=\GU_D(2)$.

\begin{prop}Fix $S$, $\lambda$ and $\pi^S$ as above. Let $G$ be $\GU_D(1,1)$ for a quaternion algebra $D$ over $k$, then for every pair of matching functions $(f_S,f^H_S)$ there is an equality
\begin{gather*}
 \sum_{\pi'_\infty,\ \pi'_S} m_{H}(\pi'_\infty\pi'_S\pi^S)\tr(\pi'_S)(f_S^H) 
\!=\!  \sum_{\pi_\infty,\ \pi_S} m_{G}(\pi_\infty\pi_S\pi^S)\frac{\chi(\pi_\infty,\lambda)}{d(G)}\tr(\pi_S)(f_S) 
\end{gather*}
where $m_{H}(\pi')$ and $m_{G}(\pi)$ denote the multiplicity of $\pi'$ and $\pi$ in the cuspidal spectrum of $H$ and $G$, respectively. The sum runs over all irreducible cuspidal representations $\pi=\pi_\infty\pi_S\pi^S$ resp. $\pi'=\pi'_\infty\pi'_S\pi^{'S}$ with fixed
$\pi^S\cong \pi^{'S}$, where
$\pi_S$ and $\pi'_S$ are representations of $G(\A_S)$ and $H(\A_S)$ respectively, and all irreducible archimedean factors contribute to cohomology with parameter $\lambda\,.$
The factor ${\chi(\pi_\infty,\lambda)}/{d(G)}$ equals one if $D_\infty$ splits,
otherwise it is $\pm 2$ depending on whether $\pi_\infty \cong A_2(\lambda)$
or (up to a sign twist) $\pi_\infty \cong A_1(\lambda)$  holds.
Especially, for anisotropic $D_\infty$ and regular $\lambda$,
the factor ${\chi(\pi_\infty,\lambda)}/{d(G)}$ is two and
\begin{gather*}
 \sum_{\pi'_\infty,\ \pi'_S} m_{H}(\pi'_\infty\pi'_S\pi^S)\tr(\pi'_S)(f_S^H) =  2\cdot \sum_{\pi_\infty,\ \pi_S} m_{G}(\pi_\infty\pi_S\pi^S)
 \tr(\pi_S)(f_S) \ .
\end{gather*}
For $G=\GU_D(2)$ the corresponding statement holds with the factor ${\chi(\pi_\infty,\lambda)}/{d(G)}=1$, regardless if $D_\infty$ splits or not.
\end{prop}

\begin{proof}
We can assume that $D$ is not globally split over $k=\Q$, for otherwise we have $H\cong G$ and the assertion is tautological.
Note that $\chi(\pi'_\infty,\lambda)=2$ and $d(H)=2$, see~\cite{W_Asterisque} and prop.~\ref{prop:middle_coh_GSp}, so their quotient is one.
The same follows for the right hand side in the case of split $D_\infty$ since this is a local archimedean assertion. Thus  we may assume that the archimedean algebra $D_\infty$ is definite.
For $G=\GU_D(1,1)$, then ${\chi(\pi_\infty,\lambda)}/{d(G)} = \pm 2$ by proposition~\ref{prop:middle_coh_GSp11} and for the regular case $\lambda_1\neq \lambda_2$ this factor equals $+2$.
For $G=\GU_D(2)$, the factor ${\chi(\pi_\infty,\lambda)}/{d(G)}$ is one because both numerator and denominator equal one.
\end{proof}

\section{The cohomological spectrum of \texorpdfstring{$\GSp(4)$}{GSp(4)}}

The automorphic spectrum of $\GSp(4)$ has been described in the terms of Arthur's classification in \cite{Arthur_classification_04}.
This classification has not been fully proven, although important progress has been made by Gee and Ta\"ibi~\cite{Gee_Taibi} and others. 
These authors reduce the statements to results proved in Arthur~\cite{Arthur13_Book}.
The results of \cite{Arthur13_Book} depend on the fundamental lemma for twisted weighted endoscopy, see p.~135 in loc.~cit.
The relevant results on these fundamental lemmas were shown by M{\oe}glin and Waldspurger~\cite{Moeglin_Waldspurger_Stab_twist_trace_formula}.
However, as indicated on page~924 in loc. cit, their proof depends on a result announced by Chaudouard and Laumon~\cite{Choudouard-Laumon}, which is not yet available. In this sense, the classification of the automorphic spectrum seems to be still conditional.

\bigskip\noindent
Since we are only concerned with the cohomological spectrum, we think it may therefore be useful to give in that cohomological setting an independent proof that does not use announced,
but not yet available results on twisted weighted endoscopy.
For this see prop.~\ref{spectrum} below.

\bigskip\noindent
\textit{Remark}. Since endoscopic automorphic representations and also CAP representations
are fully understood for the group $H=\GSp(4)$, we will restrict ourselves to
describe the structure of the remaining class of irreducible cuspidal automorphic representations $\pi'$ of \emph{general type}. Notice, this notion agrees with one used in Arthur's classification \cite{Arthur_classification_04}.
In the following we thus focus on irreducible cuspidal automorphic representations
$\pi'$ for $\GSp(4)$ that are cohomological and of general type in the sense above.
This assumption means that $\pi'$ is neither CAP nor endoscopic
such that $\pi' =\pi'_\infty \otimes \pi'_{\fin}$ has its archimedean component 
$\pi'_\infty$ in the discrete series of $\GSp(4,\R)$. 

\begin{prop}\label{spectrum}
For a cuspidal irreducible representation $\pi'$ of $\GSp(4,\A)$ let $\Pi(\pi')$
denote the set of isomorphism classes
of automorphic representations of $\GSp(4)$ that are weakly equivalent to $\pi'$. 
For cuspidal cohomological $\pi'$ of general type
the set $\Pi(\pi')$ is a cartesian product of local $L$-packets such that the following holds:
\begin{enumerate}
\item Each representation in $\Pi(\pi')$ occurs with multiplicity one.
\item Each local $L$-packet has a unique generic representative.
\item The local $L$-packets have cardinality one or cardinality two.
\item The archimedean $L$-packet has cardinality two.
\item The sum over the characters of the local packets is a stable distribution.
\item The spherical local factors of $\pi'$ uniquely determine $\Pi(\pi')$.
\item The four-dimensional Galois representation attached to $\pi'$ as in \cite{W_Asterisque} uniquely determines $\Pi(\pi')$.
\end{enumerate}
\end{prop}
 
\begin{proof} %
The proof of assertion 6 and 7 is well known \cite{W_Asterisque}.
Assertion 3, 4 and 5 have already been shown in \cite{Weissauer_Endo_GSp4}.
For the proof of the remaining assertions 1 and 2 of the proposition it suffices to show 
that $\Pi(\pi')$ is the product of the local $L$-packets of the local representations
$\pi'_v$ of $\pi'$ up to some common multiplicity factor. This 
immediately implies assertion 2 using theorem 1 of  \cite{W_Schiermonnikoog},
which guarantees a globally generic member in $\Pi(\pi')$. This 
 generic member has strong multiplicity one by  \cite{Jiang_Soudry}.
 This shows that the common multiplicity factor must be one and
 therefore implies the stronger multiplicity one assertion 1 for all members
 of $\Pi(\pi')$, once we know that $\Pi(\pi')$ is a product of local $L$-packets.

\bigskip\noindent
It remains to be shown that $\Pi(\pi')$ is the product of the local $L$-packets of the local representations
$\pi'_v$ of $\pi'$ up to some common multiplicity factor $m$ for all $\pi$ in $\Pi(\pi')$.
We use Weselmann's twisted topological trace formula~\cite{Weselmann_trace_formula} as in \cite[thm.~4]{W_Schiermonnikoog} for $G_1=\GSp(4)$,
to construct a weak automorphic lift of $\pi'$ to an irreducible cuspidal automorphic representation of $\GSO(3,3)$.
Recall that the group $\GSO(3,3)$ is a quotient of $\Gl(4)\times \Gl(1)$.
First, the existence of the weak lift is an immediate consequence of thm.~4 (i) of loc. cit.
Second, since strong multiplicity one holds for the product  $\Gl(4)\times \Gl(1)$,
this automorphic lift $\Pi = \Pi(\pi')$ is uniquely determined by its spherical local representations.
By  thm.~4 (ii) of loc. cit. the local spherical representations of this lift $\Pi$ are uniquely determined
by the spherical local representations $\pi'_v$.
Also notice, all local nonarchimedean representations at $v$ for the global representations in $\Pi(\pi')$ 
are the spherical and isomorphic to $\pi'_v$ if $\pi'_v$ is spherical.
Finally  thm.~4 (ii) of loc.~cit.~gives the trace identity
$$T(f,\sigma,\GSO(3,3),\chi) = c\cdot T(f_1,\id,\GSp(4),\chi_1)$$
for all strongly matching test function
$(f,f_1)$ on the groups over the finite adeles that follows from Weselmann's twisted 
cohomological trace formula.
In the formula above, $c\neq 0$ is a universal constant not depending on $(f,f_1)$ and $\sigma$ is an involution on $\GSO(3,3)$,
so that $T(f,\sigma,\GSO(3,3),\chi)$ is the $\sigma$-twisted trace of $f$ on 
the automorphic cohomology of $\GSO(3,3)$.
This cohomology is considered as a superspace, hence $T(f,\sigma,\GSO(3,3),\chi)$
is the alternating sum of the twisted traces on the cohomology groups in the different cohomology degrees (i.e. the Lefschetz number). The cohomology is defined with respect to a  certain sheaf defined by $\chi$ as in loc.~cit., page 293.
On the other hand $T(f_1,\id,\GSp(4),\chi_1)$ is the ordinary Lefschetz number 
for the trivial automorphism $id$ that we study in this paper and $\chi_1$ in loc. cit.
essentially denotes the sheaf defined by our $\lambda$.
Since $\pi'$ is cohomological of general type and hence not CAP, the trace identity 
$$T(f,\sigma,\GSO(3,3),\chi) = c\cdot T(f_1,\id,\GSp(4),\chi_1)$$
already holds under the weaker assumption that $(f,f_1)$ are matching test function
over the finite adeles.
This follows by the use of good spherical projectors defined in the twisted case in a similar way as in section~\ref{s:Finiteness}; see also page 302--303 in loc. cit. We may now use matching good spherical projectors $(f^S, f_1^S)$ for $(\Pi^S,(\pi')^S)$ and matching $(f_S,f_{1,S})$ for the finite set $S$ of nonarchimedean
 places containing those where $\pi'$ is not spherical.
The trace comparisons for $f=f_Sf^S$ and $f_1=f_{1,S}f_1^S$ than imply
$$  \mathrm{const} \cdot \Pi_S(f_S) = c \cdot \sum_{\pi \in \Pi(\pi')} m(\pi) \pi(f_{1,S}) \ .$$
Notice that $\Pi$ may occur as constituent in several cohomology degrees with
certain multiplicities.
The integer $\mathrm{const}$ is the alternating sum of these multiplicities;
its value is not important, except that it is nonzero, which implicitly follows 
from the nonvanishing assertion of \cite{W_Schiermonnikoog}, thm.~4 (i).
By the local character identity of \cite{Chan_Gan_GSp4_III}, page 3 or prop.~9.1, 
for any $v\in S$ and matching test functions $(f_v,f_{1,v})$ we have 
the trace identity between the characters of the generic $L$-packets:
$$  \Pi_v(f_v)  =  \sum_{\pi_v} \pi_v(f_{1,v}) \ .$$
Here the sum over $\pi_v$ extends over all (one or two) members of the local
$L$-packet $\{ \pi_v^+ \}$ (singleton) resp. $\{ \pi_v^+ , \pi_v^-\}$ (cardinality two)
that match with $\Pi_v$ in the local Langlands correspondence
and contain the a generic local  representation $\pi_v^+$. 
If we plug this into the global formula, we obtain
$$ \mathrm{const} \cdot \prod_{v\in S} \Pi_v(f_v)\ = \ \mathrm{const}\cdot \prod_{v\in S} \sum_{\pi_v} \pi_v(f_{1,v}) \ .$$
The right hand side is equal to $c \cdot \sum_{\pi \in \Pi(\pi')} m(\pi) \pi(f_{1,S})$,
so linear independence of characters implies that $\Pi(\pi')$ is the product of the local $L$-packets of the local representations
$\pi'_v$ of $\pi'$ up to a common multiplicity factor $m=\mathrm{const}/c \neq 0$,
which of course must be an integer. 
\end{proof}

As already mentioned, the assertions of prop.~\ref{spectrum} have 
been formulated in greater generality by Mok~\cite[thm.~3.1]{Mok}.
Mok's proof depends on Arthur~\cite{Arthur_classification_04}, Gee and Ta\"{i}bi~\cite{Gee_Taibi}, which in turn depends on Arthur~\cite{Arthur13_Book}
and therefore on the weighted fundamental lemma for twisted endoscopy~\cite{Moeglin_Waldspurger_Stab_twist_trace_formula}, and hence 
depends on \cite{Choudouard-Laumon}.

\section{Inner liftings to \texorpdfstring{$\GSp(4)$}{GSp(4)}}\label{s:innerlift}

For an inner form $G$ of $H=\GSp(4)$ we now fix
a cuspidal cohomological automorphic representation $\pi$ of $G(\A)$.
Let $S$ be a finite set of places containing the archimedean place 
and all places where $G_v$ is not split and all places where $\pi_v$ is not spherical. 
Let $\Pi(\pi)$ denote the set of equivalence classes
of irreducible cuspidal automorphic representations
of $G(\A)$ that are isomorphic to $\pi_v$ for places $v\notin S$. 
By definition $\pi^S$ can be viewed as a representation of $\prod_{v\notin S} G_v$
since $G_v=H_v$ for $v\notin S$.  

\bigskip\noindent
\textit{The setting}. Our aim is now to construct
global liftings $\pi'$ of cohomological global cuspidal representation $\pi$
of $G(\A)$ such that $\pi'_v \cong \pi_v$ holds at almost all places $v$. 
For this we will always assume that $\pi^S$ is of general type, i.e. not CAP or weakly
endoscopic as a representation of $G(\A_{\fin}^S)$ since
Chan and Gan \cite{Chan_Gan_GSp4_III} have completely described these 
global representations of $G(\A)$ for the inner forms $G$ of $H=GSp(4)$. 
$\pi$ being of general type then of course also implies
that a corresponding lift $\pi'$, if it exists, is of general type.
Recall that, by abuse of notation, we also say $\pi$ is of \textit{general type}
instead of saying $\pi^S$ is not CAP
or weakly endoscopic as a representation of $G(\A_{\fin}^S)$.

\bigskip\noindent
The trace comparison of section~\ref{Vergleich} will now be reformulated, 
using the local character identities established in prop.~11.1 (i) of \cite{Chan_Gan_GSp4_III}.
These local character relations allow us to interpret the global trace identity from lemma~\ref{MaTch}
as a trace identity for characters on an open subset of the stable conjugacy classes of regular semisimple elements of $G(\A_\fin^S)$ that
will be briefly referred to as the subset of $G$-norms.

\bigskip\noindent
\textbf{Remarks}.
Concerning this, the following general comments on some of the more technical details of the following argument may be helpful.

\bigskip\noindent
1) \textit{Normalization of transfer factors}. The local trace identities of Chan and Gan depend on the
choice of the transfer factors~\cite[\S4]{Chan_Gan_GSp4_III}. The local transfer factors involve a minus sign at all places
where $G$ is not split.
If we change the local transfer factors by a minus sign at all these places (including the archimedean place
if $G_\infty$ is not split),
the global property for the new local transfer factors still holds because these changes are done at an even number of places.
Let us assume this, because for our purposes this will be more convenient for the character identities.

\bigskip\noindent
2) \textit{Local character transfer}. Using the conventions of 1) on the local trace identities for nonarchimedean places $v\in S$ of Chan and Gan~\cite[\S11.1(i)]{Chan_Gan_GSp4_III} for fixed $S$, $\pi^S$ and $\lambda$ then become
\begin{equation*}
\sum_{\pi_v} \chi_{\pi_v}(f^G_v) = \sum_{\pi_v'} \chi_{\pi_v'}(f^H_v)
\end{equation*}
for matching test functions $f^G_v$ and $f^H_v$.
The summation in these formulas is over representatives $\pi_v$ of a local $L$-packet of $G_v$ on the left,
and over representatives $\pi'_v$ of the assigned local $L$-packet of $H_v$ 
with respect of the local Langlands correspondence established in \cite{Gan_Takeda, Gan_Tantono}.
The identities above may therefore be considered as identities of stable characters:

\medskip\noindent
Recall, the norm map $n$ identifies the set of stable conjugacy classes of $G_v$ with a
subset of the set of stable conjugacy classes of $H_v$.
The open subset of $H_v$ of regular semisimple elements of $H_v$ whose
stable conjugacy class is a $G_v$-norm $n(t)$ 
will be referred to as the \textit{subset of $G_v$-norms} in $H_v$.

\begin{lemma}\label{Gan-inner}
On the open set of regular elements $n(t)$ in $H_v=GSp(4)_v$ coming from $G_v$ under the stable norm map $n$ of Kottwitz \cite{Kottwitz_2, Kottwitz_3} we have
$$  \sum_{\pi_v} \chi_{\pi_v}(t)   =  \sum_{\pi_v'} \chi_{\pi_v'}(n(t))  \ .$$
Here summation is over $\pi_v$ in a nonarchimedean local $L$-packet 
of $G_v$ on the left and over the corresponding local (generic) $L$-packet
of $H_v$ on the right that is attached to it by the local Langlands correspondence of Gan-Takeda~\cite{Gan_Takeda} and Gan-Tantono~\cite{Gan_Tantono}.
\end{lemma}

We remark that the character identity stated in lemma~\ref{Gan-inner} does not hold in one \textit{exceptional} case explained below in 4, namely for representations $\pi_v$ of type \nosf{VIc} ).
To proceed, we first need to describe the set of $G_v$-norms.

\begin{table}
\begin{footnotesize}
\begin{center}
\begin{tabular}{|c|c|c|c|c|c|c|c}\hline
$H_v$ type &  $G_v$ type               &$Q_v$&$Q_v$&$P_v$&$P_v$& $P_v$ \\\hline\hline
     I                    & --         & *   &     & *   &     &       \\\hline\hline
     IIa                  & IIa$^G$    & *   &     & *   &     &       \\\cline{1-2}
     IIb                  & --         & *   &     &     & *   &       \\\hline\hline
     IIIa                 & --         & *   &     & *   &     &       \\\cline{1-2}
     IIIb                 & --         &     & *   & *   &     &       \\\hline\hline
      IVa                 & IVa$^G$    & *   &     & *   &     &       \\\cline{1-2}
      IVb                 & --         & *   &     &     & *   &       \\\cline{1-2}
      IVc                 & IVc$^G$    &     & *   & *   &     &       \\\cline{1-2}
      IVd                 & --         &     & *   &     & *   &       \\\hline\hline
      Va                  & Va$^G$     & *   &     & *   & *   &       \\
      Va*     & $\chi_0\otimes$ Va$^G$ &     &     &     &     &       \\\cline{1-2}
      Vb                  & Vb$^G$     & *   &     & *   &     & *     \\\cline{1-2}
      Vc       & $\chi_0\otimes$Vb$^G$ & *   &     &     & *   &       \\\cline{1-2}
      Vd                  & --         & *   &     &     &     & *     \\\hline\hline
      VIa                 & --         & *   &     & *   &     &       \\
      VIb                 &            & *   &     &     & *   &       \\\cline{1-2}
      VIc                 & VIc$^G$    &     & *   & *   &     &       \\\cline{1-2}
      VId                 & --         &     & *   &     & *   &       \\\hline\hline
      VII                 & --         & *   &     &     &     &       \\\hline\hline
      VIIIa               & --         & *   &     &     &     &       \\
      VIIIb               & --         & *   &     &     &     &       \\\hline\hline
      IXa                 & --         & *   &     &     &     &       \\\cline{1-2}
      IXb                 & --         & *   &     &     &     &       \\\hline\hline
      X                   & X          &     &     & *   &     &       \\\hline\hline
      XIa                 & XIa        &     &     & *   &     &       \\
      XIa*                & XIa*$^G$   &     &     &     &     &       \\\cline{1-2}
      XIb                 & XIb        &     &     & *   &     &       \\\hline\hline
 $\theta_+(\pi_1,\pi_2)$  & $\theta(\pi_1^{\mathrm{JL}},\pi_2)$  &  &   &   &   & \\
 $\theta_-(\pi_1,\pi_2)$  & $\theta(\pi_2^{\mathrm{JL}},\pi_1)$  &  &   &   &   & \\\hline\hline
 $\theta(\sigma)$         & $\theta(\sigma^{\mathrm{JL}})$       &  &   &   &   & \\\hline
\end{tabular}
\end{center}
\end{footnotesize}
\begin{footnotesize}
\caption{\textit{Local representations}. The first two columns list the irreducible representations of $H_v$ and its non-split inner form $G_v$ at non-archimedean local fields. The roman numerals have been introduced by Roberts and Schmidt~\cite{Roberts-Schmidt}.
In these columns, the horizontal lines or double lines divide the $L$-packets. For the Langlands parameters, see table~A.7 in loc.\ cit.
In the right hand columns, for each block the sum of the indicated $H_v$-representations is parabolically induced from $Q_v$ resp. $P_v$ (semisimplified).
In the bottom lines, $\pi_1$ and $\pi_2$ are pairwise non-isomorphic cuspidal representations of $\Gl(2)$ and for the cuspidal theta-liftings see \cite[\S8.2,\S11.3]{Chan_Gan_GSp4_III}.
In the last line, $\sigma$ is an irreducible cuspidal representation of $\Gl(4)$ and
$\pi^{\mathrm{JL}}$ is its Jacquet Langlands-lift to a representation of the unit group of a rank four division algebra over $\GlobalField_v$. Their theta lifts are cuspidal irreducible representations of $H_v$ resp.~$G_v$.\label{Maintable}}
\end{footnotesize}
\end{table}

\bigskip\noindent
3) \textit{The set of $G_v$-norms}.
For a nonarchimedean place $v$, up to isomorphism there exists only one nonsplit inner form $G_v$ of $H_v=\GSp(4)$.
Recall that the characters of irreducible admissible representations of $H_v$ are linear independent
on the subset $H_v^{\mathrm{reg}}$ of regular semisimple elements of $H_v$ and are there represented by smooth functions.
The Grothendieck group $R(H_v)$ generated by irreducible admissible representations of $H_v$
can thus be embedded into the a space of functions on $H_v^{\mathrm{reg}}$ by assigning to each irreducible admissible 
representation $\pi_v$ its character $\chi_{\pi_v}$.
In section~\ref{s:tori} we explicitly describe the set of maximally tori of $H_v$ whose regular elements are not in the set of $G_v$-norms.
Indeed, these are the maximal tori of $H_v$ that up to conjugation are contained in the Levi subgroup of the Klingen parabolic subgroup $Q_v$ of $H_v$.
We only deal with stable distributions in 2) above.
By \cite{Kazdhan}~thm.~A the functions defined by stable elements of $R(H_v)$ that vanish on the regular elements in the set of $G_v$-norms are contained in the subgroup $R_I(H_v)$.
By definition $R_I(H_v)$ is spanned by the images of
$$i_{H_v,M_v}: R_{L_v} \to R(H_v)\ ,$$ where $i_{H_v,M_v}$
is defined by normalized induction $i_{H_v,M_v}$ from one of the proper Levi subgroups $M_v$ of $H_v$.
Note, by  \cite{Kazdhan} 1.3
the image of $i_{H_v,M_v}: R_{L_v} \to R(H_v)$  only depends on the conjugacy class of $M_v$ in $H_v$.
It suffices to consider a subset of representatives of
proper Levi subgroups $M_v$ up to conjugation in $H_v$.
Thus $M_v$ can be chosen to be the Levi subgroups of the Siegel parabolic $P_v$,
the Klingen parabolic $Q_v$ and the Borel subgroup $B_v$.
Furthermore, the image for the latter is contained in the image of $i_{H_v,M_v}: R(L_v) \to R(H_v)$
for the Levi subgroup $M_v$ of $Q_v$.
So it suffices to consider the Siegel Levi subgroup and the
Klingen Levi-subgroup. By \cite{vanDijk}, see also
 \cite{Kazdhan} lemma 3, all elements in the image of $i_{H_v,M_v}$ 
for the Klingen Levi vanish on the regular elements of the set of $G_v$-norms
(using proposition \ref{prop:tori_forms} of section~\ref{s:tori}) by an analogue of Mackey's lemma. 
But the same argument also shows that this is not the case for any elements in the image of
$i_{H_v,M_v}: R(L_v) \to R(H_v)$ for the Siegel parabolic
unless they already come from the image of the Levi subgroup of the Borel group $B_v$.
This easily follows from the Langlands classification.
Indeed the argument of \cite{Borel_Wallach} XI.2.14. shows that the induced
representations $I_{P,\sigma,\chi}=i_{H_v,M_v}(\sigma,\chi)$ obtained
from the Langlands data $(P,\sigma,\chi)$ for the proper standard parabolic subgroups define a basis for $R_I(H_v)$.
See also \cite[prop.~1]{Kazdhan}. By prop.~\ref{prop:tori_forms} this implies:

\begin{lemma}\label{Klingeninduced}
For a stable finite linear combination $T$ of characters of irreducible admissible representations of $\GSp(4)_v$ the following assertions are equivalent:
\begin{enumerate}
 \item $T$ vanishes on the subset of $G_v$-norms in $H_v$ for the non-split inner form $G_v$,
 \item $T$ is in the image of parabolic induction $i_{H_v,M_v}: R(M_v) \to R(H_v)$ for the Levi subgroup $M_v$ of the Klingen parabolic group $Q_v$.
\end{enumerate}
\end{lemma}

\bigskip\noindent
4) \textit{The exceptional case VI}.
A remark concerning the local character identities described in 2) is in order.
It seems that there is one exceptional local $L$-packet overseen in the proof of \cite{Chan_Gan_GSp4_III} prop.~11.1 (i), where the formula given in lemma~\ref{Gan-inner} does not hold.
This exceptional case belongs to the local  $L$-packet VIc  of type VI, in the notation of
Roberts and Schmidt~\cite{Roberts-Schmidt} and Sally and Tadi\'{c} \cite{Sally-Tadic}.

\begin{lemma}\label{lem:case_VI}
The non-archimedean local $L$-packets of type \nosf{VIc} do not satisfy a character relation in the sense of lemma~\ref{Gan-inner}.
\end{lemma}

\begin{proof}
Indeed, the local $L$-packet $\Pi_\phi(G)=\{\pi_v\}$ of type \nosf{VIc} is the singleton composed of the parabolically induced representation
$\pi_v=\Ind_{P_{G,v}}^{G_v}(\rho)$ from a representation $\rho=\nu_D^{\pm1/2}\boxtimes \sigma$ of the Levi $L_{G,v}\cong D_v^\times\boxtimes\GlobalField_v$ of the standard minimal parabolic subgroup $P_{G,v}$ of $G_v$.
As shown by Chan and Gan \cite[\S11.2]{Chan_Gan_GSp4_III}, there is a character relation between $\pi$ and the Siegel induced representation $\pi'=\Ind^{H_v}_{P_{H,v}}(\rho')$ for the representation $\rho'=\nu^{\pm1/2}\mathrm{St}_{\Gl(2)}\boxtimes \sigma$ of $L_{H,v}(\GlobalField_v)$.
However, $\chi_{\pi'_v}=\chi_{VIa}+\chi_{VIc}$ has two constituents of type \nosf{VIa} and of type \nosf{VIc}.
If the proposed character relation of lemma~\ref{Gan-inner} would hold for type~\nosf{VIc}, then the character $\chi_{VIa}$ would be zero on the regular $G_v$-norms. But then lemma~\ref{Klingeninduced} implies that $\chi_{VIa}$ would be Klingen-induced and it is well-known that this is not true.
\end{proof}
One may expect that the \lq{vanishing\rq} of the character $\chi_{VIa} + \chi_{VIb}$ of the  generic $L$-packet of type VI on the set of semisimple regular $G_v$-norms
 might describe the general rule
that detects those cases where the behaviour of analytic distributional 
character lift is not compatible with the Langlands conjecture.
Concerning the global situation,
type VI may indeed appear for Saito-Kurokawa representations,
where locally VIa${}^G$ may lift to an Arthur packet with character $\chi_{VIc} - \chi_{VIb}$.
Notice that the characters $\chi_{VIc} - \chi_{VIb} $ and $ \chi_{VIa} + \chi_{VIc}$ coincide on all regular elements of $H_v$ that are norms from $G_v$.

\bigskip\noindent 
5) \textit{Main result.}
We are now prepared to formulate and prove the main result of this section. 
A cuspidal representation $\pi_\fin$ of $G(\A_\fin)$ of general type will
be said to be \emph{cohomologically relevant} if $\pi_\fin$ is cohomological such that the coefficient
$c_G(\pi)=m_{G,\mathrm{coh}}(\pi_\fin)/d(G) $ does not vanish.

\begin{thm}\label{mainthm}
Let $G$ be an inner form of $H=\GSp(4)$ that is anisotropic at the archimedean place and let $S_G$ denote the set of nonarchimedean places $v$ where $G_v$ does not split.
Let $\pi$ be an irreducible automorphic representation of $G(\A)$.
Suppose $\pi$ is of general type, i.e neither weakly $H$-Eisenstein nor endoscopic.
Then $\pi$ has a lift to $H$ given by a unique irreducible cuspidal generic cohomological automorphic representation
$\pi'$ of $H(\A)$ that is weakly equivalent to $\pi$,
i.e. $\pi^S \cong (\pi')^S$ holds for a sufficiently large finite set $S$ of places.
Furthermore, the following assertions hold:
\begin{enumerate}
\item Every irreducible automorphic representation of $G(\A)$ in the weak packet $\Pi(\pi)$ 
of all irreducible automorphic representations that are weakly equivalent to $\pi$ has multiplicity one. 
\item $\Pi(\pi)$ is a product of local $L$-packets so that the local $L$-packets
are in 1-1 correspondence with respect to the local Langlands parameters.
See table~\ref{Maintable} for an explicit description of the packets in terms of the Sally-Tadi\'{c} classification.
\item For every global cohomological representations $\pi$ of $G(\A)$
of general type and every $v\in S_G$, the local representation $\pi_v$ cannot belong to case \nosf{VIc}${}^G$ .
\item For an irreducible cohomological automorphic representation $\pi'$ of $H(\A)$
of general type the following assertions are equivalent:
\begin{enumerate}
\item $\pi'$ is the lift of some irreducible automorphic representation $\pi$ of $G$.
\item for every place $v\in S_G$, the local $L$-packet containing $\pi'_v$ is not fully Klingen induced, i.e. the direct sum of its constituents is not equivalent to the semisimplification of a Klingen induced representation.
\end{enumerate}
\end{enumerate}
\end{thm}

\begin{proof}
The representation $\pi_\infty$ defines a unique weight $\lambda$.
Consider the trace formula comparison of section~\ref{Vergleich} for the coefficient systems attached to $\lambda$.
By thm.~\ref{innerlifting} there already exists a weak lift $\pi'$ of $\pi$ to $H(\A)$, 
and without loss of generality we can assume this lift $\pi'$ to be generic by \cite{W_Schiermonnikoog} and prop.~\ref{spectrum}.
Each term of the cohomological trace formula of $H(\A)$
appears in form of packets  $\sum_{\pi_v'} \chi_{\pi_v'}(f^H_v) $ 
by prop.~\ref{spectrum}.
On the other hand  by \cite{Gan_Takeda}, see also lemma~\ref{Gan-inner},  
we know $\sum_{\pi_v} \chi_{\pi_v}(f_v)   =  \sum_{\pi_v'} \chi_{\pi_v'}(f^H_v)  $.
Hence the cohomological trace formula comparison of the 
$\pi^S=(\pi')^S$ isotypic components become
$$  \sum_{\pi_S} \chi_{\pi_S}(f_S) \  = \  \prod_{v\in S} \sum_{\pi_v} \chi_{\pi_v}(f_v) \ $$
for test functions with support in the set of regular semisimple elements.
Hence assertion~1 and 2 follow from the linear independence of characters on the set of regular semisimple elements.
In particular this implies that $\sum_{\pi_v} \chi_{\pi_v}$ does not vanish identically on the set of $G_v$-norms in $H_v$.
By lemma~\ref{lem:case_VI}, this excludes that $\pi'_v$ can be of type~\nosf{VI},
and hence excludes that $\pi_v$ can be of type~\nosf{VI} either.
This implies assertion~3.

\bigskip\noindent
For the last assertion we now write $S:=S_G$, for simplicity of notation. 
We use the trace comparison above on $G(\A_S)$ for $S=S_G$.
Reading it as a trace identity on $H(\A_S)$ we find the following:
Let $S$ still denote the set of nonarchimedean places where $G$ does not
split. Fix $(\pi')^S$ where $\pi'$ is a cuspidal automorphic representation of $H(\A)$
of general type with $\pi'_\infty$ in the $L$-packet attached to $\lambda$.
We then say, $\pi'$ is not a lift from $G$ if it is not in $\Pi(\pi')$ for some lift $\pi'$
of an irreducible representation $\pi$ of $G(\A)$ with $\pi_\infty$ determined by 
the given weight $\lambda$. The the finite sum over all $\pi'$ in $\Pi(\pi')$ that 
are not liftings from $G(\A)$ vanishes on the open set of $G$-norms. In other words,
for all $f$ with regular semisimple support in $G(\A_S)$ and suitable matching
functions $f^H$ we conclude that
$   \sum_{\pi'_S}  \pi'_S(f^H)  =  0$
where the summation is over all $\pi'_S$ such that $\pi'_\infty(\lambda)\pi'_S
(\pi')^S$ is not a lift from $G$ where we fix $(\pi')^S$ by using a good projector
for the places outside $S$. By our knowledge of the cohomological spectrum 
prop.~\ref{spectrum}
we then obtain the stronger result
$$   \prod_{v\in S} \sum_{\pi'_v}  \pi'_v(f_v^H)  =  0  \ . $$
Here $\sum_{\pi'_v}  \pi'_v(f_v^H)$ is the sum over constituents
of a local $L$-packet of $H_v$ that contains a generic member.
The product over $v\in S$ vanishes if and only if one of its
summands  $\sum_{\pi'_v}  \pi'_v(f_v^H)$ vanishes for
all test functions $f_v$ with regular support 
and locally matching $(f_v,f_v^H)$. By lemma~\ref{Klingeninduced}
this implies that the sum of representations $\sum_{\pi'_v}  \pi'_v$ in the local $L$-packet
is in the image of  $$i_{H_v,M_v}: R(M_v) \to R(H_v)$$ 
for the Levi subgroup $M_v$ of the Klingen parabolic group $Q_v$.
\end{proof}

The same arguments that prove theorem~\ref{mainthm} 
also can be used for arbitrary inner forms $G$ of $H=\GSp(4)$.
However, the lifting assertions only carry over for cohomologically
relevant representations $\pi$ of $G(\A)$. The description
of $\Pi(\pi)$ a priori may not carry over. However, under the assumption of
conjecture~\ref{conj:middle_coh} everything carries over verbatim.
Let us call $\lambda$ to be $G$-regular,
if the Lie algebra cohomology with respect to $V_\lambda$
is nontrivial only for irreducible admissible representations of $G_\infty$
in the discrete series of $G_{\infty}$.
For $G$-regular weights $\lambda$ the conjecture~\ref{conj:middle_coh} holds by trivial reason.
Hence the proof  of the last theorem carries over verbatim in these cases and we obtain:

\begin{thm} \label{mainthm-2}
Fix an arbitrary inner form $G$ of $H=\GSp(4)$
and let $S_G$ denote the set of nonarchimedean places where $G$ does not split.
Let $\pi$ be an irreducible automorphic representation of $G(\A)$ of general type such that $\pi_\infty$ belongs to the discrete series of $G_\infty$
attached to a $G$-regular highest weight $\lambda$.
Then $\pi$ has a lift $\pi'$ to $H$. More precisely, there is a unique irreducible 
cuspidal generic cohomological automorphic representation
$\pi'$ of $H(\A)$ that is weakly equivalent to $\pi$,
i.e. so that $\pi_v \cong \pi'_v$ holds for almost every non-archimedean place $v$ with finitely many exceptions.
Furthermore, the following assertions hold:
\begin{enumerate}
\item Every irreducible automorphic representation of $G(\A)$ in the weak packet $\Pi(\pi)$ has multiplicity one. This also holds without the assumption that $\pi$ is of general type.
\item $\Pi(\pi)$ is a product of local L-packets so that the local $L$-packets are
in bijective correspondence with respect to the local Langlands parameters.
\item  For global cohomological representations $\pi$ of $G(\A)$
of general type, the local factors $\pi_v$ for $v\in S_G$ cannot belong to case VI${}^G$ .
\item For an irreducible cohomological automorphic representation $\pi'$ of $H(\A)$
of general type with $\pi'_\infty$ belonging to the $L$-packet attached to the fixed $G$-regular weight $\lambda$
the following assertions are equivalent:
\begin{enumerate}
\item $\pi'$ is a global lift of an irreducible automorphic representation $\pi$
of $G$ with $\pi_\infty$ in the discrete series attached to $\lambda$
\item for every $v\in S_G$, the local $L$-packets attached to $\pi_v$ is not Klingen induced, i.e.~the direct sum of its constituents is not the semisimplification of a Klingen induced representation.
\end{enumerate}
\end{enumerate}  
\end{thm}
See table~\ref{Maintable} on page~\pageref{Maintable} for an explicit description of the $L$-packets in terms of the roman numerals introduced by Roberts and Schmidt~\cite{Roberts-Schmidt}. We remark that in assertion 1 the assumption that $\pi$ is of general type can be dropped since multiplicity one for CAP and endoscopic $L$-packets of $G$ is known by \cite{Chan_Gan_GSp4_III} and \cite{Gan_SK_inner}.
As an immediate consequence of theorem~\ref{mainthm} and theorem~\ref{mainthm-2},
an inspection of table \ref{Maintable} implies the following:  

\begin{cor}[Local types for $G$]\label{cor:local_types_G}
Fix an inner form $G$ of $H=\GSp(4)$ that is not split at a non-archimedean place $v$.
Suppose $\pi$ is a global cohomologically relevant irreducible representation of $G(\A)$ of general type that lifts to some cuspidal representation $\pi'$ of $H(\A)$.
Then the local factors $\pi_v$ belong to type~\nosf{IIa}${}^G$, \nosf{IVa}${}^G$, \nosf{Va}${}^G$,
$\chi_0\otimes$ \nosf{Va}${}^G$, \nosf{X}${}^G$, \nosf{XIa}${}^G$, \nosf{XIa*}${}^{G}$,
or belong to $L$-packets whose generic member is cuspidal.
\end{cor} %

\begin{cor}[Local types for $H$ not in the lift]\label{cor:local_H_not_in_lift}
For a global cohomological irreducible representation $\pi'$ of $H=\GSp(\A)$ of general type and an inner form $G$ of $H$,
the following assertions are equivalent:
\begin{enumerate}
\item $\pi'$ is not a lift of some irreducible automorphic cohomological representation $\pi$ of $G(\A)$ with archimedean components $\pi_\infty$ attached to the same weight $\lambda$ as $\pi'_\infty$.
\item There is some local nonarchimedean place $v$ where $G_v$ is not split, such that the local constituent $\pi'_v$ is of type \nosf{I}, \nosf{IIIa}, \nosf{VIa}, \nosf{VIb}, \nosf{VII}, \nosf{VIIIa} or \nosf{VIIIb}.
\end{enumerate}
\end{cor}

\begin{cor}[Not existing local types for $H$]
A cuspidal cohomological irreducible representation $\pi'$ of $\GSp(4,\A)$ of general type
cannot have any local constituents $\pi'_v$ of type \nosf{IXa} at any nonarchimedean place $v$.

\end{cor}
\begin{proof}
Indeed, assume $\pi'_v$ is of type \nosf{IXa}. Fix the quaternion algebra $D$ that is not split at $v$ and $\infty$ but splits at all other places.
Then $\pi'$ cannot be a inner lift of any representation of $G=\GU_D(2)$, because $G_v$ does not admit local representations with $L$-parameter of type \nosf{IXa} by corollary~\ref{cor:local_types_G}.
By corollary~\ref{cor:local_H_not_in_lift}, $\pi'_v$ cannot be of type \nosf{IXa}.
\end{proof}

\section{Conjectures of Ibukiyama and Kitayama}

Ibukiyama and Kitayama \cite{Ibukiyama17} formulate  conjectures relating
holomorphic paramodular Siegel newforms of squarefree level $N$ for $\GSp(4)$
with certain automorphic forms for $G=\GU_D(2)$ and the quaternion algebra $D$ of discriminant $N$ and ramified at $\infty$, respectively holomorphic modular forms on $G'=\GU_D(1,1)$ if $D$ has discriminant $N$ and splits at $\infty$. We now prove these conjectures.
Let $\omega(N)$ denote the number of prime factors dividing $N$.

\medskip 
In \cite[p.~603]{Ibukiyama17} maximal congruence subgroups $U(N) \subset G(\Q)$ resp. \mbox{$U'(N)\subset G'(\Q)$} and spaces $\mathfrak{M}_{k+j-3,k-3}(U(N))\cong \bigoplus_{\pi} \pi_{\fin}^{U(N)} $ resp. $S_{k,j}(U'(N)) \cong \bigoplus_\pi \pi_{\fin}^{U'(N)}$
were defined. The sum extends over all cuspidal automorphic representations
 $\pi=\pi_\infty\pi_{\fin}$, where the central character is trivial and
where $\pi_\infty$ belongs to the archimedean discrete series of weight $\lambda$, which in loc. cit, p. 598 is explicitly described as $\det^k \mathrm{Sym}^j$ by the two parameters $k,j$.

\begin{prop}[Yoshida type lifting]\label{prop:Ibukiyama_endo}
For $j\geq0$ and $k\geq3$ and squarefree $N$ there is an injective embedding, sending eigenforms to eigenforms,
$$\iota:\bigoplus_{M\mid N\ ,\ \omega(M)\ \text{odd}~,}S^{\new}_{j+2}(M)\times S^{\new}_{2k+j-2}(N/M)\hookrightarrow \begin{cases}\mathfrak{M}_{k+j-3,k-3}(U(N))&\omega(N)\ \text{odd}\ ,\\S_{k,j}(U'(N)) & \omega(N)\ \text{even}\ . \end{cases}$$
The image of these maps define all endoscopic forms for the congruence subgroups $U(N)$ resp. $U'(N)$.
In particular, for eigenforms $f$ and $g$ the spinor $L$-factor of $\iota(f,g)$ is given by $L(s,f)L(s-k+2,g)$ up to local factors at the bad places.
\end{prop}

\begin{proof}For inner forms of $\GSp(4)$, the endoscopic spectrum has been described by Chan and Gan~\cite[\S3]{Chan_Gan_GSp4_III}.
Fix eigenforms $(f_1,f_2)$ in $S^{\new}_{j+2}(M)\times S^{\new}_{2k+j-2}(N/M)$,
then they generate cuspidal automorphic representations $\tau_1$ and $\tau_2$ of $\Gl(2,\mathbb{A})$.
By construction, they are locally tamely ramified at the places dividing $M$ resp.~$N/M$. Since $N$ is squarefree, these places are disjoint.
Let $D_1$, $D_2$ and $D$ be the unique quaternion algebra over $\Q$ with discriminant $M$ resp. $N/M$ resp $N$.
Then $D$ ramifies at $\infty$ if and only if $\omega(N)$ is odd.

\medskip
The Jacquet-Langlands lifting sends $\tau_1\boxtimes\tau_2$ to an automorphic representation $\tau_1^{D_1}\boxtimes \tau_2^{D_2}$ of $D_{1,\A}^\times \times D_{2,\A}^\times$.
By the endoscopic lifting \cite[\S3]{Chan_Gan_GSp4_III} this is attached to (a weak $L$-packet of) cuspidal automorphic representations $\pi=\pi^{\mathrm{endo}}_{f,g}$ of $\GU_D(2)$.
The local factors of $\tau_i$ are well-known, they are unramified principal series at the places not dividing the level and they are Steinberg representations twisted by an unramified character at the places dividing the level.
By the results on local endoscopy of Chan and Gan~\cite{Chan_Gan_GSp4_III}, the local factors $\pi_v$ are of type \nosf{IIa}$^G$ in the sense of Roberts and Schmidt~\cite{Roberts-Schmidt} at the non-archimedean places $v$ dividing $N$.
They are spherical principal series representation of $G_v=\GSp(4,\GlobalField_v)$ of type \nosf{I} at the other places.
As Narita and Schmidt~\cite{Narita} have shown, there is a non-zero $U(N)$-invariant vector in $\pi_\fin$, unique up to scalars.

For odd $\omega(M)$, the archimedean factor $\pi_\infty$ either belongs to the non-quasi-split group $\GU_D(2)(\R)$ if $\omega(N)$ is odd or to the holomorphic discrete series of $G'_\infty\cong \GSp(4,\R)$ if $\omega(M)$ is even, see \cite[\S 2.4]{Chan_Gan_GSp4_III} \cite[\S4]{Weissauer_Endo_GSp4}.
The theory of algebraic modular forms resp. holomorphic Siegel modular forms yields the required eigenforms in $S_{k,j}(U'(N))$ resp.~$\mathfrak{M}_{k+j-3,k-3}(U(N))$.
Notice that the even values of $\omega(M)$ do not contribute. Indeed, if $\omega(N)$ is odd otherwise $\omega(N/M)$ would also be odd, hence $D_2$ and $D$ were both anisotropic at $\infty$. But this excludes that the endoscopic lifting goes to $G$ or $G'$  by \cite[thm.~3.1]{Chan_Gan_GSp4_III}.
On the other hand, if $\omega(N)$ and $\omega(M)$ are both even, then $\omega(N/M)$ would also be even. Then the local lift at $\infty$ would go the non-holomorphic generic discrete series of $\GSp(4)$ \cite{Weissauer_Endo_GSp4}.
In both of these cases, the endoscopic lifting does not contribute to algebraic modular forms on $G$ or to holomorphic Siegel modular forms on $G'$, see \cite[thm.~3.1]{Chan_Gan_GSp4_III}.
\end{proof}
This shows the first assertion of conjecture 4.5 in \cite{Ibukiyama17}.

\begin{prop}[Saito-Kurokawa type lifting]\label{prop:IK_SK_inner}
For squarefree $N$ and every divisor $M$ of $N$, there are injective embeddings for $k\geq3$
$$\iota_M: S^{\new, -(-1)^{\omega(M)}}_{2k-2}(\Gamma_0(N/M)) \hookrightarrow \begin{cases}\mathfrak{M}_{k-3,k-3}(U(N)) & \omega(N)\ \text{odd}\ , \\ S_{k}(U'(N)) & \omega(N)\ \text{even}\ . \end{cases}$$
and the image defines precisely the CAP liftings of level $N$ of weight $\lambda$.
For eigenforms $f\in S_{2k-2}^{\new,-(-1)^{\omega(M)}}(\Gamma_0(N/M))$ the spinor L-function $L(s,\iota_M(f),\mathrm{spin})$ of $\iota_M(f)$ is given by $\zeta(s-k+1)\zeta(s-k+2)L(s,f)$  up to finitely many Euler factors.
\end{prop}

\begin{proof}In the language of automorphic representations, the Saito-Kurokawa liftings to all inner forms $G=\GU_D(2)$ and $G'=\GU_D(1,1)$ have been described by Gan~\cite{Gan_SK_inner}. 
The local factors are explicitly given by prop.~6.5 in loc.~cit.
At the places away from the discriminant $N$ of $D$, the local lifting is the well-known Saito-Kurokawa lifting \nosf{IIb}.
At the places dividing $M$, the local factor $\tau_v$ is an unramified principal series and the corresponding Saito-Kurokawa lifting is a singleton of type \nosf{IIa}$^G$. At the (tamely ramified) places dividing $N/M$, the representation $\tau_v$ is an unramified twist of the Steinberg representation and the attached Arthur packet is either the singleton $\{$\nosf{VIc}$^G\}$ or the pair $\{$\nosf{Va}$^G$,\nosf{Vb}$^G\}$, depending on the sign of the Atkin-Lehner involution.
By results of Narita and Schmidt~\cite{Narita}, types \nosf{IIa}$^G$, \nosf{Vb}$^G$ and \nosf{VIc}$^G$ admit a $U'(N)$-invariant vector, unique up to scalars, but \nosf{Va}$^G$ has no $U(N)$-invariants.
The rest of the argument is analogous to the endoscopic case.
The archimedean factor $\pi_\infty$ is the unique representation $\pi_\lambda^{+-}$ for anisotropic $G_\infty$ and it is the holomorphic discrete series $\pi_\lambda^{\hol}$ if $G'_\infty\cong \GSp(4,\R)$ splits, see~\cite{Gan_SK_inner, Chan_Gan_GSp4_III}.
\end{proof}
This proves assertion (2) and (3) in conj.~4.5 of Ibukiyama and Kitayama~\cite{Ibukiyama17} and thus also conjecture~1.4 of loc. cit.
We now consider conjecture 1.3 in the same paper, where spaces of newforms with respect to $U(N)$ and $U'(N)$ are defined as the orthocomplement of all Yoshida liftings and all Saito-Kurokawa liftings of lower level. To address their definition, by abuse of notation, we will formally consider one-dimensional CAP representations $\pi$ of $G(\A)$  
also to be Saito-Kurokawa liftings of lower level. By weak approximation they only lift to one-dimensional representations of $H(\A)$, hence they are CAP-representations of $G(\A)$ with respect to the Borel group $B$ of $H=GSp(4)$, but they do not admit a weak cuspidal lift to $H(\A)$. For them $\pi_\infty$ is one-dimensional, i.e. $\lambda=0$. Hence they arise in the cuspidal spectrum of $G(\A)$ only for $k=3$ and $j=0$. 
Finally notice the following typo in \cite{Ibukiyama17}: In formula (7) on page 618 in \cite{Ibukiyama17}, the case $M=1$ should be excluded. In other words, Saito-Kurokawa-liftings coming from $S_{2k-2}^{\new}(\Gamma_0(N))$ are newforms with respect to $U(N)$ resp. $U'(N)$. We now prove the conjecture 1.3.

\begin{prop}\label{prop:Ibukiyama17_conj13}
For squarefree $N$ and $k\geq 3$ and $j\geq0$, there is a linear isomorphism, sending eigenforms to eigenforms and preserving $L$-functions up to finitely many factors,
\begin{gather*}
 S^{\new}_{k,j}(K(N))\stackrel{\cong}\longrightarrow \begin{cases}\mathfrak{M}^{\new}_{k+j-3,k-3}(U(N)) & \omega(N)\ \text{odd}\ ,\\ S^{\new}_{k,j}(U'(N)) & \omega(N)\ \text{even}\ .\\\end{cases}
\end{gather*}
\end{prop}
The equality of dimensions in linear isomorphism
of prop.~\ref{prop:Ibukiyama17_conj13}, without the assertion on eigenforms and $L$-functions, was obtained in \cite{Ibukiyama17}.
The distinction of cases occurs because the number of places, where the inner form $G$ does not split, is even. This distinction disappears if prop.~\ref{prop:Ibukiyama17_conj13} is reformulated  in terms of irreducible automorphic representations of $H=\GSp(4)$ and $G, G'$ with trivial central character.
This is based on the results of Roberts and Schmidt~\cite{Roberts-Schmidt}. 

\bigskip\noindent
To reformulate the assertion of prop.~\ref{prop:Ibukiyama17_conj13}, it is most convenient to separate the cases
in terms of weak $L$-packets for $H(\A)$, distinguishing the
case of general type, the endoscopic cases and the Saito-Kurokawa cases (in the sense above) by comparing them separately.
In terms of automorphic representations,
the assertion is equivalent 
to the existence of a bijection between automorphic representations with trivial central character preserving spinor $L$-factors
\begin{gather*}
 \left\{\pi'\mid 
 N_{\mathrm{par}}(\pi')=N\ ,\ \omega_{\pi'}=1\ ,\ \pi'_\infty=\pi_\lambda^{\hol}
 \right\}
 \longleftrightarrow\\
 \left\{\pi\mid \pi_\fin^{U'(N)}\neq0\ ,\ \pi_\infty=\pi^G_\lambda\ ,\ \omega_\pi=1\ ,\ \text{not SK-lift from smaller level, not endoscopic}\right\}\ ,
\end{gather*}
where $N_{\mathrm{par}}(\pi')$ denotes the squarefree paramodular newform level and where $\pi_\lambda^{\hol}$ is the archimedean holomorphic discrete series representation attached to $\lambda$.

\begin{proof}
For endoscopic liftings to inner forms see prop.~\ref{prop:Ibukiyama_endo}.
By the definition used in \cite{Ibukiyama17}, endoscopic liftings do not define newforms for $U(N)$ or $U'(N)$.
Endoscopic liftings to $\GSp(4)$ with holomorphic archimedean factor are never paramodular.
Indeed, endoscopic liftings are locally tempered, so being paramodular implies they are generic at every non-archimedean place. Since $\pi_infty$ is holomorphic, hence non-generic, this would violate the condition that in the endoscopic cases the number of non-generic places must be even.

\bigskip\noindent
Now we discuss the cases where $\pi'$ and $\pi$ are CAP. This is only possible for $j=0$ and given by the theory of Saito-Kurokawa liftings.
For every $M\mid N$, there is an embedding
\begin{equation*}
\iota'_M:S^{\new,-}_{2k-2}(\Gamma_0(N/M)) \hookrightarrow S_{k,0}(K(N/M))\ .
\end{equation*}
The representations corresponding to eigenforms in $S^{\new,+}_{2k-2}(\Gamma_0(N/M))$ also admit Saito-Kurokawa liftings, but these are not paramodular \cite[table~A.12]{Roberts-Schmidt}.
For $M>1$ the Saito-Kurokawa liftings into $S_{k,0}(K(N/M))$ embed into the paramodular oldforms $S^{old}_{k,0}(K(N))$.
Thus the paramodular newforms of Saito-Kurokawa type are given by the image of the above embedding for $M=1$.
Thus the number of Saito-Kurokawa liftings in the space of paramodular newforms is given by $\dim S^{\new,-}_{2k-2}(\Gamma_0(N))$.
The spinor $L$-factor of $\iota'_1(f)$ is given by $\zeta(s-k+1)\zeta(s-k+2)L(s,f)$ for eigenforms $f\in S^{\new,-}(\Gamma_0(N))$ up to finitely many Euler factors.
For the inner forms with level structure $U(N)$ resp. $U'(N)$, 
Saito Kurokawa-type liftings exist and are injective by prop.~\ref{prop:IK_SK_inner}:
$$\iota_M:S^{\new,-(-1)^{\omega(M)}}_{2k-2}(\Gamma_0(N/M)) \hookrightarrow \begin{cases}S_{k,0}(U'(N))&\omega(N)\ \text{even,}\\\mathfrak{M}_{k-3,k-3}(U'(N))&\omega(N)\ \text{odd.}\end{cases}$$
Again, the spinor L-function of $\iota_M(f)$ is given by $\zeta(s-k+1)\zeta(s-k+2)L(s,f)$ for eigenforms $f\in S^{\new,-}(\Gamma_0(N/M))$ up to finitely many Euler factors \cite[page.~618]{Ibukiyama17}.
By definition, only liftings for $M=1$ can be newforms. 
Thus for each eigenform $f\in S^{\new,-}_{2k-2}(\Gamma_0(N))$ the proposed bijection relates $\pi'$ generated by $\iota'_1(f)$ and $\pi$ generated by $\iota_1(f)$.

\medskip
We now consider the representations of general type.
Recall that we have to construct a bijection between representation $\pi'$ and $\pi$ of general type
preserving Langlands parameters as above.
We claim that the inner lifting of thm.~\ref{mainthm} provides this correspondence.

\bigskip\noindent
To make this precise, keep in mind that a priori thm.~\ref{mainthm} only defines a correspondence between $L$-packets. 
However at the archimedean place  the $L$-packet attached to $H_\infty$ with parameter $\lambda$ contains two discrete series representations. By the holomorphicity assumption in loc. cit., only the holomorphic discrete series attached to $\lambda$ is relevant.
Furthermore at the non-archimedean places, since $N$ is squarefree,  only singleton  local $L$-packets can contribute.
We claim that at the non-split nonarchimedean places of $G$ all local factors $\pi'_v$ and $\pi_v$ are necessarily of type \nosf{IIa} resp.~\nosf{IIa}${}^G$. For $H=\GSp(4)$ this
follows from explicit description of the paramodular level structures
of Roberts and Schmidt~\cite[table~A.12]{Roberts-Schmidt}. Indeed, the local nonarchimedean representations of $H_v$ with paramodular level one are of type \nosf{IIa}, \nosf{IVcd}, \nosf{Vbc}, and \nosf{VIc} \cite[A.12]{Roberts-Schmidt}.
Since only \nosf{IIa} is contained in a generic $L$-packet,
 among these  only \nosf{IIa} can occur globally for automorphic representations $\pi'$ of general type by prop.~\ref{spectrum}. 
Since the number of non-generic local constituents $\pi'_v$ for endoscopic representations $\pi'$ of $H(\A)$ is even and $\pi'_\infty$ is holomorphic and thus non-generic and since spherical representations and \nosf{IIa} define singleton generic $L$-packets,
this implies there do not exist nontrivial holomorphic endoscopic forms of squarefree paramodular level.
On the other hand, for the local nonarchimedean inner forms $G_v$ of $\GSp(4)$
Narita and Schmidt~\cite{Narita} show that the paramodular local $L$-packets of level one on $H_v$ uniquely correspond to $L$-packets on the nonarchimedean inner form that have a non-zero invariant vector under a specific locally maximal compact subgroup $K_2$ of $G=\GU_D(1,1)$.
These representations are of type \nosf{IIa}${}^G$, \nosf{IVc}${}^G$, \nosf{Vbc}${}^G$, \nosf{VIc}${}^G$; for the details see \cite[Appendix]{Narita}.
Again, only type \nosf{IIa}$^G$ belongs to a corresponding generic $L$-packet of $H_v$, so the others cannot occur with automorphic representations $\pi$ of general type.
Thus it turns out that at the ramified nonarchimedean places the local factors $\pi'_v$ and $\pi_v$
are all of type \nosf{IIa} resp.~\nosf{IIa}${}^G$ whenever $G_v$ does not split
resp.~are unramified of type $\pi'_v\cong \pi_v$ if $G_v$ splits.

\bigskip\noindent
To apply thm.~\ref{mainthm}, we have to be careful
if $\omega(N)$ is even. Since then $G'_\infty\cong \GSp(4,\R)$ is split,
we are in a situation where conjecture~\ref{conj:middle_coh} is not proven unless $G'=\GSp(4)$.
Therefore, for the cases $j=0$ where the weight $\lambda$ is not regular, 
this requires an additional argument. 
Fix the archimedean weight $\lambda$ and the archimedean discrete series representations
$\pi_\infty^{\hol}, \pi_\infty^W$
that contribute to the cohomology of $V_\lambda$ in degree~3.
Furthermore let $\tilde\pi_\infty^{2,0}$ resp.~$\tilde\pi_\infty^{1,1}$
denote the archimedean representations that contribute to the cohomology of $V_\lambda$ in degree~2 and~4.
It is well-known that the cohomology vanishes in the other degrees~\cite{Borel_Wallach}.
For fixed representations $\pi_{\fin}$ of $G(\A_\fin)$, one has
$m_{\coh}(\pi_\infty^W \pi_{\fin}) = m_{\coh}(\pi_\infty^{\hol}\pi_{\fin})$ \cite{W_Asterisque}.
Further, by the results of Narita and Schmidt~\cite{Narita}
$$\dim S^{\new}_{k,j}(U'(N)) = \sum_{\pi_{\fin}} m_{\cusp}(\pi_\infty^{\hol}\pi_{\fin})$$
where the sum runs over all $\pi_{\fin}$ with $\pi_\fin^{U'(N)}\neq0$ which are newforms for $U'(N)$ as defined above.
Similarly, $S^\new_{k,j}(K(N))  =  \sum_{\pi'_\fin} m_{\cusp}(\pi_\infty^{\hol}\pi'_{\fin})$
 for all $\pi'_\fin$ with paramodular level $N$.
To get the final conclusion we now use that by $m_{\coh}((\pi')^W_\infty \pi_{\fin})=m_{\coh}((\pi')_\infty^{\hol}\pi_{\fin})$
we obtain the trace identity
$$\sum_{
\pi_{\fin}}(2 m_{\coh}(\pi_\infty^{\hol}\pi_{\fin}) 
- m_{\coh}(\tilde\pi^{(2,0)}_\infty\pi_{\fin}) 
- m_{\coh}(\tilde\pi^{(1,1)}_\infty\pi_{\fin}))
= \sum_{\pi'_{\fin}} 2 m_\coh((\pi')_\infty^{\hol}\pi'_{\fin})\ .$$
Summation on the right is over all $\pi'_{\fin}$ for which $(\pi')_{\fin}^{K(N)}\neq 0$.
Summation on the left is over all  
$\pi_{\fin}$ for which $\pi_{\fin}^{U'(N)}\neq 0$ that either match with $\pi'_{\fin}$ in the sense of thm.~\ref{mainthm} or that satisfy $2 m_{\coh}(\pi_\infty^{\hol}\pi_{\fin}) 
- m_{\coh}(\tilde\pi^{(2,0)}_\infty\pi_{\fin}) -  
m_{\coh}(\tilde\pi^{(1,1)}_\infty\pi_{\fin}))=0$.
The dimension formula \cite{Ibukiyama17},~thm.1.2~on the other hand gives
$$\sum_{\pi_{\fin}} m_{\cusp}(\pi_\infty^{\hol}\pi_{\fin}) = \sum_{\pi_{\fin}'} m_{\cusp}((\pi')_\infty^{\hol}\pi'_{\fin}) \ .$$ for all $\pi_\fin$ resp.~$\pi'_\fin$ of level $N$.
This implies 
$\sum_{\pi_{\fin}} m_{\coh}(\tilde\pi^{(2,0)}_\infty\pi_{\fin}) + m_{\coh}(\tilde\pi^{(1,1)}_\infty\pi_{\fin}) =0$
and thus proves conjecture \ref{conj:middle_coh} for the relevant irreducible representations
$\pi_{\fin}$ of $G'(\A_{\fin})$ that satisfy $\pi_{\fin}^{U'(N)}\neq 0$.

\bigskip\noindent
Finally, it remains to be seen that every paramodular $\pi'$ corresponds to some $\pi$ with the same $L$-function under the inner lift. Indeed, otherwise corollary~\ref{cor:local_H_not_in_lift} would imply that there is a place $v$ dividing $N$ where the local factor $\pi'_v$ is
of type \nosf{I}, \nosf{IIIa}, \nosf{VIab}, \nosf{VIIIab} or \nosf{VII}.
However, for these types the paramodular level cannot be one \cite[A.12]{Roberts-Schmidt},
so $\pi'$ does not contribute to paramodular newforms of level $N$.
\end{proof}

\bigskip\noindent
\textbf{Another conjecture of Ibukiyama.}
We finally draw attention to conjecture 1.1 in \cite{Ibukiyama08}, which relates
certain holomorphic Siegel modular forms of genus 2 of integral respectively half-integral weight.
This should be considered as a conjectural generalization of the well known Shimura lifting in the theory of elliptic modular forms.
As explained in \cite{Ibukiyama08}, such a generalized Shimura lifting would allow
to transfer Harder's conjecture~\cite{Harder} to the more accessible conjecture~1.2
of \cite{Ibukiyama08} on modular forms of half-integral weights. As already observed by Ihara, 
a lifting from holomorphic Siegel modular forms of weight $\det^k\, \mathrm{Sym}^j$ to algebraic modular forms on the inner form that is isogenous to a quinary definite orthogonal group, followed by a theta lifting to the genus 2 metaplectic group, will provide holomorphic Siegel modular forms of weight $\det^{(j+5)/2}\, \mathrm{Sym}^{k-3}$; see \cite[p.\,112]{Ibukiyama08}.
Therefore thm. \ref{mainthm} provides a lifting from holomorphic Siegel cusp forms of weight $\det^k\, \mathrm{Sym}^j$
of general type to holomorphic Siegel modular forms of weight $\det^{(j+5)/2}\, \mathrm{Sym}^{k-3}$.
Of course, there are two restrictions:
First, this lifting $\theta(\pi)$ is only defined for automorphic representations $\pi'$ that  are the image of some $\pi$ under the inner lifting.
Second, the image of the theta lifting $\theta(\pi)$ might a priori vanish. If the latter is not the case, the
image $\theta(\pi)$ is necessarily cuspidal if $\pi$ is of general type. Due to these restrictions this so far only gives a partially defined generalized Shimura correspondence in the genus 2 case.
Although this partially defined generalized Shimura correspondence will not prove
conjecture 1.1 of \cite{Ibukiyama08} since holomorphic Siegel cusp forms for the full Siegel modular group are not in the image of the inner lifting, it is interesting to ask whether this allows to attack the proof of generalizations of the Harder's congruence conjecture, now formulated for $L$-values of elliptic modular forms of  level $N>1$. 

\section{Appendix on maximal tori and conjugacy classes}\label{s:tori}

\noindent
We recall the classification of maximal tori of $\GSp(4)$ and its inner forms over local fields, compare \cite[\S4.4.2]{Weissauer_Endo_GSp4} and \cite{Hina_Masumoto_tori}.

\bigskip\noindent
\textit{The intrinsic involution}. Let $G=\GSp(2n)$ be the symplectic group of similitudes considered over
a field $k$ of characteristic zero.
The defining equation of $G$ in $\Gl(2n)$ is $g' J g = \lambda(g) J$, where $\lambda$
denotes the similitude character on $G$, and we may assume
$J=\left(\begin{smallmatrix}0&I_n\\-I_n&0\end{smallmatrix}\right)$
to be Siegel's matrix.
The \lq{intrinsic}\rq\ anti-involution $g\mapsto \widehat g= \lambda(g) g^{-1}$ on $G$ hence 
coincides with the anti-involution $g\mapsto g^*=J^{-1} g' J$, and the latter obviously
extends to an anti-involution of the matrix algebra $A=M_{2n,2n}(k)$.
Conjugation with elements $g\in G$ preserves the involution: $(g^{-1}tg)^* =
g^* t^* (g^{-1})^* = g^{-1} t^* g$ for all $t\in A$; indeed $g^*=\lambda(g) g^{-1}$
and $\lambda$ commutes with $t^*$. So conjugation with $g\in G$ defines an automorphism of $(A,*)$. 

\bigskip\noindent 
\textit{The algebra $(E,*)$}.
Suppose $t$ is a regular element of $G(k)$.
The centralizer $T$ of $t$ in $G$ is a maximal torus of $G$. 
Let $E$ be the centralizer of $t$ in $A$ and 
$H$ be the intersection of $G$ with the centralizer $C_A(E)$ of $E$
in $A$. Since $T$ is stable under the intrinsic 
anti-involution $g\mapsto \widehat g$, the algebra $E$
is invariant under the anti-involution $*$ of $A$.
Let $E^+=\{ x\in E\ \vert \ x^*=x\}$ be its fixed algebra in $E$. 
Let $C_A(E^+)$ its centralizer in $A$ and $H=H_{E^+}$ be the intersection of $C_A(E^+)$ with $G$.
We claim that $E$ is a a commutative $k$-algebra,
the involution $*$ defines an $k$-algebra automorphism of $E$ of order two with fixed
algebra $E^+$ of dimension $[E^+:k]=n$ such that
$T=\{ x\in E^*\ \vert\ N_{E/E^+}(x) \in k^* \}$ is a maximal
torus in the $k$-algebraic subgroup $H\subseteq G$
that is isomorphic to a $k$-form of $\Gl(2,E^+)^0$,
the subgroups of $\Res^{E^+}_k(\Gl(2))$ of elements
whose determinant is in the ground field $k$.
Moreover $E=\prod_{\nu=1}^r E_\nu$ is a product of field extensions 
$E_\nu/k$ of degree $[E_\nu:k]=n_\nu$ such that 
$\sum_{\nu=1}^r n_\nu = 2n$. The $k$-algebra with involution $(E,*)$
is an invariant of the maximal torus $T$. The involution $*$ collects
the factors $E_\nu$ into $*$-orbits of length one or two.   
To see this, one passes to the algebraic closure of $k$.

\bigskip\noindent
\textit{Algebraically closed fields}.
Suppose now $k=\overline{k}$ is algebraically closed.
Then $T$ is conjugate by some $g\in G(k)$ to the torus of diagonal matrices in $G$.
For all roots $\alpha$ of $G$ with respect to the diagonal torus $T$ in $G$ we have $\alpha(t)\neq 1$.
Each non-diagonal coordinate entry of a $2n\times 2n$ matrix defines an eigenspace
of $T$ by the conjugation action of $T$ on the matrix algebra $A=M_{2n,2n}(k)$
whose eigenvalue is defined by a root of $G$.
Hence the centralizer of $t$ in $A$ is the commutative algebra $E \cong k^{2n}$.
Hence the intrinsic involution defines an automorphism of this algebra of order 2. It is nontrivial
on $E$, since it is nontrivial on $T$.
Hence the fixed algebra $E^+$ has dimension $n$ and is given by all
matrices $\diag(x_1,x_2,....,x_n, x_1,x_2,...,x_n)$ with $x_1,...,x_n\in k$.
Its centralizer in $A$ is isomorphic to $\prod_{i=1}^n M_{22}(k)$ and it is
easy to see that $H \subseteq G$ is isomorphic to the subgroup
of elements $g=(g_1,...,g_n)\in \prod_{i=1}^n \Gl(2,k)$ such that $\det(g_1)=\cdots
= \det(g_n)$ holds. Notice that the normalizer $N_G(T)$  of $T$ in $G$  
normalizes the subgroup $H$ of $G$ and $N_G(T)\cap H$ is a normal subgroup
of $N(T)$ with quotient group the symmetric group $S_n$, acting by permutations
of the components of $g=(g_1,...,g_n)\in H_G(k)$ of the factors.
This action is encoded into the natural action of $N_G(T)$ and $W$ on the algebra $E$.   
Although  $g$ is only unique up to an element in $N_G(T)$,
hence the group $H$ is independent of the choice of $g$ 
up to an $k$-isorphism of $H$, as well as the center $E^+$ of $H$
and the torus~$T$.

\begin{lemma}\label{8.2}
For every field $k$ of characteristic zero, $H^1(k,H)=1$ holds for the groups $H=\Gl(2,k^2)^0$ and their $K$-forms.
\end{lemma}

\noindent
\begin{proof}
If $H$ is quasi-split, there is an exact sequence
$$ 0\to H \to \prod_\nu \Res_k^{L_\nu} \Gl(2) \to  S \to 1$$
where the right hand map is the composite of the determinant
with the quotient map $ \Res_k^{L_\nu} \Gl(1) \to  S$
whose kernel is $\Gl(1)$ diagonally embedded into
$\Res_k^{L_\nu} \Gl(1)$. This defines $S$.
By Hilbert thm.~90 it suffices to see that
$\Res_k^{L_\nu} \Gl(2)(k) \to  S(k)$ is surjective on $k$-rational points.
The determinant is surjective, so it suffices to see that
$\Res_k^{L_\nu} \Gl(1)(k) \to  S(k)$
is surjective. The latter follows again from Hilbert thm.~90.
The argument for inner forms is essentially the same.
\end{proof}
Fix an embedding $i: T \hookrightarrow G$, such that the image is contained in $H$.

\begin{lemma} \label{list}
For the group $\GSp(4)$ over a field k with characteristic zero we have either $E^+=k^2$ (split case), or 
$E^+$ is a quadratic extension field of $k$ (anisotropic case).
So there are the following cases:
\begin{enumerate}
\item $E^+$ is anisotropic and $E$ is a field, either non-Galois
over $k$, or Galois with cyclic Galois group $\mathbb Z/4\mathbb Z$.
\item $E^+=k^2$ splits and $(E,*)=(L_1,*)\oplus (L_2,*)$ for quadratic fields $L_1\not\cong L_2$ over $k$. 
\item $E^+=k^2$ splits and $(E,*)=(L,*) \oplus (k^2,*)$. In this case
$T$ is contained in the Levi subgroup of the Klingen parabolic subgroup $Q$ of $G$.
\item Elliptic case I: $E^+=k^2$ splits and $(E,*)=(L,*)\oplus (L,*)$ for some quadratic
extension field $L$ of $k$ (a degeneration of case of 2).
\item Elliptic case II: $E^+$ is anisotropic and $E/k$ is Galois with Galois group
$\mathbb Z/2\mathbb Z\oplus \mathbb Z/2\mathbb Z$ and fields $L,L',E^+$
in between $E$ and $k$. 
\item $E^+$ is anisotropic, but $(E,*) = E^+\oplus E^+$ splits. Then $*$ must permute the two copies of $E^+$. In this case
$T$ is isomorphic to $\Res^{E^+}_k(\mathbb G_m) \times \mathbb G_m$
and $T$ can be embedded into the Siegel parabolic subgroup $P$ of $G$. 
\item $E^+$ splits and $E=k^4$ and $T$ can be embedded into the Borel subgroup.
\end{enumerate} 
\end{lemma}

\bigskip\noindent
For more details see \cite[p.94--95]{Weissauer_Endo_GSp4}.
Suppose $H=\Res^{E^+}_k(\Gl(2))^0$ can be embedded into $G$ as a $k$-subgroup.
Then every maximal $k$-torus of $H$ defines a maximal $k$-torus of $G$
since both groups have the same rank over $\overline k$.
Fix such a $k$-subgroup $H$ of $G$ for each isomorphism class of quadratic $k$-algebra $E^+/k$,
whenever such a subgroup exists.
\newcommand{\Gal}{\mathrm{Gal}}
\begin{lemma}
Every maximal $k$-torus in $G=\GSp(4)$ is isomorphic to the torus given by
$$T(k)=\{x\in E \mid \mathrm{Norm}_{E/E^+}(x)\in k^\times\}$$
with some $(E,\ast)$ and $E^+$ listed in lemma~\ref{list}.
\end{lemma}
\begin{proof}
We use that for $E^+=k^2$ such an embedding exists for the standard split torus $T$ with $H\cong (\Gl(2,k)\times\Gl(2,k))^0$.
Any two subgroups $H=H(T)$ and $H'=H(T')$ in $G$ attached to
maximal $k$-tori $T$ resp.~$T'$ in $G$ are conjugate over the algebraic closure $\overline{k}$ since the two tori are conjugate over $\overline{k}$ in the sense that
$H' = gHg^{-1}$ for some $g\in G(\overline{k})$.
Then for every $\sigma \in \Gal(\overline{k}/k)$ the element $g^{-1}\sigma(g)$ normalizes $H$. Its cohomology class in $H^1(k,N_G(H))
=\ker(H^1(k,N_G(H))\to H^1(k,G))$
measures whether $H'$ is $G(k)$-conjugate to $H$. If $H$ belongs to the split case $E^+=k^2$, we see that 
$N_G(H)$ is the semidirect product of $H$ and a cyclic group of order two generated by an element in the Weyl group. Hence $H^1(k,H) \to H^1(k,N_G(H))\to H^1(k,\mu_2)$
is surjective and the fibers are trivial by~lemma~\ref{8.2}.
In fact it is obvious that this class in $H^1(k,\mu_2)= \Hom(\Gal(\overline{k}/k),\mu_2)$ is uniquely determined
by the class of the extension $K'/k$ associated to the algebra $K'$
defined by $K'=(E')^+$. Hence any subgroup $H=H_{T'}$ is $G(k)$-conjugate
to our fixed chosen group $H(E_+)$.
Since $T'\subseteq H_{T'}$ by construction, our claim follows.
Since $H^1(k,G)=H^1(T,G)=1$, therefore the $G(k)$-conjugacy classes
of maximal tori are represented by the disjoint union of the $H(E_+)(k)$-conjugacy
classes of maximal $k$-tori in our fixed groups $H(E_+)$.
Furthermore for every isomorphism class of the $k$-algebra $E^+$ there exists an $k$-embedding of $\Res^{E^+}_k(\Gl(2))^0$ into $G$ and such an embedding is unique up to conjugation in $G(k)$.
\end{proof}

\bigskip\noindent
\begin{prop}\label{prop:tori_forms}
Fix an anisotropic quaternion algebra $D_v$ over a non-archimedean $k_v$. 
\begin{enumerate}
\item A maximal torus comes from the non-quasi-split group $G=\GSp(1,1)=\GU_D(1,1)$ over $k_v$ if and only if it is of case 1, 2, 4, 5 or 6. That means a $\GSp(4)$-torus comes from $G_v$ if and only if it does not embed into the Klingen parabolic of $\GSp(4)$.
\item A maximal torus comes from the elliptic endoscopic group $M_v=(\Gl(2)^2)^0(k_v)$ if and only if it is of case 4, 5, 6 or 7.
The elliptic tori coming from $M$ belong to case 4 and 5.
\end{enumerate}
\end{prop}
\begin{proof}
 It suffices to see that for split $E^+=k\oplus k$ the group $H$ is realized as a non-quasi-split
form of $\Gl(2,E^+)^0$ and for all fields $E^+=K$ it is realized with $H=\Gl(2,K)^0$.
 For the second assertion see \cite{Weissauer_Endo_GSp4}, section~4.4.2.
\end{proof}

For the action of the transposition $g\mapsto g'$ on conjugacy classes in $\GSp(4)$ see \cite{Weissauer_Endo_GSp4}, lemma~7.11 and its corollary~7.2.

\begin{footnotesize}
\bibliographystyle{amsalpha}

\vskip 20 pt
\centering{Mirko R\"osner\\ Mathematisches Institut, Universit\"at Heidelberg\\ Im Neuenheimer Feld 205, 69120 Heidelberg\\ email: mroesner@mathi.uni-heidelberg.de}

\vskip 10 pt
\centering{Rainer Weissauer\\ Mathematisches Institut, Universit\"at Heidelberg\\ Im Neuenheimer Feld 205, 69120 Heidelberg\\ email: weissauer@mathi.uni-heidelberg.de}

\end{footnotesize}
\end{document}